\documentclass[a4paper,10pt]{amsart}
\usepackage[a4paper,tmargin=40mm,bmargin=35mm,hmargin=30mm]{geometry}
\usepackage[T1]{fontenc}
\usepackage{lmodern}
\usepackage{amsmath}
\usepackage{amssymb}
\usepackage{amsthm}
\usepackage{stmaryrd}
\usepackage{tikz}
\usetikzlibrary{cd}
\usepackage{booktabs}
\usepackage{multirow}
\usepackage{multicol}
\usepackage{here}
\usepackage{aliascnt}
\usepackage{hyperref}
\usepackage[noabbrev]{cleveref}

\newcommand{\NewTheorem}[3]{
	\newaliascnt{#1}{TheoremEnvironment}
	\newtheorem{#1}[#1]{#2}
	\aliascntresetthe{#1}
	\crefname{#1}{#2}{#3}
	\Crefname{#1}{#2}{#3}
}

\theoremstyle{definition}

\NewTheorem{definition}{Definition}{Definitions}
\NewTheorem{remark}{Remark}{Remarks}
\NewTheorem{axiom}{Axiom}{Axioms}
\NewTheorem{example}{Example}{Examples}
\NewTheorem{observation}{Observation}{Observations}
\NewTheorem{convention}{Convention}{Conventions}
\NewTheorem{notation}{Notation}{Notations}
\NewTheorem{setting}{Setting}{Settings}
\NewTheorem{question}{Question}{Questions}
\NewTheorem{answer}{Answer}{Answers}
\NewTheorem{conjecture}{Conjecture}{Conjectures}
\NewTheorem{problem}{Problem}{Problems}
\NewTheorem{solution}{Solution}{Solutions}
\NewTheorem{goal}{Goal}{Goals}
\NewTheorem{comment}{Comment}{Comments}
\NewTheorem{aim}{Aim}{Aims}
\NewTheorem{caution}{Caution}{Cautions}
\NewTheorem{exercise}{Exercise}{Exercises}

\theoremstyle{plain}
\NewTheorem{proposition}{Proposition}{Propositions}
\NewTheorem{lemma}{Lemma}{Lemmas}
\NewTheorem{theorem}{Theorem}{Theorems}
\NewTheorem{corollary}{Corollary}{Corollaries}

\crefname{appendix}{Appendix}{Appendices}
\Crefname{appendix}{Appendix}{Appendices}

\crefname{enumi}{}{}
\Crefname{enumi}{}{}
\creflabelformat{enumi}{(#2#1#3)}
\crefrangeformat{enumi}{(#3#1#4)--(#5#2#6)}
\crefname{enumii}{}{}
\Crefname{enumii}{}{}
\creflabelformat{enumii}{(#2#1#3)}
\crefrangeformat{enumii}{(#3#1#4)--(#5#2#6)}
\crefname{enumiii}{}{}
\Crefname{enumiii}{}{}
\creflabelformat{enumiii}{(#2#1#3)}
\crefrangeformat{enumiii}{(#3#1#4)--(#5#2#6)}

\makeatletter
\renewcommand{\p@enumii}{}
\renewcommand{\p@enumiii}{}
\makeatother

\numberwithin{equation}{section}
\crefname{equation}{}{}
\Crefname{equation}{}{}
\creflabelformat{equation}{(#2#1#3)}

\newcommand{\SwapSymbols}[1]{
	\expandafter\let\expandafter\temporarysymbol\csname #1\endcsname
	\expandafter\let\csname #1\expandafter\endcsname\csname var#1\endcsname
	\expandafter\let\csname var#1\endcsname\temporarysymbol
}

\SwapSymbols{epsilon}
\SwapSymbols{phi}
\SwapSymbols{Gamma}
\SwapSymbols{Delta}
\SwapSymbols{Theta}
\SwapSymbols{Lambda}
\SwapSymbols{Xi}
\SwapSymbols{Pi}
\SwapSymbols{Sigma}
\SwapSymbols{Upsilon}
\SwapSymbols{Phi}
\SwapSymbols{Psi}
\SwapSymbols{Omega}

\newcommand{\cA}{\mathcal{A}}

\newcommand{\cX}{\mathcal{X}}
\newcommand{\cY}{\mathcal{Y}}

\newcommand{\ka}{\mathfrak{a}}
\newcommand{\kb}{\mathfrak{b}}

\newcommand{\km}{\mathfrak{m}}
\newcommand{\kn}{\mathfrak{n}}

\newcommand{\p}{\mathfrak{p}}
\newcommand{\q}{\mathfrak{q}}

\let\originalleft\left
\let\originalright\right
\renewcommand{\left}{\mathopen{}\mathclose\bgroup\originalleft}
\renewcommand{\right}{\aftergroup\egroup\originalright}

\newcommand{\into}{\hookrightarrow}

\newcommand{\onto}{\twoheadrightarrow}

\newcommand{\isoto}{\xrightarrow{\smash{\raisebox{-0.25em}{$\sim$}}}}

\newcommand{\set}[2][]{\mathopen{#1\{}#2\mathclose{#1\}}}

\newcommand{\setwithtext}[2][]{\mathopen{#1\{}\,\textnormal{#2}\,\mathclose{#1\}}}
\newcommand{\setwithcondition}[3][]{\mathopen{#1\{}\,#2\mathrel{#1|}#3\,\mathclose{#1\}}}

\newcommand{\op}{\mathrm{op}}
\newcommand{\compl}{\mathrm{c}}
\newcommand{\transpose}{\mathrm{t}}

\renewcommand{\subset}{\subseteq}

\newcommand{\resp}{resp.\ }

\newcommand{\id}{\mathrm{id}}

\DeclareMathOperator{\Hom}{Hom}
\DeclareMathOperator{\End}{End}

\DeclareMathOperator{\Ext}{Ext}

\DeclareMathOperator{\Ker}{Ker}
\DeclareMathOperator{\Cok}{Cok}
\let\Im\relax
\DeclareMathOperator{\Im}{Im}

\newcommand{\Ab}{\mathrm{Ab}}

\DeclareMathOperator{\Mod}{Mod}
\let\mod\relax
\DeclareMathOperator{\mod}{mod}

\DeclareMathOperator{\Flat}{Flat}
\DeclareMathOperator{\Cot}{Cot}

\DeclareMathOperator{\artin}{artin}

\DeclareMathOperator{\Ann}{Ann}
\DeclareMathOperator{\rad}{rad}
\DeclareMathOperator{\soc}{soc}

\newcommand{\Zg}{\mathrm{Zg}}
\newcommand{\inj}{\mathrm{inj}}
\newcommand{\flcot}{\mathrm{flcot}}

\DeclareMathOperator{\Spec}{Spec}
\DeclareMathOperator{\Max}{Max}

\DeclareMathOperator{\Supp}{Supp}

\DeclareMathOperator{\fp}{fp}

\newcommand{\Do}{D}
\newcommand{\Dc}{D}

\title{Flat cotorsion modules over Noether algebras}

\author{Ryo Kanda}
\address[Ryo Kanda]{Department of Mathematics, Graduate School of Science, Osaka City University, 3-3-138, Sugimoto, Sumiyoshi, Osaka, 558-8585, Japan}
\email{ryo.kanda.math@gmail.com}

\author{Tsutomu Nakamura}
\address[Tsutomu Nakamura]{Graduate School of Mathematical Sciences, University of Tokyo, 3-8-1 Komaba Meguro-ku Tokyo 153-8914, Japan}
\email{ntsutomu@ms.u-tokyo.ac.jp}

\begin{document}

\begin{abstract}
	For a module-finite algebra over a commutative noetherian ring, we give a complete description of flat cotorsion modules in terms of prime ideals of the algebra, as a generalization of Enochs' result for a commutative noetherian ring. As a consequence, we show that pointwise Matlis duality gives a bijective correspondence between the isoclasses of indecomposable injective left modules and the isoclasses of indecomposable flat cotorsion right modules. This correspondence is an explicit realization of Herzog's homeomorphism induced from elementary duality of Ziegler spectra.
\end{abstract}

\maketitle
\tableofcontents

\section{Introduction}
\label{sec.Intro}

A right module $M$ over a ring $A$ is called \emph{cotorsion} if $\Ext^1_A(F, M)=0$ for every flat right $A$-module $F$.
This class of modules was originally studied in the context of abelian groups (see \cite[\S54]{MR0255673}), and Enochs \cite{MR754698} extended it to the current definition, in relation to the precedent work \cite{MR636889} containing the question whether flat covers exist for an arbitrary ring. This question, later called the \emph{flat cover conjecture}, was affirmatively solved by Bican, El Bashir, and Enochs \cite{MR1832549}, showing that the class of flat modules and the class of cotorsion modules form a complete cotorsion pair, i.e., given any module $M$, there exists a surjection from a flat module to $M$ with cotorsion kernel and an injection from $M$ into a cotorsion module with flat cokernel.
This cotorsion pair is called the \emph{flat cotorsion pair}.

Like torsion pairs, cotorsion pairs are a general notion in abelian categories, which initially appeared in \cite{MR565595}. A \emph{cotorsion pair} consists of two classes of objects in an abelian category such that they are the orthogonal subcategory of each other with respect to the first extension functor $\Ext^{1}(-,-)$.
This notion is closely related to abelian model structures (\cite{MR1938704}, \cite{MR2355778}), and plays an important role in homological algebra and representation theory (e.g., \cite{MR1044344}, \cite{MR1097029}, \cite{MR2009441}, \cite{MR2584945}, \cite{MR4013804}, \cite{MR4091895}, \cite{MR4076700}), 
extending its scope to exact and triangulated categories (e.g., \cite{MR2385669}, \cite{MR2811572}, \cite{MR2861070}, \cite{MR3931945}, \cite{MR3928292}, \cite{arXiv:2007.06536}).

Given a cotorsion pair, it is often important to consider the intersection of the two classes, called the \emph{core} of the cotorsion pair. 
For the flat cotorsion pair, its core consists of all flat cotorsion modules, and they have nice homological properties close to projective modules and injective modules.
To explain such aspects, let us pay our attention to complexes of modules.

Gillespie \cite{MR2052954} showed that the flat cotorsion pair induces two complete cotorsion pairs in the category of complexes, and this fact, along with the work of Bazzoni, Cort\'{e}s-Izurdiaga, and Estrada \cite{MR4140057},
enables us to show that the (unbounded) derived category of modules is equivalent to the homotopy category of $K$-flat complexes of flat cotorsion modules; see \cite[Appendix~A]{MR4127282}. In fact, this remarkable equivalence can be regarded as a restriction of a bigger equivalence, by identifying the derived category with the homotopy category of $K$-projective complexes of projective modules. Indeed, Neeman \cite{MR2439608} proved that the homotopy category of projective modules is equivalent to the pure derived category of flat modules (in the sense of Murfet and Salarian \cite{MR2737778}), which turns out to be also equivalent to the homotopy category of flat cotorsion modules as shown by \v{S}\v{t}ov{\'\i}\v{c}ek \cite[Corollary~5.8]{arXiv:1412.1615}; see also \cite[Remark~A.9]{MR4127282}. 

If the ring is left coherent and all flat right modules have finite projective dimension, then the equivalence between the  homotopy category of projective modules and that of flat cotorsion modules also induces an equivalence between their full subcategories consisting of totally acyclic complexes. Furthermore, the homotopy category of totally acyclic complexes of projective modules is equivalent to the stable category of Gorenstein-projective modules (\cite{Buc86}; see also \cite[Proposition~7.2]{MR2157133}), and the homotopy category of totally acyclic complexes of flat cotorsion modules is equivalent to the stable category of Gorenstein-flat cotorsion modules (studied in \cite{MR3623182}); see \cite{MR4132086} for more details.

These facts motivate us to determine the structure of flat cotorsion modules. The aim of this paper is to give a noncommutative generalization of Enochs' structure theorem \cite{MR754698} for flat cotorsion modules over a commutative noetherian ring $R$. Enochs showed that
an $R$-module $M$ is flat cotorsion if and only if $M$ is isomorphic to
\begin{equation*}
\prod_{\p\in \Spec R} T_\p,
\end{equation*}
where each $T_\p$ is the $\p$-adic completion of some free $R_\p$-module, that is,
\begin{equation*}
T_\p=(R_\p^{(B_\p)})^\wedge_\p:=\varprojlim_{n\geq 1}(R_\p^{(B_\p)}\otimes_R R/\p^n)
\end{equation*}
for a basis set $B_\p$. The cardinality of $B_\p$ for each $\p\in \Spec R$ is determined by $M$.
Enochs reached this formulation by using Matlis' result \cite{MR99360} on the structure of injective $R$-modules and an isomorphism 
\begin{equation}\label{intro.EnochsGriffith}
T_\p\cong \Hom_R(E_R(R/\p), E_R(R/\p)^{(B_\p)}),
\end{equation}
where $E_R(R/\p)$ denotes the injective envelope of $R/\p$.
We generalize Enochs' structure theorem to Noether algebras, which are a simultaneous generalization of commutative noetherian rings and finite-dimensional algebras over a field. Noether algebras have been studied from various aspects (e.g., \cite{MR617088,MR630621}, \cite{MR1898632}, \cite{MR2427009}, \cite{arXiv:1912.07117}, \cite{arXiv:2010.05676}, \cite{arXiv:2006.01677}, \cite{arXiv:2106.00469}).

Let $R$ be a commutative noetherian ring. A \emph{Noether $R$-algebra} is a ring $A$ together with a ring homomorphism $\phi\colon R\to A$ such that the image of $\phi$ is contained in the center of $A$ and $A$ is finitely generated as an $R$-module.
Denote by $\Spec A$ the set of prime (two-sided) ideals of $A$.
The structure homomorphism $R\to A$ induces a canonical map $\Spec A\to \Spec R$ given by $P\mapsto \phi^{-1}(P)$. For brevity, we write $P\cap R:=\phi^{-1}(P)$.

It is known that Matlis' result on injective $R$-modules is generalized to a Noether algebra $A$; there is a one-to-one correspondence
\begin{equation}\label{intro.ClassifOfIndecInjOverNoethAlg}
	\Spec A\isoto\setwithtext{isoclasses of indecomposable injective right $A$-modules}
\end{equation}
in which each $P\in\Spec A$ corresponds to $I_{A}(P)$, the unique indecomposable direct summand of the injective envelope of $A/P$.
Using the injective module $I_{A^{\op}}(P)$ over the opposite ring $A^\op$,
we define
\begin{equation*}
T_A(P):=\Hom_{R}(I_{A^{\op}}(P),E_{R}(R/\p)),
\end{equation*}
which is an indecomposable flat cotorsion right $A$-module (\cref{CompletionNecessity}) and also an indecomposable projective right module over $\widehat{A_\p}:=(A_\p)^\wedge_\p$ (\cref{DecompOfCompOfAlg}).
The following is one of the main results of this paper:

\begin{theorem}[\cref{ClassifOfFlCotOverNoethAlg}]\label{intro.ClassifOfFlCotOverNoethAlg}
	Let $A$ be a Noether $R$-algebra. A right $A$-module $M$ is flat cotorsion if and only if $M$ is isomorphic to
	\begin{equation*}
		\prod_{P\in\Spec A}(T_A(P)^{(B_{P})})_{\p}^{\wedge}
	\end{equation*}
	for some family of sets $\{B_{P}\}_{P\in\Spec A}$, where $T_A(P)^{(B_{P})}$ is the direct sum of $B_P$-indexed copies of $T_A(P)$ and $\p:=P\cap R$. The cardinality of each $B_{P}$ is uniquely determined by $M$.
\end{theorem}

This theorem recovers Enochs' result because $T_R(\p)\cong \widehat{R_\p}$ and $(T_R(\p)^{(B_{\p})})_{\p}^{\wedge}\cong (R_\p^{(B_{\p})})_{\p}^{\wedge}$ for each $\p\in\Spec R$ and any set $B_\p$.
Moreover, each component of the direct product in \cref{intro.ClassifOfFlCotOverNoethAlg} has a description
\begin{equation*}
(T_{A}(P)^{(B_P)})_{\p}^{\wedge}\cong \Hom_{R}(I_{A^{\op}}(P),E_{R}(R/\p)^{(B_P)}),
\end{equation*}
which recovers the isomorphism \cref{intro.EnochsGriffith} (see \cref{DualOfInj}).

As a consequence of \cref{intro.ClassifOfFlCotOverNoethAlg}, we obtain the following result:

\begin{corollary}[\cref{ClassifOfIndecFlCotOverNoethAlg}]\label{intro.ClassifOfIndecFlCotOverNoethAlg}
	Let $A$ be a Noether $R$-algebra. Then there is a one-to-one correspondence 
	\begin{equation*}
		\Spec A\isoto\setwithtext{isoclasses of indecomposable flat cotorsion right $A$-modules}
	\end{equation*}
	given by $P\mapsto T_{A}(P)$.
\end{corollary}

We denote by $\inj_{A}$ (\resp $\flcot_{A}$) the set of the isoclasses of indecomposable injective (\resp flat cotorsion) right $A$-modules.
By \cref{intro.ClassifOfIndecInjOverNoethAlg,intro.ClassifOfIndecFlCotOverNoethAlg}, there is a bijection $\inj_{A^\op}\isoto \flcot_{A}$ given by $I_{A^\op}(P)\mapsto T_A(P)$.  
We interpret this bijection as a phenomenon on Ziegler spectra.

An exact sequence of right modules over a ring $A$ is said to be \emph{pure exact} if its exactness is preserved by the functor $-\otimes_AU$ for every left $A$-module $U$. A right $A$-module $N$ is called \emph{pure-injective} if the functor $\Hom_A(-,N)$ sends pure exact sequences to exact sequences.
The isoclasses of indecomposable pure-injective right modules form a topological space $\Zg_A$ called the \emph{Ziegler spectrum} of $A$.
There is a bijection, called  \emph{elementary duality}, between the open subsets of $\Zg_A$ and those of $\Zg_{A^\op}$. Note that this does not mean that these topological spaces are homeomorphic in general.
Our assumption that $A$ is a Noether $R$-algebra ensures that $\inj_{A}$ and $\flcot_{A}$ are closed subsets of $\Zg_A$.
We endow $\inj_{A}$ and $\flcot_{A}$ with the topologies induced from $\Zg_A$.

\begin{theorem}[\cref{HomeoBwFlCotAndInjOverNoethAlg,FlCotAndInjAreRefl}]\label{intro.HomeoBwFlCotAndInjOverNoethAlg}
	Let $A$ be a Noether $R$-algebra. The bijection $\inj_{A^{\op}}\isoto \flcot_{A}$ given by $I_{A^{\op}}(P)\mapsto T_{A}(P)$ is a homeomorphism. The open sets of these topological spaces bijectively correspond to the specialization-closed subsets of $\Spec A$.
\end{theorem}

It should be mentioned that Herzog \cite{MR1091706} observed the existence of a homeomorphism $\inj_{A^\op}\isoto \flcot_{A}$ for a certain class of rings, which includes all left noetherian rings.
The homeomorphism was obtained as a restriction of a bijection between certain points of Ziegler spectra, called \emph{reflexive points}; for each reflexive point $N\in\Zg_{A^\op}$, the corresponding reflexive point $DN\in\Zg_A$ is determined by the property that the closure of $N$ corresponds to the closure of $DN$ by elementary duality (regarded as a bijection for closed subsets). The following result, together with \cref{intro.HomeoBwFlCotAndInjOverNoethAlg}, shows that our homeomorphism in \cref{intro.HomeoBwFlCotAndInjOverNoethAlg} is an explicit realization of Herzog's homeomorphism for Noether algebras:

\begin{corollary}\label{intro.ElementaryDualityAndHomeomorphisms}
Let $A$ be a Noether $R$-algebra. 
For each $P\in \Spec A$, the points $I_{A^{\op}}(P)\in\Zg_{A^{\op}}$ and $T_{A}(P)\in\Zg_{A}$ are the unique generic points in their closures, and these closed subsets correspond to each other by elementary duality.
\end{corollary}

This paper is organized as follows. In \cref{sec.Prelim}, we recall basic facts on Noether algebras, including those on the flat cotorsion pair and pure-injective modules. 
In \cref{sec.RedToLocComp}, we show that every flat cotorsion module over a Noether $R$-algebra $A$ can be decomposed as a direct product of $\p$-local $\p$-complete modules for various $\p\in \Spec R$. 
In \cref{sec.LocCompFlModAsFlCov}, we prove that each $\p$-local $\p$-complete flat (\resp $\p$-local $\p$-torsion injective) $A$-module is a flat cover (\resp injective envelope) of a semisimple $A_\p$-module. 
Furthermore, we observe that the flat cover (\resp injective envelope) of a semisimple right $A_\p$-module can be obtained by applying a variant of Matlis duality to the injective envelope (\resp flat cover) of a simple left $A_\p$-module.
In \cref{DescripOfLocCompFlMod}, we show that every $\p$-local $\p$-complete flat $A$-module is cotorsion and such a module is characterized as the $\p$-adic completion of a direct sum of indecomposable projective modules over $\widehat{A_{\p}}$.
In \cref{sec.DescripOfFlCotMod}, we complete the proofs of 
\cref{intro.ClassifOfFlCotOverNoethAlg,intro.ClassifOfIndecFlCotOverNoethAlg}.
In \cref{sec.FlCotModAsFlCovCotEnv}, we give a result that realizes flat cotorsion $A$-modules as nontrivial flat covers and pure-injective (or cotorsion) envelopes.
In \cref{sec.ZgAndElementaryDual}, we first recall some known results on Ziegler spectra and elementary duality, and then show that Herzog's homeomorphism applied to a Noether algebra coincides with the homeomorphism in \cref{intro.HomeoBwFlCotAndInjOverNoethAlg}.
\cref{sec.IdealAdicComp} provides some basic facts on ideal-adic completion over Noether algebras, which are used throughout the paper.

\subsection*{Acknowledgments}
\label{subsec.Acknowledgment}

Ryo Kanda was supported by JSPS KAKENHI Grant Numbers JP16H06337, JP17K14164, and JP20K14288, Leading Initiative for Excellent Young Researchers, MEXT, Japan, and Osaka City University Advanced Mathematical Institute (MEXT Joint Usage/Research Center on Mathematics and Theoretical Physics JPMXP0619217849).

Tsutomu Nakamura was supported by Grant-in-Aid for JSPS Fellows JP20J01865.

\section{Preliminaries}
\label{sec.Prelim}

Throughout the paper, let $A$ be a Noether $R$-algebra unless otherwise specified. That is, $R$ is a commutative noetherian ring, $A$ is a ring together with a ring homomorphism $R\to A$, called the \emph{structure homomorphism}, whose image is contained in the center of $A$, and $A$ is finitely generated as an $R$-module. It follows that $A$ is a left and right noetherian ring. We denote by $\Mod A$ the category of right $A$-modules, and interpret $\Mod A^\op$ as the category of left $A$-modules, where $A^\op$ is the opposite ring of $A$.

In this section, we collect some known results, which we will use in later sections. 

\subsection{Cotorsion modules and pure-injective modules}
\label{subsec.CotAndPInj}

	A right $A$-module $M$ is called \emph{cotorsion} if $\Ext_{A}^{1}(F,M)=0$ for all flat right $A$-modules $F$. A \emph{flat cotorsion module} is a module that is flat and cotorsion. A short exact sequence $0\to L\to M\to N\to 0$ in $\Mod A$ is said to be \emph{pure exact} if it remains exact after applying $-\otimes_{A}U$ for every $U\in\Mod A^{\op}$.
A right $A$-module $N$ is called \emph{pure-injective} if $\Hom_{A}(-,N)$ sends each pure exact sequence in $\Mod A$ to an exact sequence. Every injective module is pure-injective by definition.

\begin{proposition}\label{PureInjectiveCotorsion}
Every pure-injective right $A$-module is cotorsion. 
\end{proposition}

\begin{proof}
See \cite[Lemma~5.3.23]{MR1753146}.
\end{proof}

\begin{proposition}\label{InjDualProperties}
	Fix an injective $R$-module $E$ and consider the exact functor
	\begin{equation*}
		(-)^{*}:=\Hom_{R}(-,E)\colon\Mod A \to\Mod A^\op.
	\end{equation*}
	For a right $A$-module $M$, the following hold:
	\begin{enumerate}
		\item\label{InjDualProperties.Cot} $M^{*}$ is pure-injective, and hence cotorsion.
		\item\label{InjDualProperties.Double} If $E$ is an injective cogenerator, then the canonical morphism $M\to M^{**}$ is a pure monomorphism. In particular, this map splits if $M$ is pure-injective.
		\item\label{InjDualProperties.FlBecomesInj} 
	 If $M$ is flat, then $M^{*}$ is injective. The converse holds if $E$ is an injective cogenerator.
	 		\item\label{InjDualProperties.InjBecomesFlCot} 
		If $M$ is injective, then $M^{*}$ is flat (cotorsion). The converse holds if $E$ is an injective cogenerator.
			\end{enumerate}
\end{proposition}

\begin{proof}
	\cref{InjDualProperties.Cot}: See \cite[Proposition~5.3.7]{MR1753146} or \cite[Proposition~4.3.29]{MR2530988}.
	
	\cref{InjDualProperties.Double}: See \cite[Proposition~5.3.9]{MR1753146} or \cite[Corollary~2.21(b)]{MR2985554}.
	
	\cref{InjDualProperties.FlBecomesInj}: See \cite[Theorem~3.2.9]{MR1753146}.
	
	\cref{InjDualProperties.InjBecomesFlCot}: See \cite[Theorem~3.2.16]{MR1753146}.
\end{proof}

\begin{proposition}\label{InjDualProperties3}
Every flat cotorsion right $A$-module is pure-injective.
\end{proposition}

\begin{proof}
This is \cite[Lemma~3.2.3]{MR1438789}, but we give a proof here as it will be used in the proof of \cref{FlCotIffDSummandOfProdOfDual}.

Let $E$ be an injective cogenerator in $\Mod R$, and put $(-)^*:=\Hom_R(-,E)$.
If $M$ is a flat right $A$-module, then $M^{*}$ is injective and $M^{**}$ is flat by \cref{InjDualProperties}\cref{InjDualProperties.FlBecomesInj} and \cref{InjDualProperties.InjBecomesFlCot}. 
Thus the cokernel of the pure monomorphism $M\to M^{**}$ in \cref{InjDualProperties}\cref{InjDualProperties.Double} is flat. If in addition $M$ is cotorsion, then the pure monomorphism splits, so $M$ is pure-injective by  \cref{InjDualProperties}\cref{InjDualProperties.Cot}.
\end{proof}

By \cref{PureInjectiveCotorsion,InjDualProperties3}, a flat cotorsion right $A$-module is nothing but a flat pure-injective right $A$-module.

\begin{remark}\label{NonnoethereianCase}
Although we are focusing on a Noether $R$-algebra, \cref{PureInjectiveCotorsion}, and \cref{PInjEnv,SpeFlCovCotEnv} below hold for an arbitrary ring. \cref{InjDualProperties}\cref{InjDualProperties.Cot,InjDualProperties.Double,InjDualProperties.FlBecomesInj} hold for a ring $A$ together with a ring homomorphism from a commutative ring $R$ to the center of $A$. The first claim of \cref{InjDualProperties.InjBecomesFlCot} holds if in addition $A$ is right coherent, and the second claim holds if $A$ is right noetherian; see \cite[Corollary~2.18(b)]{MR2985554}.
\cref{InjDualProperties3} and \cref{CotEnvAndPInjEnvForFlatMod} below hold for a left coherent ring.
\end{remark}

\subsection{Covers and envelopes}
\label{subsec.CovAndEnv}

Let $\cA$ be an additive category and let $\cX$ be a full subcategory of $\cA$ closed under isomorphisms. A morphism $f\colon N\to M$ in $\cA$ is called \emph{right minimal} if every $g\in\End_{\cA}(N)$ satisfying $fg=f$ is an isomorphism. A \emph{left minimal} morphism is defined dually, that is, it is a morphism that is right minimal in the opposite category.

A morphism $f\colon X\to M$ in $\cA$ is called an \emph{$\cX$-precover}, or a \emph{right $\cX$-approximation}, if $X\in\cX$ and, for every $X'\in\cX$, the induced map $\Hom_{\cA}(X',X)\to\Hom_{\cA}(X',M)$ is surjective. The latter condition means that every morphism from an object in $\cX$ to $M$ factors through $f$. An \emph{$\cX$-cover}, or a \emph{right minimal $\cX$-approximation}, is an $\cX$-precover $X\to M$ that is right minimal.  It is immediate that an $\cX$-cover is unique up to isomorphism in the sense that, if $f\colon X\to M$ and $f'\colon X'\to M$ are $\cX$-covers, then there exists an isomorphism $h\colon X'\to X$ such that $fh=f'$. An \emph{$\cX$-preenvelope} (or a \emph{left $\cX$-approximation}) and an \emph{$\cX$-envelope} (or a \emph{left minimal $\cX$-approximation}) are defined dually.
If an $\cX$-cover $X\to M$ (\resp an $\cX$-envelope $M\to X$) exists, then the object $X$ is often called the $\cX$-cover (\resp the $\cX$-envelope) of $M$ since the isoclass (i.e., isomorphism class) of $X$ is uniquely determined by $M$.

Now let $\cA$ be an abelian category. A \emph{cotorsion pair} in $\cA$ is a pair $(\cX,\cY)$ of full subcategories of $\cA$ such that
\begin{align*}
	\cX&=\setwithcondition{M\in\cA}{\textnormal{$\Ext_{\cA}^{1}(M,Y)=0$ for all $Y\in\cY$}}\quad\text{and}\\
	\cY&=\setwithcondition{M\in\cA}{\textnormal{$\Ext_{\cA}^{1}(X,M)=0$ for all $X\in\cX$}}.
\end{align*}
A cotorsion pair $(\cX,\cY)$ is called \emph{hereditary} if $\Ext_{\cA}^{i}(X,Y)=0$ for all $X\in\cX$, $Y\in\cY$, and $i\geq 1$. A cotorsion pair $(\cX,\cY)$ is \emph{complete} if, for every $M\in\cA$, there exist exact sequences
\begin{align*}
	0\to Y\to X\to M\to 0\quad\text{and}\quad 0\to M\to Y'\to X'\to 0
\end{align*}
with $X,X'\in\cX$ and $Y,Y'\in\cY$. Morphisms $X\to M$ and $M\to Y'$ fitting into such exact sequences are often called a \emph{special $\cX$-precover} and a \emph{special $\cY$-preenvelope}, respectively. It is easy to see that they are indeed an $\cX$-precover and a $\cY$-preenvelope.

Denote by $\Flat A$ (\resp $\Cot A$) the full subcategory of $\Mod A$ consisting of all flat (\resp cotorsion) modules. If $\cX=\Flat A$, then an $\cX$-(pre)cover is called a \emph{flat (pre)cover}, which is necessarily an epimorphism. If $\cY=\Cot A$, then a $\cY$-(pre)envelope is called a \emph{cotorsion (pre)envelope}, which is necessarily a monomorphism. 
It is known that $(\Flat A,\Cot A)$ is a complete hereditary cotorsion pair in $\Mod A$ and every right $A$-module has a flat cover and a cotorsion envelope (see \cite[the proof of Proposition~3.1.2, Lemma~3.4.1, and Theorem~3.4.6]{MR1438789} and \cite{MR1832549}), where these facts are proved for an arbitrary ring.
Given a right $A$-module $M$, we denote the flat cover of $M$ by $F_{A}(M)\to M$ and the cotorsion envelope of $M$ by $M\to C_{A}(M)$.

Projective (pre)covers, injective (pre)envelopes, and pure-injective (pre)envelopes can be defined in the same way. A projective precover is merely an epimorphism from a projective module, and an injective preenvelope is merely a monomorphism to an injective module.
An injective envelope is nothing but an essential monomorphism to an injective module.
Pure-injective (pre)envelopes also have an alternative characterization, as in \cref{PInjEnv}. Recall that a monomorphism $L\to M$ is called a \emph{pure monomorphism} if it fits into a pure exact sequence. 
Moreover, a pure monomorphism $f\colon L\to M$ is called a \emph{pure-essential monomorphism} if, for every morphism $h\colon M\to M'$ such that $hf$ is a pure monomorphism, $h$ is a pure monomorphism.

\begin{proposition}\label{PInjEnv}
	Let $f\colon M\to N$ be a morphism in $\Mod A$ with $N$ pure-injective.
	\begin{enumerate}
		\item\label{PInjEnv.Preenv} $f$ is a pure-injective preenvelope if and only if $f$ is a pure monomorphism.
		\item\label{PInjEnv.Env} $f$ is a pure-injective envelope if and only if $f$ is a pure-essential monomorphism. 
			\end{enumerate}
\end{proposition}

\begin{proof}
\cref{PInjEnv.Preenv}: The ``if'' part is straightforward.
To show the ``only if'' part, suppose that $f$ is a pure-injective preenvelope. Let 
$E$ be an injective cogenerator in $\Mod R$ and set $(-)^{*}:=\Hom_{R}(-,E)$. Then we have the canonical pure monomorphism $g\colon M\to M^{**}$, where $M^{**}$ is pure-injective; see \cref{InjDualProperties}\cref{InjDualProperties.Cot,InjDualProperties.Double}.
Then there is a morphism $h\colon N\to M^{**}$ such that $hf=g$. Since $g$ is a pure monomorphism, it follows that $f$ is a pure monomorphism.
		
\cref{PInjEnv.Env}:	 If $f$ is a pure-essential monomorphism, then it is a pure-injective preenvelope by \cref{PInjEnv.Preenv}
and is also left minimal by the definition of pure-essentiality
because every pure monomorphism from a pure-injective module splits. To show the ``only if'' part, suppose that $f$ is a pure-injective envelope.
Let $h\colon N\to N'$ be a morphism such that $hf$ is a pure monomorphism.
To observe that $h$ is a pure monomorphism, it suffices to show that the composition of $h$ with the canonical pure monomorphism $N'\to N'^{**}$ is a pure monomorphism. Therefore, replacing $N'^{**}$ by $N'$, we may assume that $N'$ is pure-injective. Then the morphism $hf\colon M \to N'$ is a pure-injective preenvelope, but then $h$ is a split monomorphism since $f$ is a pure-injective envelope.
\end{proof}

\begin{remark}
Our pure-essentiality is the same as that of \cite[p.~145]{MR2530988}, and this definition is, in general, strictly stronger than the classical definition, in which a pure monomorphism $f\colon L\to M$ is called a pure-essential monomorphism if, for every morphism $h\colon M\to M'$ such that $hf$ is a pure monomorphism, $h$ is a \emph{monomorphism}.
It has been known to experts that some of the proofs for the existence of pure-injective envelopes (pure-injective hulls) do not work due to this difference; see \cite[p.~197, Remarks]{MR1758412}. 
However, the notion of pure-injective envelopes is consistent in any case, and they do exist over any ring. For valid proofs on the existence of pure-injective envelopes, we refer the reader to \cite[Theorem~4.3.18]{MR2530988} or \cite[\S 18-5]{MR1301329}. The former uses a functor category, and the latter (based on the classical pure-essentiality) uses a cardinality argument. The definitions of pure-injective envelopes therein are given in different ways, but they both agree with ours defined as $\cX$-envelopes for the class $\cX$ of pure-injective modules; see \cite[Proposition~4.3.16]{MR2530988} and \cite[Theorem~18-5.9]{MR1301329}.
\end{remark}

As mentioned above, every right $A$-module $M$ has a pure-injective envelope, which is unique up to isomorphism. It is denoted by $M\to H_{A}(M)$, following the notation in \cite[\S4.3.3]{MR2530988}.

Flat precovers and cotorsion preenvelopes are not necessarily special, but flat covers and cotorsion envelopes are:

\begin{lemma}\label{SpeFlCovCotEnv}
The kernel of a flat cover is cotorsion. The cokernel of a cotorsion envelope is flat. 
\end{lemma}

\begin{proof}
This is a consequence of Wakamatsu's lemma; see \cite[Lemmas~2.1.1 and 2.1.2]{MR1438789} or \cite[Lemma~5.3.25 and Proposition~7.2.4]{MR1753146}.
\end{proof}

By this lemma and \cref{InjDualProperties3}, a cotorsion envelope of a flat right $A$-module is a pure monomorphism into a pure-injective module, that is, a pure-injective preenvelope, by \cref{PInjEnv}\cref{PInjEnv.Preenv}. It is a pure-injective envelope by the left minimality of the cotorsion envelope. Hence we have:

\begin{proposition}\label{CotEnvAndPInjEnvForFlatMod}
For a flat right $A$-module, its cotorsion envelope and pure-injective envelope coincide. 
\end{proposition}

\begin{remark}\label{ArtinFlcovProjcov}
Over a right artinian ring, all flat right (and also left) modules are projective (\cite[Theorem~28.4 and Corollary~28.8]{MR1245487}), and hence all right (and left) modules are cotorsion. So flat covers and projective covers are the same notion, and cotorsion envelopes are identity morphisms.
\end{remark}

\subsection{Prime ideals and localization}
\label{subsec.PrimeIdealsLocalization}

An \emph{ideal} means a two-sided ideal unless otherwise specified. A \emph{prime ideal} of $A$ is an ideal $P\subsetneq A$ such that, for any $a,b\in A$, the condition $aAb\subset P$ implies that $a\in P$ or $b\in P$. A \emph{maximal ideal} of $A$ is an ideal $Q\subsetneq A$ that is maximal among all ideals except $A$ itself. Every maximal ideal is a prime ideal. Denote by $\Spec A$ (\resp $\Max A$) the set of all prime (\resp maximal) ideals of $A$.

Denote by $\phi\colon R\to A$ the structure homomorphism of the Noether $R$-algebra $A$. This homomorphism induces a canonical map $\Spec A\to\Spec R$ which sends each $P\in \Spec R$ to its preimage $\phi^{-1}(P)$; see \cref{MapOfSpectra} below. Although $R$ is not necessarily a subring of $A$, we write $P\cap R$ for $\phi^{-1}(P)$.

\begin{lemma}\label{SpecOfNoethAlgMax}
For every $P\in\Spec A$, we have $P\in\Max A$ if and only if $P\cap R\in\Max R$. In particular, the map $\Spec A\to\Spec R$ restricts to $\Max A\to\Max R$.
\end{lemma}

\begin{proof}
This follows from \cite[10.2.12 and 10.2.13]{MR1811901}. We give a more direct proof in \cref{NonnoetherianSpectra} below for the reader's convenience.
\end{proof}

For each $\p\in\Spec R$, the $R_{\p}$-module $A_{\p}$, the localization of $A$ at $\p$ as an $R$-module, is naturally a Noether $R_{\p}$-algebra. Moreover, for every right $A$-module $M$, the localization $M_{\p}$ has a structure of a right $A_{\p}$-module, and it holds that $M_\p\cong M\otimes_RR_\p \cong M\otimes_AA_\p$.
We say that $M$ is \emph{$\p$-local} if the canonical $A$-homomorphism $M\to M_{\p}$ is an isomorphism. In this case, $M$ itself can be regarded as a right $A_{\p}$-module.

\begin{remark}\label{LocAdj}
The localization functor $(-)_\p\colon\Mod A\to \Mod A_{\p}$ has a fully faithful right adjoint $\Mod A_{\p}\to \Mod A$, which sends each $A_\p$-module to itself but regarded as an $A$-module along the canonical ring homomorphism $A\to A_\p$. The essential image of the right adjoint consists of all $\p$-local right $A$-modules.
\end{remark}

\begin{proposition}\label{LocOfNoethAlg}
	Let $\p\in\Spec R$.
	\begin{enumerate}
		\item\label{LocOfNoethAlg.Spec} There is an order-preserving bijection
		\begin{equation*}
			\setwithcondition{P\in\Spec A}{P\cap R\subset\p}\isoto\Spec A_{\p}
		\end{equation*}
		given by $P\mapsto P_{\p}=PA_{\p}$. The inverse map is given by $Q\mapsto f^{-1}(Q)$, where $f\colon A\to A_{\p}$ is the canonical ring homomorphism.
		\item\label{LocOfNoethAlg.Max} The bijection in \cref{LocOfNoethAlg.Spec} restricts to a bijection
		\begin{equation*}
			\setwithcondition{P\in\Spec A}{P\cap R=\p}\isoto\Max A_{\p}.
		\end{equation*}
	\end{enumerate}
\end{proposition}

\begin{proof}
	\cref{LocOfNoethAlg.Spec}: This follows from \cite[2.1.16, Proposition(vii)]{MR1811901}. See also \cref{NonnoetherianSpectra} below.
	
	\cref{LocOfNoethAlg.Max}: Let $P\in\Spec A$ such that $P\cap R\subset\p$. \cref{SpecOfNoethAlgMax} applied to the Noether $R_{\p}$-algebra $A_{\p}$ implies that $PA_{\p}\in\Max A_{\p}$ if and only if $PA_{\p}\cap R_{\p}\in\Max R_{\p}$. Since $PA_{\p}\cap R_{\p}=(P\cap R)R_{\p}$, the latter condition is equivalent to $P\cap R=\p$.
\end{proof}

For $\p\in\Spec R$, the residue field at $\p$ is denoted by $\kappa(\p):=R_{\p}/\p R_{\p}$. 
Note that $A\otimes_{R}\kappa(\p)$ is a \emph{finite-dimensional} $\kappa(\p)$-algebra in the sense that the Noether $\kappa(\p)$-algebra $A\otimes_{R}\kappa(\p)$ is finite-dimensional as a $\kappa(\p)$-vector space. In particular, it is a left and right artinian ring.

\begin{proposition}\label{SpecOfNoethAlg}
	Let $\p\in\Spec R$. There is a bijection
	\begin{equation*}
		\setwithcondition{P\in\Spec A}{P\cap R=\p}\isoto\Spec(A\otimes_{R}\kappa(\p))=\Max(A\otimes_{R}\kappa(\p))
	\end{equation*}
	given by $P\mapsto P_{\p}/\p A_{\p}$. The inverse map is given by $Q\mapsto f^{-1}(Q)$, where $f\colon A\to A\otimes_{R}\kappa(\p)$ is the canonical ring homomorphism.
	
	Consequently, the fiber $\setwithcondition{P\in\Spec A}{P\cap R=\p}$ over each $\p\in \Spec R$ is a (possibly empty) finite set.
\end{proposition}
\begin{proof}
The bijection in \cref{LocOfNoethAlg}\cref{LocOfNoethAlg.Spec} induces an injection 
\begin{equation*}
		\setwithcondition{P\in\Spec A}{P\cap R=\p}\to \setwithcondition{Q\in\Spec A_\p}{\p A_\p\subseteq Q},
		\end{equation*}
and an elementary argument shows that every $P\in \Spec A$ with $P\cap R \subseteq \p$ and $\p A_\p\subseteq P_\p$ satisfies $P\cap R=\p$, so the above injection is bijective, and the right-hand side can naturally be identified with $\Spec (A\otimes_{R}\kappa(\p))$ since $A\otimes_{R}\kappa(\p)=A_{\p}/\p A_{\p}$. Thus we have the desired bijection.
By \cite[Theorem~3.4 and Proposition~4.19]{MR2080008}, $A\otimes_{R}\kappa(\p)$ has only finitely many prime ideals, which are all maximal.
\end{proof}

If the structure homomorphism $R\to A$ is injective, then  the induced map $\Spec A\to \Spec R$ is surjective, i.e., each fiber is nonempty (\cite[10.2.9, Theorem]{MR1811901}).

\begin{remark}\label{NonnoetherianSpectra}
\cref{LocOfNoethAlg}\cref{LocOfNoethAlg.Spec} holds for a ring $A$ together with a ring homomorphism from a commutative ring $R$ to the center of $A$. In fact, it can be proved in a similar way to the case $A=R$ (\cite[p.~22, Theorem~4.1 and Example~2]{MR1011461}); note that, if $P\in\Spec A$, $a\in A$, $s\in R\setminus(P\cap R)$, then $as\in P$ implies that $aAs\subset P$ and hence $a\in P$. \cref{SpecOfNoethAlgMax}, \cref{LocOfNoethAlg}\cref{LocOfNoethAlg.Max}, and \cref{SpecOfNoethAlg} also hold if in addition $A$ is finitely generated as an $R$-module. 
We give here a proof of \cref{SpecOfNoethAlgMax}, which works in this setting.

The ``only if'' part of the lemma follows from \cite[10.2.10, Corollary(iii)]{MR1811901}, and it can also be proved as follows: Let $P \in \Max A$ and set $\p := P \cap R$. Then we have an injection $R/\p \hookrightarrow A/P$, so we may suppose $P=0$ and $\p=0$, and hence $R$ is a domain and $A$ is a simple ring. Suppose that there is a prime ideal $0\neq \q$ of $R$. Localization of the injection $R \hookrightarrow A$ at $\q$ yields an injection $R_\q \hookrightarrow A_\q$, where $A_\q$ is a simple ring by \cref{LocOfNoethAlg}\cref{LocOfNoethAlg.Spec}. Since $A_\q$ is a nonzero finitely generated $R_\q$-module, $A_\q/\q A_\q$ is nonzero by Nakayama's lemma, so the ring $A_\q/\q A_\q$ contains a prime ideal. This means that $A_\q$ contains a prime ideal $Q$ with $\q A_\q \subseteq Q$, but $0\neq \q R_\q \subseteq \q A_\q$, so this contradicts the fact that $A_\q$ is simple. 

To prove the ``if'' part, assume that $\p:=P\cap R\in\Max R$. Then $R/\p$ is a field, and $A\otimes_{R}(R/\p)=A/\p A$ is a finite-dimensional $(R/\p)$-algebra. Since $\p A\subset P$, we have a canonical surjective ring homomorphism $A/\p A\to A/P$, and hence $A/P$ is also a finite-dimensional $(R/\p)$-algebra. So all prime ideals of $A/P$ are maximal ideals by \cite[Proposition~4.19]{MR2080008}. In particular, the zero ideal $0=P/P$ of $A/P$ is a maximal ideal. Therefore $P\in\Max A$.

\end{remark}

\begin{remark}\label{MapOfSpectra}
In general, for a ring homomorphism $f\colon A\to B$ of noncommutative rings and a prime ideal $Q$ of $B$, the ideal $f^{-1}(Q)$ of $A$ is not necessarily a prime ideal; see \cite[10.2.3]{MR1811901}.

However, if we assume that $B$ is a \emph{centralizing} extension of $f(A)$ (cf.\ \cite[10.1.3]{MR1811901}), that is, as a right (or equivalently, left) $f(A)$-module, $B$ is generated by a (possibly infinite) subset $S\subset B$ such that every element of $S$ commutes with every element of $f(A)$, then $f^{-1}(Q)$ is a prime ideal of $A$ for every prime ideal $Q$ of $B$ (cf.\ \cite[10.2.4, Theorem]{MR1811901}). The proof is straightforward. This assumption is satisfied if the homomorphism $f$ is surjective or $f(A)$ is contained in the center of $B$.
\end{remark}

\subsection{Simple modules and injective modules}
\label{subsec.SimpleIndecInjAnd}

We will recall that (semi)simple modules and injective modules over a Noether $R$-algebra $A$ are controlled by maximal ideals and prime ideals, respectively. First we assign a simple module to each prime ideal.

Let $P\in \Spec A$ and put $\p:=P\cap R$. The ring $A_{\p}/P_{\p}$ is a simple right artinian ring, and hence decomposes as a finite direct sum of copies of a simple right $A_{\p}/P_{\p}$-module, where the simple module is unique up to isomorphism (\cite[Theorems (3.3) and (3.10)]{MR1125071}).
We denote the simple right $A_{\p}/P_{\p}$-module by $S_{A}(P)$ and its multiplicity in $A_{\p}/P_{\p}$ by $n_{P}$, that is,
\begin{equation}\label{nP}
A_{\p}/P_{\p}\cong S_A(P)^{n_P}.
\end{equation}
By construction, $S_{A}(P)\cong S_{A_{\p}}(P_{\p})$ and $S_{A}(P)$ is also a simple right $A_\p$-module. It is often regarded as a right $A$-module (which is not necessarily simple).

Denote by $\rad A$ the \emph{Jacobson radical of $A$}, which is the intersection of all annihilators of simple right (or left) $A$-modules (or equivalently, the intersection of all maximal right (or left) ideals of $A$; see \cite[Proposition~3.16]{MR2080008}). In general, the annihilator of a simple module over an arbitrary ring is a prime ideal, and any maximal (two-sided) ideal is the annihilator of some simple module (\cite[Proposition~3.15]{MR2080008}).
In particular, the Jacobson radical of a finite-dimensional algebra over a field (or more generally, a right artinian ring) equals to the intersection of all maximal ideals (\cite[Corollary~4.16 and Proposition~4.19]{MR2080008}). The following fact implies that the same characterization holds for a Noether $R$-algebra $A$:

\begin{theorem}\label{SimpleOverNoethAlg}
	There is a bijection
	\begin{equation*}
		\Max A\isoto\setwithtext{isoclasses of simple right $A$-modules}
	\end{equation*}
	given by $P\mapsto S_{A}(P)$. The inverse map is given by $S\mapsto\Ann_{A}(S)$.
\end{theorem}

\begin{proof}
	For a simple right $A$-module $S$, let $P:=\Ann_{A}S$ and $\p:=P\cap R$. Then, by \cite[Proposition~9.1(a) and Corollary~9.5]{MR2080008}, $P$ is a maximal ideal of $A$ and the right $A$-module $A/P$ is a finite direct sum of copies of $S$. Since each $a\in R\setminus\p$ does not annihilate $S$, it acts as an isomorphism on the simple $A$-module $S$. This means that $S$ is $\p$-local, and hence $A/P=A_{\p}/P_{\p}$ is a finite direct sum of copies of $S$. Therefore $S_{A}(P)$ is isomorphic to $S$ by the definition of $S_{A}(P)$.
	
	Let $Q\in\Max A$ and $\q:=Q\cap R$. Again by the definition of $S_{A}(Q)$, we have $\Ann_{A}S_{A}(Q)=\Ann_{A}(A_{\q}/Q_{\q})=Q$. This completes the proof.
\end{proof}

It follows from \cref{SimpleOverNoethAlg} that
\begin{equation}\label{JacobsonRadical}
\rad A=\bigcap_{P\in \Max A} P.
\end{equation}
Given $\p\in \Spec R$, we have $\Max A_\p=\setwithcondition{P_\p}{\textnormal{$P\in \Spec A$, $P\cap R=\p$}}$ by \cref{LocOfNoethAlg}. Thus \cref{JacobsonRadical} implies that
	\begin{equation}\label{JacobsonRadicalLocal}
	\rad A_\p=\bigcap_{\substack{P\in\Spec A\\P\cap R=\p}}P_\p.
	\end{equation}

\begin{proposition}\label{TopOfFiber}
	For every $\p\in\Spec R$, we have
	\begin{equation*}
		A_\p/\rad A_\p\cong\bigoplus_{\substack{P\in\Spec A\\P\cap R=\p}}S_{A}(P)^{n_{P}}
	\end{equation*}
	as right $A$-modules.
\end{proposition}

\begin{proof}
	We have  
	\begin{equation*}
		A_\p/\rad A_\p\cong\prod_{\substack{P\in\Spec A\\P\cap R=\p}}A_{\p}/P_{\p}\cong\bigoplus_{\substack{P\in\Spec A\\P\cap R=\p}}S_{A}(P)^{n_{P}},
	\end{equation*}
	where the first isomorphism of rings follows from \cref{JacobsonRadicalLocal} and the Chinese remainder theorem (\cite[\S7, Exercise 13]{MR1245487}), and the second follows from \cref{SpecOfNoethAlg} and \cref{nP}.
\end{proof}

\begin{remark}\label{JacobsonRadicalAlgebra}
By \cref{LocOfNoethAlg}\cref{LocOfNoethAlg.Max} and \cref{SpecOfNoethAlg}, the maximal ideals of $A_\p$ naturally correspond to those of $\overline{A_\p}:=A\otimes_R\kappa(\p)=A_{\p}/\p A_\p$, and hence the canonical surjection $A_{\p}\to\overline{A_{\p}}$ induces the isomorphism
\begin{equation*}
A_\p/\rad A_\p\isoto\overline{A_\p}/\rad \overline{A_\p}
\end{equation*}
of finite-dimensional $\kappa(\p)$-algebras.
\end{remark}

Now we turn our attention to injective modules. The following remark will be used later:

\begin{remark}\label{ClosureEnvCovLocal}
Let $\p\in \Spec R$. Injective envelopes, pure-injective envelopes, cotorsion envelopes, and flat covers in $\Mod A_\p$ are those in $\Mod A$. In particular, in view of \cref{LocAdj}, the full subcategory of $\Mod A$ formed by $\p$-local modules is closed under taking such envelopes and covers.

Indeed, left or right minimality of morphisms for $\p$-local $A$-modules is the same as that in $\Mod A_\p$. Thus we only need to show that such envelopes and covers in $\Mod A_\p$ become preenvelopes and precovers in $\Mod A$, respectively. (Pure-)injective envelopes in $\Mod A_\p$ are (pure-)monomorphisms, and they are also (pure-)monomorphisms in $\Mod A$. Flat covers and cotorsion envelopes in $\Mod A_\p$ have cotorsion kernels and flat cokernels in $\Mod A_\p$, respectively. So the desired claims will follow if we show that an injective (\resp pure-injective, cotorsion, flat) $A_\p$-module is also injective (\resp pure-injective, cotorsion, flat) in $\Mod A$. This can easily be observed by using the exactness of the left adjoint $(-)_\p\colon\Mod A \to \Mod A_\p$ to the scalar restriction functor $\Mod A_\p\to \Mod A$.
\end{remark}

Now we assign an indecomposable injective module to each prime ideal of $A$. Let $P\in \Spec A$ and put $\p:=P\cap R$. Take injective envelopes $A/P\to E_{A}(A/P)$ and $A_\p/P_\p \to E_{A_\p}(A_\p/P_\p)$ in $\Mod A$ and $\Mod A_\p$, respectively. As observed in \cref{ClosureEnvCovLocal}, $E_{A_\p}(A_\p/P_\p)\cong E_{A}(A_\p/P_\p)$. The canonical $A$-homomorphism $A/P\to A_\p/P_\p$ is injective, and it extends to a monomorphism $E_{A}(A/P) \to E_{A}(A_\p/P_\p)$, which splits. So $E_A(A/P)$ is $\p$-local. Localizing the injective envelope $A/P\to E_A(A/P)$, we obtain an essential extension $A_\p/P_\p\to E_A(A/P)$ in $\Mod A$. Therefore 
\begin{equation}\label{InjectiveEnvelopeDecomposition}
E_A(A/P)\cong E_{A}(A_\p/P_\p) \cong E_A(S_A(P))^{n_P},
\end{equation}
where the second isomorphism follows from \cref{nP}.
This fact is essentially observed in the proof of \cite[Proposition~2.5.2]{MR1898632}. 

We set 
\begin{equation*}
I_A(P):=E_A(S_A(P)),
\end{equation*}
which is an indecomposable injective right $A$-module.
By construction, it holds that 
\begin{equation}\label{IndecomposableInjectiveLocal}
I_{A}(P)\cong I_{A_{\p}}(P_{\p}),
\end{equation}
so $I_{A}(P)$ is $\p$-local.
If $A=R$, then $P=\p$, so $I_A(P)=E_R(R/\p)$ and $I_{A_{\p}}(P_{\p})=E_{R_\p}(\kappa(\p))$.

\begin{theorem}\label{InjOverNoethRing}
There is a bijection
		\begin{equation*}
			\Spec A\isoto\setwithtext{isoclasses of indecomposable injective right $A$-modules}
		\end{equation*}
		given by $P\mapsto I_{A}(P)$.
\end{theorem}

\begin{proof}
See \cite[Lemma~5.14, Proposition~9.1(a), and Theorem~9.15]{MR2080008}.
\end{proof}

\subsection{Matlis duality}
\label{subsec.MatlisDual}

Completion and Matlis duality play a central role in the proof of our main results. Here, and also in \cref{sec.IdealAdicComp}, we collect some basic facts on these operations.

For an ideal $\ka\subset R$, define the \emph{$\ka$-adic completion functor} $\varLambda^{\ka}=(-)_{\ka}^{\wedge}\colon\Mod A\to\Mod A$ by
	\begin{equation*}
	\varLambda^{\ka}M=M_{\ka}^{\wedge}:=\varprojlim_{n\geq 1}M/\ka^{n}M.\end{equation*}
We say that a right $A$-module $M$ is \emph{$\ka$-adically complete} (or \emph{$\ka$-complete} for short) if the canonical $A$-homomorphism $M\to M_{\ka}^{\wedge}$ is an isomorphism.
In particular, $M_{\ka}^{\wedge}$ is $\ka$-complete (\cref{CompletionIdempotent}). 
The $\ka$-adic completion $A^\wedge_\ka$ of $A$ naturally has a ring structure, and the canonical map $A\to A^\wedge_\ka$ is a ring homomorphism. Moreover, for each right $A$-module $M$, $M^\wedge_\ka$ has a (unique) right $A^\wedge_\ka$-module structure that is compatible with the right A-module structure on $M^\wedge_\ka$ (see \cref{TwoActionOfCompletion} and \cref{UniqueCompleteStructure}). 
We often write $M_{\ka}^{\wedge}$ as $\widehat{M}$ when $M$ is $\p$-local and $\ka=\p$ for some $\p\in \Spec R$. For example, $\widehat{A_{\p}}$ means $(A_{\p})_{\p}^{\wedge}$.

If $M$ is a finitely generated right $A$-module, then the canonical $A$-homomorphism $M\otimes_RR_{\ka}^{\wedge} \to M_{\ka}^{\wedge}$ is an isomorphism (\cite[Theorem~8.7]{MR1011461}), so $M^\wedge_\ka$ is finitely generated as an $R_{\ka}^{\wedge}$-module.
The structure map $R\to A$ induces a ring homomorphism $R^\wedge_\ka\to A\otimes_R R^\wedge_\ka\cong A^\wedge_\ka$ whose image is contained in the center, and $R_{\ka}^{\wedge}$ is a commutative noetherian ring (\cite[Theorem~8.12]{MR1011461}). Thus $A_{\ka}^{\wedge}$ is a Noether $R_{\ka}^{\wedge}$-algebra. In particular, $\widehat{A_{\p}}$ is a Noether $\widehat{R_{\p}}$-algebra for each $\p\in\Spec R$. Since $R^\wedge_\ka$ is flat over $R$ (\cite[Theorem~8.8]{MR1011461}) and $A\otimes_{R}R^\wedge_\ka \cong R^\wedge_\ka\otimes_{R}A$, $A^\wedge_\ka$ is flat as a left and ring $A$-module.
If $R$ is $\ka$-complete, then all finitely generated $R$-modules are $\ka$-complete (\cite[Theorem~8.7]{MR1011461}), and hence all finitely generated right $A$-modules are $\ka$-complete.

When $R$ is a local ring with maximal ideal $\km$ and residue field $k$, the $\km$-adic completion $\widehat{R}$ is a (commutative noetherian) local ring with maximal ideal $\widehat{\km}=\km \widehat{R}$, and $k=R/\km \cong (R/\km) \otimes_R \widehat{R}\cong \widehat{R}/\widehat{\km}$; see \cite[Proposition~10.16]{MR0242802} or \cite[p.~63]{MR1011461}. The completion map $R\to\widehat{R}$ is a faithfully flat ring homomorphism (\cite[Theorem~8.14]{MR1011461}), thus a pure monomorphism in $\Mod R$ by \cite[Theorem~7.5(i)]{MR1011461}. Applying $-\otimes_{R}A$ to the completion map, we conclude that $\widehat{A}$ is faithfully flat right $A$-module and the completion map $A\to\widehat{A}$ is a pure monomorphism in $\Mod A$.

Dually, the \emph{$\ka$-torsion functor} $\varGamma_\ka\colon\Mod A \to \Mod A$ is defined by
\begin{equation*}
\varGamma_\ka:=\varinjlim_{n\geq 1} \Hom_R(R/\ka^n,-).
\end{equation*}
For a right $A$-module $M$, we have $\varGamma_\ka M=\bigcup_{n\geq 1}\setwithcondition{x\in M}{x \ka^n=0}$. We say that $M$ is \emph{$\ka$-torsion} if the canonical inclusion $\varGamma_\ka M\to M$ is an isomorphism.
It is well-known that $E_R(R/\p)$ is $\p$-torsion for each $\p\in \Spec R$ (\cite[Theorem~18.4(v)]{MR1011461}), and more generally, $I_A(P)$ is $\p$-torsion for each $P\in \Spec A$ and $\p:=P\cap R$; see \cref{ArtinianInjective} below.

\begin{remark}\label{TostionAdjoint}
The $\ka$-torsion functor $\varGamma_\ka$ is left exact and commutes with arbitrary direct sums. Moreover, we can regard $\varGamma_\ka$ as a right adjoint to the inclusion functor from the full subcategory consisting of $\ka$-torsion $A$-modules to $\Mod A$. See also \cref{CompletionPreservesSurjectivity,CommutativityWithDirestProduct,CompAdj} for analogous facts on $\varLambda^{\ka}$.
\end{remark}

The following fact relates $\ka$-torsion modules with $\ka$-complete modules:

\begin{proposition}\label{InjDualSendsTorsToComp}
	Let $E$ be an injective $R$-module. Then the functor $\Hom_{R}(-,E)\colon\Mod A\to\Mod A^{\op}$ sends $\ka$-torsion right $A$-modules to $\ka$-complete left $A$-modules.
\end{proposition}

\begin{proof}
	If $M$ is an $\ka$-torsion right $A$-module, then it is isomorphic to $\varinjlim_{n\geq 1}\Hom_{R}(R/\ka^{n},M)$, so we have
		\begin{align*}
		\Hom_{R}(M,E)
		&\cong\varprojlim_{n\geq 1}\Hom_{R}(\Hom_{R}(R/\ka^{n},M),E)
		\cong\varprojlim_{n\geq 1}\Hom_{R}(M,E)\otimes_{R}(R/\ka^{n})=\Hom_{R}(M,E)_{\ka}^{\wedge}
	\end{align*}
	as left $A$-modules; see \cite[Theorem~3.2.11]{MR1753146} for the second isomorphism. Hence $\Hom_{R}(M,E)$ is $\ka$-complete by \cref{CompletionIdempotent}. 
\end{proof}

In the rest of this section, we assume that $R$ is local, and denote its maximal ideal and residue field by $\km$ and $k$, respectively.
The functor $\Hom_{R}(-,E_{R}(k))\colon\Mod R \to \Mod R$ gives rise to a duality, known as \emph{Matlis duality}, between the category of finitely generated $R$-modules and the category of artinian $R$-modules, provided that $R$ is $\km$-adically complete.
This duality naturally extends to the case of Noether algebras over complete local rings (\cref{MatlisDualOverNoetherianAlgebra}) as shown in \cite[Proposition~2.6.1]{MR1898632}.
Let us observe how the proof goes, collecting related facts used in later sections.
We refer the reader to \cite[Theorem~18.6]{MR1011461}, \cite[Proposition~3.2.12 and Theorem~3.2.13]{MR1251956+}, and \cite[Appendix~A, \S4]{MR2355715} for classical results on Matlis duality for the commutative case; they will be used in the rest of the section.

First, for every $\km$-torsion right $A$-module $M$,
there is a canonical isomorphism 
\begin{equation}\label{TorsionTensorCompletion}
M\isoto M\otimes_{R}\widehat{R}
\end{equation}
of right $A$-modules (\cref{TorsionIsomorphism}). This makes $M$ an $\widehat{\km}$-torsion right $\widehat{A}$-module via the isomorphism $A\otimes_{R}\widehat{R}\cong \widehat{A}$. In fact, this is the unique right $\widehat{A}$-module structure on $M$ that is compatible with the right $A$-module structure (\cref{UniqueCompleteStructure}).
Moreover, all $A$-submodules of $M$ are also $\widehat{A}$-submodules by \cref{TorsionTensorCompletion} (and vice versa), so $M$ is artinian (\resp of finite length, simple) as a right $A$-module if and only if $M$ is artinian (\resp of finite length, simple) as a right $\widehat{A}$-module; see also \cref{TorsionEquivalence}.

A key to Matlis duality is that the injective envelope $E_{R}(k)$ is an artinian $R$-module (hence $\km$-torsion). The above arguments applied to $A=R$ makes $E_{R}(k)$ an artinian $\widehat{R}$-module, and it coincides with the injective envelope of $k\cong \widehat{R}/\widehat{\km}$ in $\Mod\widehat{R}$, that is, $E_R(k)\cong E_{\widehat{R}}(k)$.

If $M$ is a finitely generated right $A$-module, then it is, as an $R$-module, a quotient of a finitely generated free $R$-module, so its Matlis dual $\Hom_R(M,E_R(k))$ is an $R$-submodule of a finite direct sum of copies of $E_{R}(k)$. This implies that $\Hom_R(M,E_R(k))$ is artinian as an $R$-module. Thus $\Hom_R(M,E_R(k))$ is an artinian left $A$-module.
Consequently, we obtain a contravariant functor 
\begin{equation}\label{MatlisDual1}
\Hom_R(-,E_R(k))\colon\mod A \to \artin A^{\op},
\end{equation}
where $\mod A$ is the category of finitely generated right $A$-modules and $\artin A^\op$ is the category of artinian left $A$-modules.

Another key to Matlis duality is that the completion map $R\to \widehat{R}$ is identified with the canonical ring homomorphism $R\to\End_R(E_R(k))$ via the isomorphism $\widehat{R}\isoto\End_{R}(E_{R}(k))$ given by the action of $\widehat{R}$ on $E_{R}(k)$.
For every finitely generated right $A$-module $M$, the standard isomorphisms $M\otimes_R\widehat{R}\cong \widehat{M}$ and $M\otimes_R\Hom_R(E_R(k), E_R(k))\cong \Hom_R(\Hom_R(M,E_R(k)), E_R(k))$ of right $\widehat{A}$-modules give a natural isomorphism
\begin{equation}\label{DoubleMatlisDual1}
\widehat{M}\isoto\Hom_R(\Hom_R(M,E_R(k)),E_R(k))
\end{equation}
of right $\widehat{A}$-modules.

On the other hand, if a given right $A$-module $M$ is artinian as an $R$-module (we will later show that every artinian right $A$-module satisfies this), then $M$ can be regarded as a right $\widehat{A}$-module via \cref{TorsionTensorCompletion}. Since such $M$ can be embedded as an $\widehat{R}$-submodule into a finite direct sum of copies of $E_{R}(k)\cong E_{\widehat{R}}(k)$, the left $\widehat{A}$-module $\Hom_{R}(M,E_{R}(k))$ is finitely generated as an $\widehat{R}$-module, and hence as an $\widehat{A}$-module. Matlis duality for $\widehat{R}$ implies that there is a natural isomorphism
\begin{equation}\label{DoubleMatlisDual2}
M\isoto\Hom_{\widehat{R}}(\Hom_{R}(M,E_{R}(k)), E_{\widehat{R}}(k))
\end{equation}
of $\widehat{R}$-modules. Since this isomorphism commutes with the action of $A$, this is an isomorphism of right $\widehat{A}$-modules.

It remains to see that every artinian $A$-module is artinian as an $R$-module. 
To this end, we prove the following fact, in which we make use of \cref{SimpleOverNoethAlg}:

\begin{proposition}\label{SimpleMatlisDuality}
Let $(R,\km, k)$ be a commutative noetherian local ring and let $A$ be a Noether $R$-algebra. For every $P\in\Max A$, there is an isomorphism
\begin{equation*}
	\Hom_{R}(S_{A^\op}(P),E_{R}(k))\cong S_A(P)
\end{equation*}
of right $A$-modules. This realizes the bijection
\begin{equation*}
\setwithtext{isoclasses of simple left $A$-modules}\isoto\setwithtext{isoclasses of simple right $A$-modules}
\end{equation*}
defined by $S_{A^\op}(P)\mapsto S_{A}(P)$.
\end{proposition}

\begin{proof}
Let $P\in \Max A$. Recall that $S_{A^\op}(P)$ is of finite length as an $R$-module since it is a finite-dimensional $k$-vector space; see \cref{SpecOfNoethAlgMax} and \cref{nP}. 
By \cref{DoubleMatlisDual2}, we have an isomorphism $S_{A^{\op}}(P)\isoto \Hom_{\widehat{R}}(\Hom_{R}(S_{A^{\op}}(P),E_{R}(k)), E_{\widehat{R}}(k))$ of left $\widehat{A}$-modules, where $\Hom_{R}(S_{A^{\op}}(P),E_{R}(k))$ is $\km$-torsion. Then $\Hom_{R}(S_{A^{\op}}(P),E_{R}(k))$ has to be simple as a right $\widehat{A}$-module, or equivalently, as a right $A$-module.
Thus, by \cref{SimpleOverNoethAlg}, $\Hom_{R}(S_{A^{\op}}(P),E_{R}(k))\cong S_{A}(Q)$ for some $Q\in\Max A$. Since the left-hand side is annihilated by $P$, we have $P\subset Q$. Hence $P=Q$ since $P$ is also maximal.
\end{proof}

\begin{remark}\label{ArtinianInjective}
Let $P\in \Max A$. Then $P\cap R=\km$ (\cref{SpecOfNoethAlgMax}). 
As shown in \cite[Proposition~2.5.5]{MR1898632}, the injective envelope $I_A(P)$ of $S_{A}(P)$ is a direct summand of $\Hom_R(A, E_R(k))$.
Indeed, there is a surjection $A\to S_{A^\op}(P)$ in $\Mod A^{\op}$ by construction, and $\Hom_R(-, E_R(k))$ sends this map  to an injection $S_{A}(P)\to\Hom_R(A, E_R(k))$ in $\Mod A$ by \cref{SimpleMatlisDuality}. Since $\Hom_R(A, E_R(k))$ is an injective right $A$-module by \cref{InjDualProperties}\cref{InjDualProperties.FlBecomesInj}, it contains $I_A(P)$ as a direct summand.

Consequently, $I_A(P)$ is artinian as an $R$-module because 
$\Hom_R(A, E_R(k))$ is an artinian $R$-module as we observed before \cref{MatlisDual1}.
In particular, $I_A(P)$ is $\km$-torsion, and hence becomes a right $\widehat{A}$-module, which is artinian as a right $A$-module and as a right $\widehat{A}$-module.
\end{remark}

Let us finally verify that every artinian right $A$-module $M$ is artinian as an $R$-module.
The socle $\soc_{A}M$ is a finite direct sum of simple $A$-modules and it is an essential $A$-submodule of $M$. Thus $E_{A}(M)\cong E_{A}(\soc_{A}M)$, and the right-hand side is a finite direct sum of copies of indecomposable injective modules $I_{A}(P)$ for various $P\in\Max A$; see \cref{SimpleOverNoethAlg}. 
Hence $E_{A}(M)$ is artinian as a right $R$-module by \cref{ArtinianInjective}, and so is $M$.
Therefore, 
\begin{equation}\label{ArtinianCharacterization}
\artin A=\setwithcondition{M\in \Mod A}{\textnormal{$M$ is artinian as an $R$-module}},
\end{equation}
where the inclusion ``$\supseteq$'' is trivial.
We have observed that there is a contravariant functor 
\begin{equation}\label{MatlisDual2}
\Hom_R(-,E_R(k))\colon\artin A^{\op}\to \mod \widehat{A}.
\end{equation}

Combining \cref{MatlisDual1}--\cref{MatlisDual2}, we obtain Matlis duality for a Noether algebra over a complete local ring:

\begin{theorem}\label{MatlisDualOverNoetherianAlgebra}
Let $(R,\km, k)$ be a commutative noetherian complete local ring and let $A$ be a Noether $R$-algebra. Then the contravariant functors 
$\mod A\to\artin A^{\op}$ and $\artin A^{\op}\to\mod A$ induced by $\Hom_R(-,E_R(k))$ are mutually quasi-inverse equivalences.
\end{theorem}

\section{Decomposition of flat cotorsion modules into local complete modules}
\label{sec.RedToLocComp}

Let $A$ be a Noether $R$-algebra. In this section, we show that every flat cotorsion right $A$-module is decomposed as a direct product of $\p$-local $\p$-complete flat cotorsion modules for various $\p\in\Spec R$
(\cref{DecompToLocCompFlCot}).

The argument in this section is based on Enochs' idea that was used to describe flat cotorsion modules over a commutative noetherian ring (\cite[p.~183]{MR754698}). However, we present our generalized proof in a more precise manner for the sake of clarity.

As the first step, we prove the following lemma. Note that for a module $M$ and a set $B$, we denote by $M^{(B)}$ (\resp $M^B$) the direct sum (\resp direct product) of $B$-indexed copies of $M$.

\begin{lemma}\label{FlCotIffDSummandOfProdOfDual}
	A right $A$-module $M$ is flat cotorsion if and only if $M$ is a direct summand of
	\begin{equation}\label{eq.ProdOfDual}
		\prod_{P\in\Spec A}\Hom_{R}(I_{A^{\op}}(P),E_{R}(R/\p)^{(B_{P})})
	\end{equation}
	for some family of sets $\{B_{P}\}_{P\in\Spec A}$, where $\p:=P\cap R$ in each component.
\end{lemma}

\begin{proof}
	Since $R$ is noetherian, the direct sum $E_{R}(R/\p)^{(B_{P})}$ of injective $R$-modules is again injective and, by \cref{InjDualProperties}\cref{InjDualProperties.InjBecomesFlCot}, each $\Hom_{R}(I_{A^{\op}}(P),E_{R}(R/\p)^{(B_{P})})$ is a flat cotorsion right $A$-module.
	It is straightforward to see that the product \cref{eq.ProdOfDual} is cotorsion, and it is also flat because $A$ is left coherent (see \cite[Theorem~2.1]{MR120260} or \cite[Theorem~3.2.24]{MR1753146}). Therefore every direct summand of \cref{eq.ProdOfDual} is flat cotorsion.
	
	Conversely, suppose that $M$ is flat cotorsion. As in the proof of \cref{InjDualProperties3}, $M$ is a direct summand of $\Hom_{R}(I,E)$ for an injective cogenerator $E$ in $\Mod R$, where $I:=\Hom_{R}(M,E)$ is an injective left $A$-module. Since $A$ is left noetherian, $I$ decomposes as a direct sum of indecomposable injective left $A$-modules (\cite[Theorem~2.5]{MR99360}). Hence, using \cref{InjOverNoethRing}, we have
	\begin{equation}\label{IndecomposableDecomposition}
		I\cong \bigoplus_{P\in\Spec A}I_{A^{\op}}(P)^{(C_{P})}
	\end{equation}
	for some family of sets $\{C_{P}\}_{P\in\Spec A}$. Then
	\begin{equation*}
		\Hom_{R}(I,E)
		\cong \prod_{P\in\Spec A}\Hom_{R}(I_{A^{\op}}(P),E)^{C_{P}}
		\cong\prod_{P\in\Spec A}\Hom_{R}(I_{A^{\op}}(P),E^{C_{P}})
	\end{equation*}
as right $A$-modules.	
	Now fix $P\in\Spec A$ and let $\p:=P\cap R$. Since $I_{A^{\op}}(P)$ is $\p$-local by \cref{IndecomposableInjectiveLocal}, we have $I_{A^{\op}}(P)\cong I_{A^{\op}}(P)\otimes_{R}R_{\p}$, and hence
	\begin{equation*}
		\Hom_{R}(I_{A^{\op}}(P),E^{C_{P}})
		\cong\Hom_{R}(I_{A^{\op}}(P)\otimes_{R}R_{\p},E^{C_{P}})
		\cong\Hom_{R}(I_{A^{\op}}(P),\Hom_{R}(R_{\p},E^{C_{P}})).
	\end{equation*}
Notice that $\Hom_{R}(R_{\p},E^{C_{P}})$ is an injective $R$-module by \cref{InjDualProperties}\cref{InjDualProperties.FlBecomesInj}. Since it is $\p$-local, it cannot contain any $\q$-torsion submodule unless $\q\subseteq\p$. Therefore we have 
	\begin{equation*}
		\Hom_{R}(R_{\p},E^{C_{P}})\cong \bigoplus_{\substack{\q\in\Spec R\\\q\subset\p}}E_{R}(R/\q)^{(C_{P}^{\q})}
	\end{equation*}
	for some family of sets $\{C_{P}^\q\}_{\q\subset\p}$.
	Then  
	\begin{equation*}
		\Hom_{R}(I_{A^{\op}}(P),E^{C_{P}})
		\cong\Hom_{R}(I_{A^{\op}}(P),\bigoplus_{\substack{\q\in\Spec R\\\q\subset\p}}E_{R}(R/\q)^{(C_{P}^{\q})})
		\cong \Hom_{R}(I_{A^{\op}}(P),E_{R}(R/\p)^{(C_{P}^{\p})}),
	\end{equation*}
	where the last isomorphism follows from \cref{TostionAdjoint} (and \cref{CompAndColocOfInj}\cref{CompAndColocOfInj.Comp} below) because $I_{A^{\op}}(P)$ is $\p$-torsion (\cref{ArtinianInjective}) and each $E_{R}(R/\q)$ is $\q$-local.
	Setting $B_{P}:=C_{P}^{\p}$, we conclude that $\Hom_{R}(I,E)$ is of the form \cref{eq.ProdOfDual}.
\end{proof}

\begin{remark}\label{DualOfInjIsLocComp}
	Each component $\Hom_{R}(I_{A^{\op}}(P),E_{R}(R/\p)^{(B_{P})})$  in  \cref{eq.ProdOfDual} is $\p$-local and $\p$-complete, by \cref{IndecomposableInjectiveLocal}, \cref{InjDualSendsTorsToComp}, and \cref{ArtinianInjective}.
	Moreover, we can rewrite \cref{eq.ProdOfDual} as $\prod_{\p\in \Spec R} M(\p)$, where $M(\p)$ is
	\begin{equation*}
	\bigoplus_{\substack{P\in \Spec A\\ P\cap R=\p}} \Hom_{R}(I_{A^{\op}}(P),E_{R}(R/\p)^{(B_{P})}),
	\end{equation*}
	which is a finite direct sum due to \cref{SpecOfNoethAlg}.
	\end{remark}

By \cref{FlCotIffDSummandOfProdOfDual} and \cref{DualOfInjIsLocComp},
every flat cotorsion right $A$-module $M$ is a direct product of $\p$-local $\p$-complete flat cotorsion modules for various $\p\in\Spec R$. The next result shows that the isoclass of the component at $\p$ is uniquely determined by $M$.

\begin{lemma}\label{CompAndColocOfFlCot}
	Let $M(\p)$ be a $\p$-local $\p$-complete right $A$-module for each $\p\in\Spec R$, and let $M:=\prod_{\p\in\Spec R}M(\p)$.
	\begin{enumerate}
		\item\label{CompAndColocOfFlCot.Comp} For every ideal $\ka\subset R$, the canonical morphism $M\to\varLambda^{\ka}M$ is a split epimorphism, and
		\begin{equation*}
			\varLambda^{\ka}M=\prod_{\substack{\p\in\Spec R\\\ka\subset\p}}M(\p)
		\end{equation*}
		as quotient modules of $M$, where the right-hand side is regarded as a quotient module via the projection.
		\item\label{CompAndColocOfFlCot.Coloc} For every multiplicatively closed set $S\subset R$, the canonical morphism $\Hom_{R}(S^{-1}R,M)\to\Hom_{R}(R,M)\isoto M$ is a split monomorphism, and
		\begin{equation*}
			\Hom_{R}(S^{-1}R,M)=\prod_{\substack{\p\in\Spec R\\\p\cap S=\emptyset}}M(\p)
		\end{equation*}
		as submodules of $M$, where the right-hand side is regarded as a submodule via the inclusion.
		
		In particular, for every $\q\in\Spec R$,
		\begin{equation*}
			\Hom_{R}(R_{\q},M)=\prod_{\substack{\p\in\Spec R\\\p\subset\q}}M(\p).
		\end{equation*}
		\item\label{CompAndColocOfFlCot.Isom} Let $\q\in\Spec R$, and let $\iota_{\q}\colon M(\q)\to M$ and $\pi_{\q}\colon M\to M(\q)$ be the inclusion and the projection, respectively. Then
		$\varLambda^{\q}\Hom_{R}(R_{\q},\iota_{\q})$ and $\varLambda^{\q}\Hom_{R}(R_{\q},\pi_{\q})$ are isomorphisms, and 
		\begin{equation*}
		\varLambda^{\q}\Hom_{R}(R_{\q},M)\cong M(\q).
		\end{equation*}
	\end{enumerate}
\end{lemma}

\begin{proof}
	\cref{CompAndColocOfFlCot.Comp}: 
	Since $\varLambda^{\ka}$ commutes with arbitrary direct products (\cref{CommutativityWithDirestProduct}), the canonical morphism $ M\to\varLambda^{\ka}M$ can be naturally identified with the direct product of the canonical morphisms $M(\p)\to\varLambda^{\ka}M(\p)$ for $\p\in \Spec R$. If $\ka\subset\p$, then the $\p$-complete module $M(\p)$ is also $\ka$-complete (\cref{bCompleteISaComplete}), so the canonical morphism $M(\p)\to\varLambda^{\ka}M(\p)$ is an isomorphism. If $\ka\not\subset\p$, then $\ka^n M(\p)=M(\p)$ for all $n\geq 1$ because $M(\p)$ is $\p$-local, so $\varLambda^{\ka}M(\p)=0$.
		
	\cref{CompAndColocOfFlCot.Coloc}: Similarly, since the functor $\Hom_{R}(S^{-1}R,-)$ commutes with arbitrary direct products, the canonical morphism $\Hom_R(S^{-1}R,M)\to M$ can be naturally identified with the direct product of 
		the canonical morphisms $\Hom_{R}(S^{-1}R,M(\p))\to M(\p)$ for $\p\in \Spec R$. 
	If $\p\cap S=\emptyset$, then $(S^{-1}R)_\p\cong R_\p$ so we can deduce from \cref{LocAdj} that $\Hom_{R}(S^{-1}R,M(\p))\cong \Hom_{R_\p}(R_\p,M(\p))\cong M(\p)$. If $\p\cap S\neq\emptyset$, then $\varLambda^\p S^{-1}R=0$, so 
	$\Hom_{R}(S^{-1}R,M(\p))\cong \Hom_{R}(\varLambda^\p S^{-1}R,M(\p))=0$ by \cref{CompAdj} since $M(\p)$ is $\p$-complete.
	
	\cref{CompAndColocOfFlCot.Isom}: This follows from \cref{CompAndColocOfFlCot.Comp} and \cref{CompAndColocOfFlCot.Coloc}.
\end{proof}

There is a dual statement of \cref{CompAndColocOfFlCot} for a direct sum of local torsion modules. We state it as a remark because the proof is immediate in view of \cref{TostionAdjoint}.

\begin{remark}\label{CompAndColocOfInj}
	Let $M(\p)$ be a $\p$-local $\p$-torsion right $A$-module for each $\p\in\Spec R$, and let $M:=\bigoplus_{\p\in\Spec R}M(\p)$.
	\begin{enumerate}
		\item\label{CompAndColocOfInj.Comp} For every ideal $\ka\subset R$, the canonical morphism $\varGamma_\ka M\to M$ is a split monomorphism, and
		\begin{equation*}
			\varGamma_{\ka}M=\bigoplus_{\substack{\p\in\Spec R\\\ka\subset\p}}M(\p)
		\end{equation*}
		as submodules of $M$, where the right-hand side is regarded as a submodule via the inclusion.
		\item\label{CompAndColocOfInj.Loc} For every multiplicatively closed set $S\subset R$, the canonical morphism $M\to M\otimes_RS^{-1}R$ is a split epimorphism, and
		\begin{equation*}
			M\otimes_RS^{-1}R=\bigoplus_{\substack{\p\in\Spec R\\\p\cap S=\emptyset}}M(\p)
		\end{equation*}
		as quotient modules of $M$, where the right-hand side is regarded as a quotient module via the projection.
		
		In particular, for every $\q\in\Spec R$,
		\begin{equation*}
			M_{\q}=\bigoplus_{\substack{\p\in\Spec R\\\p\subset\q}}M(\p).
		\end{equation*}
		\item Let $\q\in\Spec R$, and let $\iota_{\q}\colon M(\q)\to M$ and $\pi_{\q}\colon M\to M(\q)$ be the inclusion and the projection, respectively. Then
		$\varGamma_{\q}(\iota_{\q}\otimes_RR_\q)$ and $\varGamma_{\q}(R_{\q},\pi_{\q}\otimes_RR_\q)$ are isomorphisms, and 
		\begin{equation*}
		\varGamma_\q(M\otimes_RR_\q)\cong M(\q).
		\end{equation*}
	\end{enumerate}
\end{remark}

\cref{CompAndColocOfFlCot} above and \cref{IsomOfProdIffIsomOfComponent} below are shown in \cite[Lemma~2.2]{MR3904746} and  \cite[Lemma~3.1]{MR3904746}, respectively, but those were for direct products of $\p$-local $\p$-complete \emph{flat} $R$-modules for various $\p\in \Spec R$.

\begin{lemma}\label{IsomOfProdIffIsomOfComponent}
	Let $M(\p)$ and $N(\p)$ be $\p$-local $\p$-complete right $A$-modules for each $\p\in\Spec R$. For an $A$-homomorphism $f\colon\prod_{\p\in\Spec R}M(\p)\to\prod_{\p\in\Spec R}N(\p)$, the following are equivalent:
	\begin{enumerate}
		\item\label{IsomOfProdIffIsomOfComponent.Isom} $f$ is an isomorphism.
		\item\label{IsomOfProdIffIsomOfComponent.Func} $\varLambda^{\q}\Hom_{R}(R_{\q},f)$ is an isomorphism for all $\q\in\Spec R$.
		\item\label{IsomOfProdIffIsomOfComponent.Component} The composition
		\begin{equation*}
			\begin{tikzcd}
				M(\q)\ar[r,hookrightarrow] & \prod_{\p\in\Spec R}M(\p)\ar[r,"f"] & \prod_{\p\in\Spec R}N(\p)\ar[r,twoheadrightarrow] & N(\q)
			\end{tikzcd}
		\end{equation*}
		is an isomorphism for all $\q\in\Spec R$, where the first morphism is the inclusion and the last one is the projection.
	\end{enumerate}
\end{lemma}

\begin{proof}
	\cref{IsomOfProdIffIsomOfComponent.Isom}$\Rightarrow$\cref{IsomOfProdIffIsomOfComponent.Func} is obvious.
	\cref{IsomOfProdIffIsomOfComponent.Func} and \cref{IsomOfProdIffIsomOfComponent.Component} are equivalent by \cref{CompAndColocOfFlCot}.
	The proof of \cref{IsomOfProdIffIsomOfComponent.Component}$\Rightarrow$\cref{IsomOfProdIffIsomOfComponent.Isom} is parallel to that of \cite[Lemma~3.1]{MR3904746} with $T_{\p}$ and $T'_{\p}$ replaced by $M(\p)$ and $N(\p)$, respectively; \cite[(3.2)]{MR3904746} is replaced by the fact that, for for every $\p\in\Spec R$,
	\begin{equation}\label{Orthogonality}
		\Hom_{A}(\prod_{\substack{\q\in\Spec R\\\p\not\subset\q}}M(\q),N(\p))=0,
	\end{equation}
	which follows from \cref{CompAdj} and \cref{CompAndColocOfFlCot}\cref{CompAndColocOfFlCot.Comp}.
\end{proof}

The next lemma recovers and generalizes \cite[Theorem~4.1.14]{MR1438789}, which deals with direct products of $\p$-local $\p$-complete flat $R$-modules for various $\p\in \Spec R$.

\begin{lemma}\label{DecompOfProdOfLocCompMod}
	Let $M(\p)$ be a $\p$-local $\p$-complete right $A$-module for each $\p\in\Spec R$. If we have a decomposition
	\begin{equation*}
		\prod_{\p\in\Spec R}M(\p)\cong M_{1}\oplus M_{2},
	\end{equation*}
	then there are right $A$-modules $M_{i}(\p)$, indexed by $\p\in\Spec R$ and $i=1,2$, such that
	\begin{equation*}
		M_{i}\cong\prod_{\p\in\Spec R}M_{i}(\p)
	\end{equation*}
	and $M(\p)\cong M_{1}(\p)\oplus M_{2}(\p)$ for each $\p\in\Spec R$.
\end{lemma}

\begin{proof}
	Let $M:= \prod_{\p\in\Spec R}M(\p)$.
	For each $\q\in\Spec R$, we have a canonical split monomorphism $\Hom_{R}(R_{\q},M)\to M$ and a canonical split epimorphism $\Hom_{R}(R_{\q},M)\to\varLambda^{\q}\Hom_{R}(R_{\q},M)$, which both become isomorphisms upon application of $\varLambda^\q\Hom_R(R_\q,-)$; see \cref{CompAndColocOfFlCot}. Since $M\cong M_{1}\oplus M_{2}$, the same holds for $M_{1}$ and $M_{2}$. For each $i=1,2$, let $h^{\q}_{i}$ be the composition
	\begin{equation*}
		\begin{tikzcd}
			M_{i}\ar[r,"g^{\q}_{i}"] & \Hom_{R}(R_{\q},M_{i})\ar[r] & \varLambda^{\q}\Hom_{R}(R_{\q},M_{i})=:M_{i}(\q),
		\end{tikzcd}
	\end{equation*}
	where $g^{\q}_{i}$ is an arbitrary splitting of the canonical split monomorphism $f^{\q}_{i}\colon\Hom_{R}(R_{\q},M_{i})\to M_{i}$. Since $\varLambda^\q\Hom_R(R_\q,f_i^\q)$ is an isomorphism, so is $\varLambda^\q\Hom_R(R_\q,g_i^\q)$.
	
Let $h_{i}\colon M_{i}\to\prod_{\p\in\Spec R}M_{i}(\p)$ be the morphism induced by the family $\set{h^{\p}_{i}}_{\p\in\Spec R}$. If we show that 
$h_{1}\oplus h_{2}\colon M_1\oplus M_2\to \left(\prod_{\p}M_{1}(\p)\right) \oplus \left(\prod_{\p} M_{2}(\p)\right)=\prod_{\p}(M_{1}(\p)\oplus M_{2}(\p))$
 is an isomorphism, then so are $h_{1}$ and $h_{2}$, and thus the desired conclusion follows since
	\begin{equation*}
	M(\p)\cong \varLambda^{\p}\Hom_{R}(R_{\p},M) \cong  \varLambda^{\p}\Hom_{R}(R_{\p},M_{1}\oplus M_{2})\cong M_{1}(\p)\oplus M_{2}(\p)
	\end{equation*}
	by \cref{CompAndColocOfFlCot}\cref{CompAndColocOfFlCot.Isom}.
	
	By \cref{IsomOfProdIffIsomOfComponent}, it suffices to show that $\varLambda^{\q}\Hom_{R}(R_{\q},h_{1}\oplus h_{2})$ is an isomorphism for each $\q\in\Spec R$. Consider the commutative diagram
	\begin{equation*}
		\begin{tikzcd}
			M_{1}\oplus M_{2}\ar[d,twoheadrightarrow,"g^{\q}_{1}\oplus g^{\q}_{2}"']\ar[r,"h_{1}\oplus h_{2}"] & \displaystyle\prod_{\p\in\Spec R}(M_{1}(\p)\oplus M_{2}(\p))\ar[dd,twoheadrightarrow,"\textnormal{projection}"] \\
			\Hom_{R}(R_{\q},M_{1}\oplus M_{2})\ar[d,twoheadrightarrow] & \\
			\varLambda^{\q}\Hom_{R}(R_{\q},M_{1}\oplus M_{2})\ar[r,equal] & M_{1}(\q)\oplus M_{2}(\q)\rlap{,}
		\end{tikzcd}
	\end{equation*}
where the vertical morphisms become isomorphisms upon application of $\varLambda^{\q}\Hom_{R}(R_{\q},-)$. Therefore $\varLambda^{\q}\Hom_{R}(R_{\q},h_{1}\oplus h_{2})$ is an isomorphism.
\end{proof}

\begin{proposition}\label{DecompToLocCompFlCot}
	A right $A$-module $M$ is flat cotorsion if and only if 
		 $M\cong \prod_{\p\in\Spec R}M(\p)$, where each $M(\p)$ is some $\p$-local $\p$-complete flat cotorsion module. The isoclass of $M(\p)$ is uniquely determined by $M$.
\end{proposition}

\begin{proof}
	The ``if'' part is clear; see the proof of \cref{FlCotIffDSummandOfProdOfDual}. The ``only if'' part follows from \cref{FlCotIffDSummandOfProdOfDual}, \cref{DualOfInjIsLocComp}, and \cref{DecompOfProdOfLocCompMod}. The uniqueness of $M(\p)$ follows from \cref{CompAndColocOfFlCot}\cref{CompAndColocOfFlCot.Isom}.
\end{proof}

\section{Local complete flat modules as flat covers}
\label{sec.LocCompFlModAsFlCov}

Let $A$ be a Noether $R$-algebra. 
In \cref{sec.RedToLocComp}, we observed that every flat cotorsion module is uniquely decomposed as a direct product of $\p$-local $\p$-complete ones. The purpose of the present section is to realize each $\p$-local $\p$-complete flat module as a flat cover of a semisimple $A_\p$-module (\cref{BijBetweenFlAndInjAndSemisimple}).

The key result in this section is \cref{FlCovAndInjEnvAndTop}, which tells us that a certain operation to a flat module yields a (nontrivial) flat cover.
The next three results are necessary steps for this;
\cref{NilpIdealAndProjCov,LiftingHomForFl} are inspired by \cite[Lemma~1.1]{MR4127282}, and \cref{FlCovAndInjEnvAndResField}\cref{FlCovAndInjEnvAndResField.FlCov} is a generalization of \cite[Proposition~4.1.6]{MR1438789}.

\begin{lemma}\label{NilpIdealAndProjCov}
	Let $J\subset A$ be a nilpotent ideal.
	\begin{enumerate}
	\item\label{NilpIdealAndProjCov.Right} For every flat right $A$-module $F$, the canonical morphism $F\to F\otimes_{A}(A/J)$ is right minimal in $\Mod A$.
	\item\label{NilpIdealAndProjCov.Left} For every injective right $A$-module $I$, the canonical morphism $\Hom_A(A/J, I)\to I$ is left minimal in $\Mod A$.
	\end{enumerate}
\end{lemma}

\begin{proof}
	This proof works for an arbitrary ring $A$.
	
	\cref{NilpIdealAndProjCov.Right}: Denote the canonical morphism by $f$ and let $g\in\End_{A}(F)$ such that $fg=f$. Since $f\otimes_{A}(A/J)$ is an isomorphism, so is $g\otimes_{A}(A/J)$. Applying $-\otimes_{A}(A/J)$ to the exact sequence
	\begin{equation*}
		\begin{tikzcd}
			F\ar[r,"g"] & F\ar[r] &\ar[r] \Cok g\ar[r] & 0,
		\end{tikzcd}
	\end{equation*}
	we obtain $(\Cok g)\otimes_{A}(A/J)=0$. Hence $\Cok g=(\Cok g)J=(\Cok g)J^{2}=\cdots$, but $J$ is nilpotent, so $\Cok g=0$. Applying $-\otimes_{A}(A/J)$ to
	\begin{equation*}
		\begin{tikzcd}
			0\ar[r] & \Ker g\ar[r] & F\ar[r,"g"] & F\ar[r] & 0,
		\end{tikzcd}
	\end{equation*}
	we obtain $(\Ker g)\otimes_{A}(A/J)=0$ since $F$ is flat. Hence $\Ker g=0$ by the same argument.
	
	\cref{NilpIdealAndProjCov.Left}:
Given a right $A$-module $M$, $\Hom_A(A/J,M)=0$ if and only if $\Hom_A(A/J^n,M)=0$ for every $n\geq 1$, since $\Hom_A(A/J^n,M)=\setwithcondition{x\in M}{xJ^n=0}$. Thus 
$\Hom_A(A/J,M)=0$ implies $M=0$, as $J$ is nilpotent. The rest of the proof is parallel to \cref{NilpIdealAndProjCov.Right}.
\end{proof}

\begin{lemma}\label{LiftingHomForFl}
	Let $\ka\subset R$ be an ideal such that $R/\ka$ is an artinian ring. For every $\ka$-complete flat right $A$-module $F$, the canonical morphism $F\to F\otimes_{R}(R/\ka)$ is a flat cover in $\Mod A$.
\end{lemma}

\begin{proof}
	Denote the canonical morphism by $f$. We first show that it is a flat precover. Let $h\colon G\to F\otimes_{R}(R/\ka)$ be an $A$-homomorphism from a flat right $A$-module $G$. Then $h$ naturally factors through an $A$-homomorphism $\overline{h}\colon G\otimes_{R}(R/\ka)\to F\otimes_{R}(R/\ka)$. Set $g_{1}:= \overline{h}$.
	For each $n\geq 1$, $G\otimes_{R}(R/\ka^{n})$ is a flat $A/\ka^{n}A$-module and it is actually projective since $A/\ka^{n}A$ is right artinian (\cref{ArtinFlcovProjcov}). Thus, there exist $A$-homomorphisms $g_{n}\colon G\otimes_{R}(R/\ka^{n})\to F\otimes_{R}(R/\ka^{n})$, for all $n\geq 2$, such that the diagram
	\begin{equation}\label{eq.inv.system}
		\begin{tikzcd}
			\cdots\ar[r,twoheadrightarrow] & G\otimes_{R}(R/\ka^{3})\ar[r,twoheadrightarrow]\ar[d,"g_{3}"] & G\otimes_{R}(R/\ka^{2})\ar[r,twoheadrightarrow]\ar[d,"g_{2}"] & G\otimes_{R}(R/\ka)\ar[d,"g_{1}"] \\
			\cdots\ar[r,twoheadrightarrow] & F\otimes_{R}(R/\ka^{3})\ar[r,twoheadrightarrow] & F\otimes_{R}(R/\ka^{2})\ar[r,twoheadrightarrow] & F\otimes_{R}(R/\ka)
		\end{tikzcd}
	\end{equation}
	commutes, where the horizontal maps are the canonical ones. Defining $g$ to be the composition of $\varprojlim_{n\geq 1}g_{n}\colon G^\wedge_\ka\to F^\wedge_\ka$ and the canonical isomorphism $F^\wedge_\ka\cong F$, we have a commutative diagram
	\begin{equation*}
		\begin{tikzcd}
			G^\wedge_\ka\ar[d,"g"']\ar[r]& G\otimes_{R}(R/\ka)\ar[d,"g_{1}"] \\
			F\ar[r,"f"] & F\otimes_{R}(R/\ka)\rlap{,}
		\end{tikzcd}
	\end{equation*}
	where the horizontal map in the first row is the canonical one.
	The composition of the completion map $G\to G^\wedge_\ka$ and $g\colon G^\wedge_\ka \to F$ is a lifting of $h\colon G\to F\otimes_{R}(R/\ka)$ because $g_{1}= \overline{h}$.
	This shows that $f$ is a flat precover. 
	
	Next we show that $f$ is right minimal.
	Take an arbitrary $g\in\End_{A}(F)$ such that the diagram
	\begin{equation*}
		\begin{tikzcd}[row sep=tiny]
			F\ar[dd,"g"']\ar[dr,"f"] & \\
			& F\otimes_{R}(R/\ka) \\
			F\ar[ur,"f"'] &
		\end{tikzcd}
	\end{equation*}
	commutes. Letting $g_{n}:=g\otimes_{R}(R/\ka^{n})=g\otimes_{A}(A/\ka^{n}A)$ for each $n\geq 1$, we obtain a diagram of the form \cref{eq.inv.system} with $G=F$ and $g_{1}=\id_{F/\ka F}$. Letting $A_n:= A/\ka^n A$ and $J_n:=\ka A/\ka^{n}A\subset A_n$, we have $g_{n}\otimes_{A_n}(A_n/J_n)=g_{n}\otimes_{A}(A/\ka A)=\id_{F/\ka F}$. Thus \cref{NilpIdealAndProjCov}\cref{NilpIdealAndProjCov.Right}, applied to $J_n\subset A_{n}$, implies that $g_{n}$ is an isomorphism for every $n\geq 2$. Then $g$ is an isomorphism as $g=\varprojlim_{n\geq 1}g_{n}$. This concludes that $f$ is a flat cover.
\end{proof}

\begin{remark}\label{DecompositionOfCompletion}
The assumption on $\ka$ in \cref{LiftingHomForFl} is satisfied if and only if the Zariski closed subset $V(\ka):=\setwithcondition{\p \in \Spec R}{\ka \subseteq \p}$ consists of only (necessarily finitely many) maximal ideals of $R$. If this is the case, then the $\ka$-adic completion functor $\varLambda^\ka$ decomposes as the finite direct product $\prod_{\km\in V(\ka)} \varLambda^\km$. Indeed, setting $\kb:=\bigcap_{\km\in V(\ka)}\km$, we have $\kb^{n}\subset\ka\subset\kb$ for some positive integer $n$, so the proof goes in a similar way to that of \cite[Theorem~8.15]{MR1011461}.
\end{remark}

\begin{proposition}\label{FlCovAndInjEnvAndResField}
	Let $\p\in\Spec R$.
	\begin{enumerate}
		\item\label{FlCovAndInjEnvAndResField.FlCov} For every $\p$-local $\p$-complete flat right $A$-module $F$, the canonical morphism $F\to F\otimes_{R}\kappa(\p)$ is a flat cover in $\Mod A$.
		\item\label{FlCovAndInjEnvAndResField.InjEnv} For every $\p$-local $\p$-torsion injective right $A$-module $I$, the canonical morphism $\Hom_{R}(\kappa(\p),I)\to I$  is an injective envelope in $\Mod A$.
	\end{enumerate}
\end{proposition}

\begin{proof}
By \cref{LocAdj,ClosureEnvCovLocal} along with the standard isomorphism $F\otimes_{R}\kappa(\p)\cong F\otimes_{R_\p}\kappa(\p)$, we may assume that $(R,\km, k)$ is a local ring and $\p=\km$.
Then \cref{FlCovAndInjEnvAndResField.FlCov} follows from \cref{LiftingHomForFl}. 
To prove  \cref{FlCovAndInjEnvAndResField.InjEnv}, notice that the canonical map $R\to \kappa(\p)=k$ is surjective and it induces an injection $\Hom_{R}(k,I)\to I$, which is clearly an injective preenvelope in $\Mod A$. As $I$ is $\km$-torsion by assumption, every nonzero $A$-submodule $M$ of $I$ satisfies $\Hom_R(k, M)\neq 0$. This implies that the morphism $\Hom_{R}(k,I)\to I$ is an essential monomorphism, so it is an injective envelope in $\Mod A$.
\end{proof}

\begin{proposition}\label{FlCovAndInjEnvAndTop}
	Let $\p\in\Spec R$.
	\begin{enumerate}
		\item\label{FlCovAndInjEnvAndTop.Fl} Let $F$ be a $\p$-local $\p$-complete flat right $A$-module. Then the canonical morphism
		\begin{equation*}
			F\to F\otimes_{A}(A_\p/\rad A_\p)
		\end{equation*}
 is a flat cover in $\Mod A$.
		\item\label{FlCovAndInjEnvAndTop.Inj} Let $I$ be a $\p$-local $\p$-torsion injective right $A$-module. The canonical morphism
		\begin{equation*}
			\Hom_{A}(A_\p/\rad A_\p,I)\to I
		\end{equation*}
		an injective envelope in $\Mod A$.
	\end{enumerate}
\end{proposition}

\begin{proof}
	As we observed in the proof of \cref{FlCovAndInjEnvAndResField},
	we may assume that $(R,\km, k)$ is a local ring and $\p=\km$. Put $J:= \rad A$. Since $J$ is the intersection of all maximal ideals of $A$ and $P\cap R=\km$ for all $P\in\Max A$ (see \cref{SpecOfNoethAlgMax} and \cref{JacobsonRadical}), we have $\km A\subset J$. Consequently, for every right $A$-module $M$, the canonical morphism $M\otimes_{R}k\cong M\otimes_{A}(A/\km A)\to M\otimes_{A}(A/J)$ is surjective and the canonical morphism $\Hom_{A}(A/J,M)\to\Hom_{A}(A/\km A,M)\cong\Hom_{R}(k,M)$ is injective.

	\cref{FlCovAndInjEnvAndTop.Fl}: Denote the given morphism by $f$. First we show that it is a flat precover. Let $h\colon G\to F\otimes_{A}(A/J)$ be a morphism from a flat right $A$-module $G$. It naturally factors through a morphism $\overline{h}\colon G\otimes_{A}(A/J)\to F\otimes_{A}(A/J)$. 
	By the projectivity of $G\otimes_{R}k$ in $\Mod (A\otimes_Rk)$,
	there exists a morphism $\overline{g}$ such that the diagram
	\begin{equation*}
		\begin{tikzcd}
			G\otimes_{R}k\ar[r]\ar[d,"\overline{g}"'] & \displaystyle G\otimes_{A}(A/J)\ar[d,"\overline{h}"] \\
			F\otimes_{R}k\ar[r] & \displaystyle F\otimes_{A}(A/J)
		\end{tikzcd}
	\end{equation*}
	commutes. Furthermore, by \cref{FlCovAndInjEnvAndResField}\cref{FlCovAndInjEnvAndResField.FlCov},
	there exists a morphism $g$ such that the diagram
	\begin{equation*}
	\begin{tikzcd}
			G\ar[r]\ar[d,"g"'] &G\otimes_{R}k\ar[r]\ar[d,"\overline{g}"'] & \displaystyle G\otimes_{A}(A/J)\ar[d,"\overline{h}"] \\
			F\ar[r]&F\otimes_{R}k\ar[r] & \displaystyle F\otimes_{A}(A/J)
		\end{tikzcd}
	\end{equation*}
	commutes, where unadorned morphisms are canonical ones.
	 The composition of the morphisms in the first row together with $\overline{h}$ is $h$, so the diagram shows that $h$ factors through the second row, which is $f$. Hence $f$ is a flat precover.
	
	Next, for every $s\in\End_{A}(F)$ such that $fs=f$, we have a commutative diagram
	\begin{equation*}
		\begin{tikzcd}[row sep=tiny]
			F\ar[dd,"s"']\ar[r] & F\otimes_{R}k\ar[dd,"s\otimes_{R}k"']\ar[dr] & \\
			& & F\otimes_{A}(A/J)\rlap{,} \\
			F\ar[r] & F\otimes_{R}k\ar[ur] & 
		\end{tikzcd}
	\end{equation*}
	where unadorned morphisms are canonical ones.
	Here $s\otimes_{R}k$ is an isomorphism by \cref{NilpIdealAndProjCov}\cref{NilpIdealAndProjCov.Right} as $J/\km A=\rad(A\otimes_{R}k)$ (see \cref{JacobsonRadicalAlgebra}) is nilpotent. Therefore \cref{FlCovAndInjEnvAndResField}\cref{FlCovAndInjEnvAndResField.FlCov} implies that $s$ is an isomorphism. This concludes that $f$ is a flat cover.
	
	\cref{FlCovAndInjEnvAndTop.Inj}: The injection $\Hom_A(A/J,I)\to I$ is clearly an injective preenvelope. The rest of the proof is parallel to \cref{FlCovAndInjEnvAndTop.Fl}. Use \cref{NilpIdealAndProjCov}\cref{NilpIdealAndProjCov.Left} and \cref{FlCovAndInjEnvAndResField}\cref{FlCovAndInjEnvAndResField.InjEnv} instead.
\end{proof}

\begin{lemma}\label{FlatCoverIdeal}
Let $\ka$ be an ideal of $R$ and let $M$ be a right $A$-module. 
\begin{enumerate}
\item \label{FlatCoverIdeal.Complete} If $M$ is an $\ka$-complete right $A$-module, then its flat cover $F_A(M)$ is $\ka$-complete.
\item \label{FlatCoverIdeal.Torsion}If $M$ is an $\ka$-torsion right $A$-module, then its injective envelope $E_A(M)$ is $\ka$-torsion.
\end{enumerate}

\end{lemma}

\begin{proof}
\cref{FlatCoverIdeal.Complete}: Let $f\colon F_A(M)\to M$ be the flat cover. Since $M$ is $\ka$-complete, we may identify $M$ with $\varLambda^\ka M$. Then $f$ is factorized as the composition of the completion map $F_A(M)\to \varLambda^\ka F_A(M)$ and $\varLambda^\ka f\colon \varLambda^\ka F_A(M) \to M$. By \cref{FlCompIsFl}, $\varLambda^\ka F_A(M)$ is flat. Since $f$ is a flat cover,
there exists a morphism $g$ such that the diagram
\begin{equation*}
	\begin{tikzcd}
		F_{A}(M)\ar[r]\ar[dr,"f"'] & \varLambda^\ka F_A(M)\ar[d,"\varLambda^\ka f"']\ar[r,"g"] & F_{A}(M)\ar[dl,"f"] \\
		& M &
	\end{tikzcd}
\end{equation*}
commutes. The right minimality of $f$ implies that $g$ is a split epimorphism, so
$F_A(M)$ is a direct summand of $\varLambda^\ka F_A(M)$. Since $\varLambda^\ka F_A(M)$ is $\ka$-complete (\cref{CompletionIdempotent}), the direct summand $F_A(M)$ is also $\ka$-complete.

\cref{FlatCoverIdeal.Torsion}: Let $g\colon M\to E_A(M)$ be the injective envelope. Since $A$ is right noetherian, $E_A(M)$ decomposes as a direct sum of copies of $I_A(P)$ for various $P\in \Spec A$; see \cref{IndecomposableDecomposition}. The $\ka$-torsion $A$-submodule $\varGamma_\ka E_A(M)$ of $E_A(M)$ is injective by \cref{CompAndColocOfInj}\cref{CompAndColocOfInj.Comp}, and the induced map $\varGamma_\ka g\colon M\to \varGamma_\ka E_A(M)$ is a monomorphism as $\varGamma_\ka$ is left exact (\cref{TostionAdjoint}). Thus we have $\varGamma_\ka E_A(M)=E_A(M)$.
\end{proof}

\begin{remark}\label{SemisimpleModules}
Let $\p\in \Spec A$. Recall that $A_\p/\rad A_\p$ is a semisimple ring (see \cref{TopOfFiber}). Hence every module over this ring is a direct sum of simple modules, and each simple module is isomorphic to $S_A(P)$ for some $P\in \Spec A$ with $P\cap R=\p$. Moreover, the category $\Mod (A_\p/\rad A_\p)$ is naturally equivalent to the subcategory of $\Mod A_\p$ (or $\Mod A$) formed by semisimple $A_\p$-modules (see \cref{LocOfNoethAlg}\cref{LocOfNoethAlg.Max} and \cref{SimpleOverNoethAlg}).
\end{remark}

With this remark, we obtain the following result:

\begin{proposition}\label{TopOfFlCovAndSocOfInjEnv}
	Let $\p\in \Spec R$ and let $M$ be a semisimple right $A_\p$-module. 
	
	\begin{enumerate}
	\item\label{TopOfFlCovAndSocOfInjEnv.Flat} Let $f\colon F_A(M)\to M$ be a flat cover in $\Mod A$. Then $F_A(M)$ is $\p$-local and $\p$-complete. Moreover, the morphism $f$ induces an isomorphism 
	\begin{equation*}F_A(M)\otimes_{A}(A_\p/\rad A_\p)\isoto M.
	\end{equation*}
	 	
	\item\label{TopOfFlCovAndSocOfInjEnv.Injective} Let $g\colon M\to E_A(M)$ be an injective envelope in $\Mod A$. Then $E_A(M)$ is $\p$-local and $\p$-torsion. Moreover, the morphism $g$ induces an isomorphism 
	\begin{equation*}
	M\isoto \Hom_A(A_\p/\rad A_\p, E_A(M)).
	\end{equation*}
	\end{enumerate}
\end{proposition}

\begin{proof}
	We only prove \cref{TopOfFlCovAndSocOfInjEnv.Flat} because the dual argument proves \cref{TopOfFlCovAndSocOfInjEnv.Injective}.
	
Note that the semisimple $A_\p$-module $M$ is $\p$-complete as $\p^n M=0$ for every $n\geq 1$ (see \cref{SemisimpleModules}). Hence $F_{A}(M)$ is $\p$-local $\p$-complete by \cref{ClosureEnvCovLocal} and \cref{FlatCoverIdeal}\cref{FlatCoverIdeal.Complete}.
	We have a commutative diagram
	\begin{equation*}
		\begin{tikzcd}
			F_{A}(M)\ar[d,"u"']\ar[r,"f"] & M\ar[d,"\wr"] \\
			\displaystyle F_{A}(M)\otimes_{A}(A_\p/\rad A_\p)\ar[r,"h"'] & \displaystyle M\otimes_{A}(A_\p/\rad A_\p)\rlap{,}
		\end{tikzcd}
	\end{equation*}
	where the vertical morphisms are canonical and $h:=f\otimes_{A}(A_\p/\rad A_\p)$. By \cref{FlCovAndInjEnvAndTop}\cref{FlCovAndInjEnvAndTop.Fl}, $u$ is a flat cover. Since $h$ is an epimorphism between semisimple $A_{\p}$-modules, it is a split epimorphism. Thus the flat cover $F_{A}(\Ker h)\to\Ker h$ is a direct summand of the flat cover $u$. Since $F_{A}(\Ker h)$ is in the kernel of $f$, it should be zero by the right minimality of $f$. Therefore $\Ker h=0$.
\end{proof}

\begin{theorem}\label{BijBetweenFlAndInjAndSemisimple}
For every $\p\in \Spec R$, we have the following one-to-one correspondences:
	\begin{equation*}
	\begin{tikzcd}
			\setwithtext{isoclasses of $\p$-local $\p$-complete flat right $A$-modules}
			\ar[d, shift right=1.5ex, "-\otimes_{A}(A_\p/\rad A_\p)"']\ar[d,phantom,"\wr"]\\
			\setwithtext{isoclasses of semisimple right $A_\p$-modules}\ar[u, shift right=1.5ex, "F_A(-)"']
			\ar[d, shift right=1.5ex, "E_A(-)"']\ar[d,phantom,"\wr"]\\
			\setwithtext{isoclasses of $\p$-local $\p$-torsion injective right $A$-modules}\ar[u, shift right=1.5ex,  "{\Hom_A(A_\p/\rad A_\p,-)}"']\rlap{.}
		\end{tikzcd}
	\end{equation*}
\end{theorem}
\begin{proof}
This follows from \cref{FlCovAndInjEnvAndTop,SemisimpleModules,TopOfFlCovAndSocOfInjEnv}.
\end{proof}

Let $\p\in \Spec R$. As we observed in \cref{SemisimpleModules}, every semisimple right $A_\p$-module $M$ decomposes as
\begin{equation}\label{SemisimpleDecomposition}
M\cong \bigoplus_{\substack{P\in\Spec A\\P\cap R=\p}}S_{A}(P)^{(B_P)}
\end{equation}
for some family of sets $\set{B_{P}}_{P}$, where $\setwithcondition{P\in\Spec A}{P\cap R=\p}$ is a finite set (\cref{SpecOfNoethAlg}).
\cref{DualOfInjEnv} (\resp \cref{DualOfFlCov}) below shows that a flat cover (\resp injective envelope) of $S_{A}(P)^{(B_P)}$ in $\Mod A$ can be obtained by applying a variant of Matlis dual to an injective envelope (\resp flat cover) of $S_{A^\op}(P)$ in $\Mod A^\op$.

\begin{proposition}\label{DualOfInjEnv}
	Let $P\in\Spec A$ and $\p:=P\cap R$. For every set $B$, the injective envelope $S_{A^{\op}}(P)\to I_{A^{\op}}(P)$ induces a flat cover
	\begin{equation}\label{eq.FlCovOfSimple}
		\Hom_{R}(I_{A^{\op}}(P),E_{R}(R/\p)^{(B)})\to\Hom_{R}(S_{A^{\op}}(P),E_{R}(R/\p)^{(B)})\cong S_{A}(P)^{(B)}
	\end{equation}
	in $\Mod A$.
\end{proposition}

\begin{proof}
	We first recall that each $\Hom_{R}$ in \cref{eq.FlCovOfSimple} can be replaced by $\Hom_{R_{\p}}$; see \cref{LocAdj} and \cref{IndecomposableInjectiveLocal}. 
Moreover, the $\p$-local right $A$-module $\Hom_{R}(I_{A^{\op}}(P),E_{R}(R/\p)^{(B)})$ is $\p$-complete and flat as well; see \cref{FlCotIffDSummandOfProdOfDual} and \cref{DualOfInjIsLocComp}.

Next, notice that the functor $\Hom_{R_\p}(S_{A^{\op}}(P),-)\colon\Mod R_{\p}\to\Mod A$ commutes with arbitrary direct sums, because $S_{A^\op}(P)$ is finitely generated over $R_\p$. So the isomorphism in \cref{eq.FlCovOfSimple} follows from \cref{SimpleMatlisDuality}.
Since $S_A(P)\cong S_A(P)\otimes_A(A_\p/\rad A_\p)$ by construction and we have \cref{FlCovAndInjEnvAndTop}\cref{FlCovAndInjEnvAndTop.Fl}, it only remains to show that \cref{eq.FlCovOfSimple} becomes an isomorphism upon application of $-\otimes_A(A_\p/\rad A_\p)$.

To this end, we remark that there is a canonical isomorphism
	\begin{equation}\label{CanonicalIsom}
		\Hom_{R}(-,E_{R}(R/\p)^{(B)})\otimes_{A}(A_\p/\rad A_\p)\cong \Hom_{R}(\Hom_{A^{\op}}(A_\p/\rad A_\p,-),E_{R}(R/\p)^{(B)})
	\end{equation}
	of functors $\Mod {A_\p^\op}\to \Mod A$ (see \cite[Theorem~3.2.11]{MR1753146}), because $\Hom_R$, $\Hom_{A^{\op}}$, and $\otimes_A$ can be replaced by $\Hom_{R_\p}$, $\Hom_{A_\p^{\op}}$, and $\otimes_{A_\p}$, respectively.
	Under \cref{CanonicalIsom} and the natural isomorphism $\Hom_{{A^\op}}(A_\p/\rad A_\p,S_{A^{\op}}(P))\isoto S_{A^{\op}}(P)$, application of $-\otimes_A(A_\p/\rad A_\p)$ to the first map in \cref{eq.FlCovOfSimple} yields a morphism 
	\begin{equation*}
	\Hom_{R}(\Hom_{A^{\op}}(A_\p/\rad A_\p,I_{A^{\op}}(P)),E_{R}(R/\p)^{(B)})\to \Hom_{R}(S_{A^{\op}}(P),E_{R}(R/\p)^{(B)})
	\end{equation*}
	of right $A$-modules.
	This is an isomorphism since it is induced by the isomorphism 
	\begin{equation*}
	S_{A^{\op}}(P)\isoto \Hom_{A^{\op}}(A_\p /\rad A_\p,I_{A^{\op}}(P))
	\end{equation*}
	obtained by applying \cref{TopOfFlCovAndSocOfInjEnv}\cref{TopOfFlCovAndSocOfInjEnv.Injective} to the injective envelope $S_{A^{\op}}(P)\to I_{A^{\op}}(P)$.
\end{proof}

\cref{DualOfInjEnv} leads us to the following definition, which is essential for the main results of this paper:

\begin{definition}\label{def.IndecFlCot}Let $P\in \Spec A$ and $\p:=P\cap R$.
	We define
	\begin{equation*}
		T_{A}(P):=\Hom_{R}(I_{A^{\op}}(P),E_{R}(R/\p)),
	\end{equation*}
	which is a flat cover of $S_{A}(P)$ in $\Mod A$ by \cref{DualOfInjEnv}, that is, 
	\begin{equation*}
	T_{A}(P)=F_{A}(S_{A}(P))
	\end{equation*}
	as isoclasses of right $A$-modules. 
	As we recalled in the proof of \cref{DualOfInjEnv}, $T_{A}(P)$ is $\p$-local and $\p$-complete. Moreover, by \cref{InjDualProperties}\cref{InjDualProperties.Cot}, $T_{A}(P)$ is pure-injective, hence cotorsion.
\end{definition}

\begin{remark}\label{CompletionNecessity}
In \cref{eq.FlCovOfSimple} and \cref{def.IndecFlCot}, each $\Hom_{R}$ can also be replaced by $\Hom_{\widehat{R_\p}}$; see \cref{IndecomposableInjectiveLocal}, \cref{TorsionEquivalence} (applied to $A=R$), and \cref{ArtinianInjective}. On the other hand, the second $\Hom_{\widehat{R_\p}}$ in \cref{eq.InjEnvOfSimple} below cannot be replaced either by $\Hom_{R}$ or by $\Hom_{R_\p}$. To see this, consider the case where $(R,\km,k)$ is local, $A=R$, $P=\km$, and $B$ is a set consisting of one element. Then $T_R(\km)\cong \widehat{R}$, and $E_{R}(k)$ naturally becomes an $\widehat{R}$-module (see \cref{subsec.MatlisDual}), so we have a canonical injection $f\colon E_{R}(k)= \Hom_{\widehat{R}}(\widehat{R}, E_{R}(k))\to \Hom_{R}(\widehat{R}, E_{R}(k))$. This injection is not surjective as far as $R$ is not $\km$-complete, because the surjection $g\colon\Hom_R(\widehat{R},E_R(k))\to \Hom_R(R, E_R(k))=E_R(k)$ induced by the completion map $R\to \widehat{R}$ is not injective, and $gf$ is the identity map.

It should also be noticed that the last isomorphism of \cref{eq.InjEnvOfSimple} shows that $T_{A}(P)$ is indecomposable.
\end{remark}

\begin{proposition}\label{DualOfFlCov}
	Let $P\in\Spec A$ and $\p:=P\cap R$. For every set $B$, the flat cover $T_{A}(P)\to S_{A}(P)$ induces an injective envelope
	\begin{equation}\label{eq.InjEnvOfSimple}
		S_{A^{\op}}(P)^{(B)}\cong\Hom_{\widehat{R_{\p}}}(S_{A}(P),E_{R}(R/\p)^{(B)})\to\Hom_{\widehat{R_{\p}}}(T_{A}(P),E_{R}(R/\p)^{(B)})\cong I_{A^{\op}}(P)^{(B)}
	\end{equation}
	in $\Mod {A^\op}$.
\end{proposition}

\begin{proof}
	The first isomorphism in \cref{eq.InjEnvOfSimple} follows from \cref{DualOfInjEnv,CompletionNecessity}.
	
	Let $(-)^{*}:=\Hom_{\widehat{R_\p}}(-,E_{R}(R/\p))$. Since $I_{A^{\op}}(P)$ is an artinian left $\widehat{A_\p}$-module (see \cref{IndecomposableInjectiveLocal} and \cref{ArtinianInjective}), $T_{A}(P)=(I_{A^{\op}}(P))^*$ is a finitely generated right $\widehat{A_\p}$-module by \cref{MatlisDualOverNoetherianAlgebra}, and thus we have a canonical isomorphism 
	\begin{equation*}
	(T_{A}(P)^{*})^{(B)}\isoto \Hom_{\widehat{R_{\p}}}(T_{A}(P),E_{R}(R/\p)^{(B)}).
	\end{equation*}
	\cref{MatlisDualOverNoetherianAlgebra} also yields a canonical isomorphism 
	\begin{equation*}I_{A^{\op}}(P)^{(B)}\isoto (I_{A^{\op}}(P)^{**})^{(B)}= (T_{A}(P)^{*})^{(B)}
	\end{equation*}
	of left $A$-modules. Therefore the last isomorphism in \cref{eq.InjEnvOfSimple} holds.
	
	Recall that $I_{A^{\op}}(P)^{(B)}$ is a $\p$-local $\p$-torsion injective left $A$-module (see \cref{ArtinianInjective}). Since we have \cref{FlCovAndInjEnvAndTop}\cref{FlCovAndInjEnvAndTop.Inj} and $\Hom_{A^{\op}}(A_\p /\rad A_\p,S_{A^{\op}}(P)^{(B)})\cong S_{A^{\op}}(P)^{(B)}$, it suffices to show that the morphism in $\cref{eq.InjEnvOfSimple}$ becomes an isomorphism upon application of $\Hom_{A^\op}(A_\p /\rad A_\p,-)$.
	By the tensor-hom adjunction, we have a canonical isomorphism
	\begin{equation*}
		\Hom_{A^\op}(A_\p /\rad A_\p,\Hom_{\widehat{R_{\p}}}(-,E_{R}(R/\p)^{(B)}))\cong\Hom_{\widehat{R_{\p}}}(-\otimes_A  (A_\p /\rad A_\p),E_{R}(R/\p)^{(B)})
	\end{equation*}
	of functors $\Mod \widehat{A_{\p}}\to \Mod A^\op$. In addition, the flat cover $T_A(P)\to S_A(P)$ induces an isomorphism $T_{A}(P)\otimes_A  (A_\p /\rad A_\p)\isoto S_A(P)\otimes_A  (A_\p /\rad A_\p)\cong S_{A}(P)$ by \cref{TopOfFlCovAndSocOfInjEnv}\cref{TopOfFlCovAndSocOfInjEnv.Flat}. Therefore application of $\Hom_{A^\op}(A_\p /\rad A_\p,-)$ makes the morphism in \cref{eq.InjEnvOfSimple} an isomorphism.
\end{proof}

In general, the character dual of a flat precover over an arbitrary ring is an injective preenvelope; see \cref{InjDualProperties}\cref{InjDualProperties.FlBecomesInj} and \cref{NonnoethereianCase}.
Conversely, if the ring is right coherent, then the character dual of an injective preenvelope of a right module is a flat precover; see \cite[Proposition~5.3.5]{MR1753146}.

\section{Descriptions of local complete flat modules}
\label{DescripOfLocCompFlMod}

In this section, we give various descriptions of local complete flat right modules over a Noether algebra.
We first look back on some classical facts for a commutative noetherian ring $R$.

Let $\p\in \Spec R$. Gruson and Raynaud \cite[Part~II, Proposition~2.4.3.1]{MR308104} showed that every $\p$-local $\p$-complete flat $R$-module is isomorphic to the $\p$-adic completion of some free $R_{\p}$-module. More precisely, it is shown that, given a flat $R$-module $F$, there is an isomorphism $\widehat{F_\p}\cong (R_\p^{(B)})^\wedge_\p$, where $B:=\dim_{\kappa(\p)} F\otimes_R\kappa(\p)$ (see also \cite[Lemma~6.7.4]{MR1753146}). It is also shown that $(R_\p^{(B)})^\wedge_\p$ is a flat $R$-module (\cite[Part~II, (2.4.2)]{MR308104}).
Furthermore, Enochs pointed out in \cite[p.~181, Example]{MR754698} that $(R_\p^{(B)})^\wedge_\p$ is isomorphic to $\Hom_R(E(R/\p), E(R/\p)^{(B)})$ (see also \cite[Theorem~3.4.1(7)]{MR1753146}).
In particular,  $(R_\p^{(B)})^\wedge_\p$ is a flat cotorsion $R$-module (\cref{InjDualProperties}\cref{InjDualProperties.InjBecomesFlCot}).
It then follows that the following conditions are equivalent for an arbitrary $R$-module $M$: 
\begin{enumerate}
\item\label{HistoryLocalComplete.Flat} $M$ is a $\p$-local $\p$-complete flat $R$-module.
\item\label{HistoryLocalComplete.FlCot} $M$ is a $\p$-local $\p$-complete flat cotorsion $R$-module.
\item\label{HistoryLocalComplete.Free} $M$ is isomorphic to the $\p$-adic completion of a free $R_\p$-module.
\end{enumerate}
The term ``free $R_\p$-module'' in \cref{HistoryLocalComplete.Free} can be replaced by ``projective $R_\p$-module'', ``free $\widehat{R_\p}$-module'', or ``projective $\widehat{R_\p}$-module'' because $R_{\p}$ and $\widehat{R_\p}$ are local rings and $(R_\p^{(B)})^\wedge_\p \cong (\widehat{R_\p}^{(B)})^\wedge_\p$ by \cref{CompTensorRed}.

This section is devoted to generalizing these classical facts to an arbitrary Noether $R$-algebra $A$.
We start with the following lemma, which slightly refines \cite[Proposition~2.5.5]{MR1898632} and is known when $A$ is commutative  (see \cite[Theorem~1.1]{MR2471985}):

\begin{lemma}\label{DecompOfInjDualOfAlg}
	For every $\p\in\Spec R$, there is an isomorphism
	\begin{equation*}
		\Hom_{R}(A,E_{R}(R/\p))\cong\bigoplus_{\substack{P\in\Spec A\\P\cap R=\p}}I_{A^{\op}}(P)^{n_{P}}
	\end{equation*}
	of left $A$-modules.
\end{lemma}

\begin{proof}
	By \cref{InjDualProperties}\cref{InjDualProperties.FlBecomesInj}, $\Hom_{R}(A,E_{R}(R/\p))$ is an injective left $A$-module.
As $E_R(R/\p)\cong E_{R_\p}(\kappa(\p))$ by \cref{IndecomposableInjectiveLocal}, the functor $\Hom_{R}(-,E_{R}(R/\p))$ sends finitely generated right $A$-modules to $\p$-local $\p$-torsion left $A$-modules; see \cref{LocAdj} and \cref{MatlisDual1}. 
Thus, \cref{FlCovAndInjEnvAndTop}\cref{FlCovAndInjEnvAndTop.Inj} applied to $I:=\Hom_{R}(A,E_{R}(R/\p))$ implies that the canonical morphism 
$\Hom_{A^{\op}}(A_\p /\rad A_\p, I) \to I$ is an injective envelope in $\Mod A^{\op}$.  Now we have
	\begin{align*}
		\Hom_{A^{\op}}(A_\p /\rad A_\p, I)
		&\cong\Hom_{R}(A\otimes_{A}(A_\p /\rad A_\p),E_{R}(R/\p))\\
		&\cong\Hom_{R}(\bigoplus_{\substack{P\in\Spec A\\P\cap R=\p}}S_{A}(P)^{n_{P}},E_{R}(R/\p))\\
		&\cong\bigoplus_{\substack{P\in\Spec A\\P\cap R=\p}}S_{A^{\op}}(P)^{n_{P}},
	\end{align*}
	where the second isomorphism follows from \cref{TopOfFiber}, and the third follows from \cref{SpecOfNoethAlg,SimpleMatlisDuality}. Since each $I_{A^{\op}}(P)$ is the injective envelope of $S_{A^{\op}}(P)$, we obtain the desired isomorphism.
\end{proof}

\begin{proposition}\label{DecompOfCompOfAlg}
	For every $\p\in\Spec R$, there is an isomorphism
	\begin{equation*}
		\widehat{A_{\p}}\cong\bigoplus_{\substack{P\in\Spec A\\P\cap R=\p}}T_{A}(P)^{n_{P}}
	\end{equation*}
	of right $A$-modules.
\end{proposition}

\begin{proof}
By \cref{LocAdj}, \cref{IndecomposableInjectiveLocal}, and \cref{DoubleMatlisDual1}, there is a canonical isomorphism 
	\begin{equation}\label{DoubleMatlisDualOfA}
		\widehat{A_{\p}}\isoto \Hom_{R}(\Hom_{R}(A,E_{R}(R/\p)),E_{R}(R/\p))
	\end{equation}
	of right $\widehat{A_\p}$-modules.
	Thus the result follows from \cref{SpecOfNoethAlg} and \cref{DecompOfInjDualOfAlg}.
\end{proof}

\begin{remark}\label{HomomorphismBetweenLocalCompleteModules}
By \cref{LocAdj,CompleteEquivalence}, all $A$-homomorphism between $\p$-local $\p$-complete right $A$-modules are $\widehat{A_{\p}}$-homomorphisms. The isomorphism in \cref{DecompOfCompOfAlg} is therefore an isomorphism of right $\widehat{A_{\p}}$-modules. This implies that each $T_{A}(P)$ is a projective right $\widehat{A_\p}$-module.

Similarly, by \cref{LocAdj,TorsionEquivalence}, all $A$-homomorphism between $\p$-local $\p$-torsion right $A$-modules are also $\widehat{A_{\p}}$-homomorphisms. So the isomorphism in \cref{DecompOfInjDualOfAlg} is an isomorphism of left $\widehat{A_{\p}}$-modules.
\end{remark}

We will observe in \cref{CompOfDSumIsPInjEnv} that a direct sum of infinite copies of $T_A(P)$ is not necessarily cotorsion, but its $\p$-adic completion is cotorsion by the next result.

\begin{proposition}\label{DualOfInj}
	Let $P\in\Spec A$ and $\p:=P\cap R$. For every set $B$, there exists a canonical isomorphism
	\begin{equation*}
		(T_{A}(P)^{(B)})_{\p}^{\wedge}\isoto\Hom_{R}(I_{A^{\op}}(P),E_{R}(R/\p)^{(B)}).
	\end{equation*}
	of $A$-modules. In particular, $(T_{A}(P)^{(B)})_{\p}^{\wedge}$ is flat and pure-injective, that is, flat cotorsion. 
\end{proposition}

\begin{proof}
Let $S_{A^\op}(P) \to I_{A^{\op}}(P)$ be the injective envelope.
Applying $\Hom_R(-,E_R(R/\p))$ to this, we obtain the flat cover $\Hom_R(I_{A^{\op}}(P),E_R(R/\p)) \to \Hom_R(S_{A^\op}(P),E_R(R/\p))$ by \cref{DualOfInjEnv}. 
Taking the the direct sum of $B$-indexed copies of the flat cover, we obtain the first row of the following diagram: 
\begin{equation*}
\begin{tikzcd}
 T_A(P)^{(B)} \ar[r,equal]\arrow[d] &\Hom_{R}(I_{A^{\op}}(P),E_{R}(R/\p))^{(B)}  \ar[d]\ar[r]&\Hom_R(S_{A^\op}(P),E_R(R/\p))^{(B)}\ar[d,"\wr"]\\
(T_A(P)^{(B)})^\wedge_\p\ar[r]&\Hom_{R}(I_{A^{\op}}(P),E_{R}(R/\p)^{(B)})\ar[r]&\Hom_{R}(S_{A^{\op}}(P),E_{R}(R/\p)^{(B)})
\end{tikzcd}
\end{equation*}
The vertical morphisms are canonical ones, and the third is an isomorphism by the proof of \cref{DualOfInjEnv}.
The first morphism in the second row is the unique morphism making the left square commutative; this exists since $\Hom_{R}(I_{A^{\op}}(P),E_{R}(R/\p)^{(B)})$ is $\p$-complete; see \cref{DualOfInjIsLocComp} and \cref{CompAdj}. 
The second morphism in the second row is the one induced by the injective envelope $S_{A^\op}(P) \to I_{A^{\op}}(P)$, so it is a flat cover by \cref{DualOfInjEnv}. Moreover, the right square is commutative as well.

If we apply $-\otimes_A (A_\p /\rad A_\p)$ to the above commutative diagram, then the second morphism in each row becomes an isomorphism by the proof of \cref{DualOfInjEnv}, and the first vertical morphism becomes an isomorphism by \cref{CompTensorRed} and \cref{JacobsonRadicalAlgebra}, so the other morphisms in the diagram are also isomorphisms, where the third vertical morphism remains the same morphism.

Therefore, it follows from \cref{FlCovAndInjEnvAndTop}\cref{FlCovAndInjEnvAndTop.Fl} that the second row of the above diagram is a flat cover because $(T_A(P)^{(B)})^\wedge_\p$ is a $\p$-local $\p$-complete flat right $A$-module; see \cref{def.IndecFlCot,FlCompIsFl,CompletionIdempotent}.
\end{proof}

\begin{remark}\label{EnochsIsomExtended}
Let $\p\in \Spec R$ and let $B$ be a set. We can recover the known isomorphism
\begin{equation}\label{eq.hom.inj.env.dsum}
	(R_\p^{(B)})^\wedge_\p\cong \Hom_R(E(R/\p), E(R/\p)^{(B)})
\end{equation}
from \cref{DecompOfCompOfAlg,DualOfInj}.
Indeed, \cref{DecompOfCompOfAlg} applied to $A=R$ simply identifies $\widehat{R_\p}$ with $T_R(\p)=\Hom_{R}(E_{R}(R/\p),E_{R}(R/\p))$, and then \cref{DualOfInj} gives an isomorphism 
 \begin{equation*}
 (\widehat{R_\p}^{(B)})^\wedge_\p\isoto \Hom_{R}(E_{R}(R/\p),E_{R}(R/\p)^{(B)}),
 \end{equation*}
 where the left-hand side coincides with $(R_\p^{(B)})^\wedge_\p$ by \cref{CompTensorRed}.

This isomorphism can be generalized to $A$. Applying the functor $-\otimes_RA$ to \cref{eq.hom.inj.env.dsum} and using \cref{FlTensorCompAsTensor}, we obtain an isomorphism
 \begin{equation*}
 (A_\p^{(B)})^\wedge_\p\isoto \Hom_{R}(\Hom_R(A,E_{R}(R/\p)),E_{R}(R/\p)^{(B)}),
 \end{equation*}
as we deduced \cref{DoubleMatlisDual1}.
In particular, it follows that $(A_\p^{(B)})^\wedge_\p$ is a $\p$-local $\p$-complete flat cotorsion right $A$-module; see \cref{InjDualProperties,CompletionIdempotent}.
 \end{remark}

\begin{theorem}\label{DecompFlCotAtEachBasePrime}
	Let $\p\in\Spec R$. For a right $A$-module $M$, the following are equivalent:
	\begin{enumerate}
		\item\label{DecompFlCotAtEachBasePrime.LocCompFl} $M$ is a $\p$-local $\p$-complete flat right $A$-module.
		\item\label{DecompFlCotAtEachBasePrime.LocCompFlCot} $M$ is a $\p$-local $\p$-complete flat cotorsion right $A$-module.
		\item\label{DecompFlCotAtEachBasePrime.DSummand} $M$ is a direct summand of $(A_{\p}^{(B)})_{\p}^{\wedge}$ for some set $B$.
		\item\label{DecompFlCotAtEachBasePrime.DSumOfIndec} $M$ is isomorphic to
		\begin{equation*}
			\bigoplus_{\substack{P\in\Spec A\\P\cap R=\p}}(T_{A}(P)^{(B_{P})})_{\p}^{\wedge}
		\end{equation*}
		for some family of sets $\set{B_{P}}_{P}$.

	\end{enumerate}
	The cardinality of each $B_{P}$ in \cref{DecompFlCotAtEachBasePrime.DSumOfIndec} is uniquely determined by $M$.
\end{theorem}

\begin{proof}
	Assume \cref{DecompFlCotAtEachBasePrime.LocCompFl}. By \cref{FlCovAndInjEnvAndTop}\cref{FlCovAndInjEnvAndTop.Fl}, 
	 $M$ is a flat cover of $M\otimes_{A}(A_\p/\rad A_\p)$ in $\Mod A$, where 
	\begin{equation*}
		M\otimes_{A}(A_\p/\rad A_\p)\cong \bigoplus_{\substack{P\in\Spec A\\P\cap R=\p}}S(P)^{(B_{P})}
	\end{equation*}
	for a family of sets $\set{B_{P}}_{P}$; see \cref{SemisimpleModules} and \cref{SemisimpleDecomposition}. 
Then \cref{DecompFlCotAtEachBasePrime.DSumOfIndec} follows from the above decomposition and \cref{DualOfInjEnv,DualOfInj}, along with \cref{SpecOfNoethAlg} and the elementary fact that a finite direct sum of flat covers is a flat cover (\cite[Theorem~1.2.10]{MR1438789}). 

Assume \cref{DecompFlCotAtEachBasePrime.DSumOfIndec}.
We first show the uniqueness of each $B_P$.
By \cref{DualOfInjEnv,DualOfInj},
$(T_A(P)^{(B_P)})^\wedge_\p$ is the flat cover of $S_A(P)^{(B_P)}$. 
By \cref{BijBetweenFlAndInjAndSemisimple}, 
we have an isomorphism $(T_A(P)^{(B_P)})^\wedge_\p\otimes_A (A_\p/\rad A_\p) \cong S_A(P)^{(B_P)}$.
It then follows that 
\begin{equation*}
M \otimes_A (A_\p/\rad A_\p)\cong
\bigg(\bigoplus_{\substack{P\in\Spec A\\P\cap R=\p}} (T_A(P)^{(B_P)})^\wedge_\p\bigg) \otimes_A (A_\p/\rad A_\p)
\cong \bigoplus_{\substack{P\in\Spec A\\P\cap R=\p}} S_A(P)^{(B_P)}.
\end{equation*}
Therefore, the cardinality of each $B_{P}$ is uniquely determined by $M$, due to the Krull-Remak-Schmidt-Azumaya theorem; see \cite[Theorem~E.1.24]{MR2530988} for example.

Let us next show that \cref{DecompFlCotAtEachBasePrime.DSumOfIndec} implies \cref{DecompFlCotAtEachBasePrime.DSummand}.
We know from \cref{DecompOfCompOfAlg} that $T_{A}(P)^{(B_{P})}$ is a direct summand of $\widehat{A_\p}^{(B_{P})}$,
 so $(T_{A}(P)^{(B_{P})})^\wedge_\p$ is a direct summand of $(\widehat{A_\p}^{(B_{P})})^\wedge_\p$, which is isomorphic to $(A_\p^{(B_{P})})^\wedge_\p$ by \cref{CompTensorRed}.
Let $B$ be the disjoint union of all $B_{P}$. Then
\begin{equation*}(A_{\p}^{(B)})_{\p}^{\wedge}\cong \bigoplus_{\substack{P\in\Spec A\\P\cap R=\p}}(A_{\p}^{(B_P)})_{\p}^{\wedge},\end{equation*}
so \cref{DecompFlCotAtEachBasePrime.DSummand} holds.
	
The implication \cref{DecompFlCotAtEachBasePrime.DSummand} $\Rightarrow$ \cref{DecompFlCotAtEachBasePrime.LocCompFlCot} is clear in view of the last sentence of \cref{EnochsIsomExtended}. The remaining implication \cref{DecompFlCotAtEachBasePrime.LocCompFlCot} $\Rightarrow$ \cref{DecompFlCotAtEachBasePrime.LocCompFl} is trivial.
\end{proof}

\begin{remark}
\label{LocTorsFlModIsPInj}
By \cref{DecompOfCompOfAlg}, \cref{HomomorphismBetweenLocalCompleteModules,EnochsIsomExtended}, and \cref{DecompFlCotAtEachBasePrime}, the following conditions are equivalent for a right $A$-module $M$:
\begin{enumerate}
\item\label{ExtendedLocalComplete.Flat} $M$ is a $\p$-local $\p$-complete flat right $A$-module.
\item\label{ExtendedLocalComplete.FlCot} $M$ is a $\p$-local $\p$-complete flat cotorsion right $A$-modules.
\item\label{ExtendedLocalComplete.Free} $M$ is isomorphic to the $\p$-adic completion of a projective $\widehat{A_\p}$-module.
\end{enumerate}
If this is the case, then the projective $\widehat{A_\p}$-module in \cref{ExtendedLocalComplete.Free} can be taken as a direct sum of indecomposable projective $\widehat{A_\p}$-modules.

Let $F$ be a flat right $A$-module. Then its localization $F_\p$ is also a flat right $A$-module, so its $\p$-adic completion $\widehat{F_\p}$ is a $\p$-local $\p$-complete flat right $A$-module (see \cref{FlCompIsFl,CompletionIdempotent}).
Thus \cref{DecompFlCotAtEachBasePrime} yields an isomorphism 
$\widehat{F_\p}\cong \bigoplus_{\substack{P\in\Spec A\\P\cap R=\p}}(T_{A}(P)^{(B_P)})_{\p}^{\wedge}
$,
and the proof of the theorem shows that the index sets $B_{P}$ are determined by a decomposition
\begin{equation*}
	F\otimes_A(A_\p/ \rad A_\p)\cong\widehat{F_\p}\otimes_A(A_\p/ \rad A_\p)\cong\bigoplus_{\substack{P\in\Spec A\\P\cap R=\p}}S(P)^{(B_{P})},
\end{equation*}
where the first isomorphism follows from \cref{JacobsonRadicalAlgebra} and $\widehat{F_\p}\otimes_{R}\kappa(\p)\cong F\otimes_{R}\kappa(\p)$ (see \cref{CompTensorRed}). If $A=R$, then the left-most side is $F\otimes_{R}\kappa(\p)$. Therefore, all the classical facts mentioned at the beginning of this section have been generalized to Noether algebras.
\end{remark}

\begin{remark}
	Contrary to the classical case, the term ``projective $\widehat{A_\p}$-module'' in \cref{LocTorsFlModIsPInj}\cref{ExtendedLocalComplete.Free} cannot be replaced either by ``free $A_\p$-module'', ``projective $A_\p$-module'', or ``free $\widehat{A_\p}$-module'', even if $A$ is commutative. We give a counter-example to all of these at the same time.
	
	Let $k$ be a field and $R:=k[x,y]/(y^{2}-x^{2}(x+1))$. The ring $R$ can be embedded into the polynomial ring $A:=k[t]$ by $x\mapsto t^{2}-1$ and $y\mapsto t(t^{2}-1)$. Then $R$ and $k[t]$ have the same quotient field $k(t)$, and $k[t]$ is the integral closure of $R$ in the quotient field $k(t)$. Note that $A=k[t]$ is a Noether $R$-algebra since $A=R+Rt$.
	
Consider the maximal ideal $\km:=(x,y)\subset R$. We have 
$\km^{n}A=(\kn_{1}\kn_{-1})^{n}$ for each $n\geq 1$, where $\kn_{i}=(t-i)\in\Max A$. The $\km$-adic completion of $A_{\km}$ has a decomposition $\widehat{A_{\km}}=A_{\km}^{\wedge}\cong A_{\kn_{1}}^{\wedge}\times A_{\kn_{-1}}^{\wedge}$
	as a ring (\cref{DecompositionOfCompletion}). Letting $M:=A_{\kn_{1}}^{\wedge}$, we have $M=\widehat{A_{\kn_{1}}}\cong T_A(\kn_{1})$, so this is an indecomposable flat cotorsion $A$-module and also is an indecomposable projective $\widehat{A_{\km}}$-module, which is $\km$-complete.
Thus $M$ satisfies the equivalent conditions in \cref{LocTorsFlModIsPInj}, setting  $\p:=\km$.

	However, $M$ is not isomorphic to the $\km$-adic completion of any free $\widehat{A_{\km}}$-module since such a completion is a direct sum of copies of $A_{\kn_{1}}^{\wedge}\times A_{\kn_{-1}}^{\wedge}$ (which is decomposable or zero). We also show that $M$ is not isomorphic to the $\km$-adic completion of any projective $A_{\km}$-module, either. Given a nonzero projective $A_{\km}$-module $P$, we have $P_{\km}^{\wedge}\cong P_{\kn_{1}}^{\wedge}\oplus P_{\kn_{-1}}^{\wedge}$ (\cref{DecompositionOfCompletion}). Since $\bigcap_{n\geq 1} \kn_i^n F=0$ for every free $A_\km$-module $F$ (see \cite[Theorem~8.10]{MR1011461}), the canonical map $F\to F_{\kn_{i}}^{\wedge}$ is injective, so the same holds for $P$. Therefore each $P_{\kn_{i}}^{\wedge}$ is nonzero, and hence $P_{\km}^{\wedge}$ is not indecomposable.
\end{remark}

\section{Structure of flat cotorsion modules}
\label{sec.DescripOfFlCotMod}

Let us now complete the proof of the structure theorem for flat cotorsion modules (\cref{intro.ClassifOfFlCotOverNoethAlg}):

\begin{theorem}\label{ClassifOfFlCotOverNoethAlg}
	Let $A$ be a Noether $R$-algebra. A right $A$-module $M$ is flat cotorsion if and only if $M$ is isomorphic to
		\begin{equation}\label{eq.FlCotDecomp}
			\prod_{P\in\Spec A}(T_{A}(P)^{(B_{P})})_{P\cap R}^{\wedge}
		\end{equation}
		for some family of sets $\{B_{P}\}_{P\in\Spec A}$.
The cardinality of each $B_{P}$ is uniquely determined by $M$.
\end{theorem}

\begin{proof}
This follows from \cref{DecompToLocCompFlCot} and \cref{DecompFlCotAtEachBasePrime}.
\end{proof}

Consequently, we obtain a complete description of indecomposable flat cotorsion modules (\cref{intro.ClassifOfIndecFlCotOverNoethAlg}):

\begin{corollary}\label{ClassifOfIndecFlCotOverNoethAlg}
	Let $A$ be a Noether $R$-algebra. Then there is a bijection
	\begin{equation*}
		\Spec A\isoto\setwithtext{isoclasses of indecomposable flat cotorsion right $A$-modules}
	\end{equation*}
	given by $P\mapsto T_{A}(P)=\Hom_R(I_{A^{\op}}(P), E_R(R/P\cap R))$.
\end{corollary}

\begin{proof}
By \cref{CompletionNecessity}, $T_{A}(P)$ is indecomposable. The uniqueness of the cardinalities of $B_{P}$ in \cref{ClassifOfIndecFlCotOverNoethAlg} implies that the map in the statement is injective.

To observe the surjectivity, take an indecomposable flat cotorsion right $A$-module $M$. By \cref{ClassifOfFlCotOverNoethAlg}, $M$ is isomorphic to $(T_{A}(P)^{(B)})_{\p}^{\wedge}$ for some $P\in\Spec A$ and a nonempty set $B$, where $\p:=P\cap R$. 
Since $T_{A}(P)\cong T_{A}(P)_{\p}^{\wedge}$ is a direct summand of $(T_{A}(P)^{(B)})_{\p}^{\wedge}$, the indecomposability of $M$ implies that the cardinality of $B$ is one, and hence $M\cong T_{A}(P)$.
\end{proof}

\begin{example}\label{ex.TriMatAlg}
	Let $R$ be a commutative noetherian ring and let $A$ be the $2\times 2$ lower triangular matrix algebra over $R$, that is,
	\begin{equation*}
		A=\begin{pmatrix}R & 0 \\ R & R\end{pmatrix}.
	\end{equation*}
	Then $A$ is a Noether $R$-algebra. We describe all isoclasses of simple, indecomposable injective, and indecomposable flat cotorsion right $A$-modules. The algebra $A$ has a decomposition
	\begin{equation*}
		A=\begin{pmatrix}R & 0\end{pmatrix}\oplus\begin{pmatrix}R & R\end{pmatrix}
	\end{equation*}
	as a right $A$-module, where the action of $A$ is matrix multiplication.
	For each $\p\in\Spec R$,
	\begin{equation*}
		P_{1}(\p):=\begin{pmatrix}\p & 0 \\ R & R\end{pmatrix}\quad\text{and}\quad P_{2}(\p):=\begin{pmatrix}R & 0 \\ R & \p\end{pmatrix}
	\end{equation*}
	are prime ideals of $A$, and varying $\p$, these are all the prime ideals of $A$ (see \cref{SpecOfNoethAlg}). The simple right $A_{\p}$-modules are
	\begin{equation*}
		S_{A}(P_{1}(\p))=\begin{pmatrix}\kappa(\p) & 0\end{pmatrix}\quad\text{and}\quad S_{A}(P_{2}(\p))=\frac{\begin{pmatrix}\kappa(\p) & \kappa(\p)\end{pmatrix}}{\begin{pmatrix}\kappa(\p) & 0\end{pmatrix}}.
	\end{equation*}
	By \cref{nP}, $n_{P_{i}(\p)}=1$ for $i=1, 2$. On the other hand, the algebra $A$ has a decomposition
	\begin{equation*}
		A=\begin{pmatrix}R \\ R\end{pmatrix}\oplus\begin{pmatrix}0 \\ R\end{pmatrix}
	\end{equation*}
	as a left $A$-module, and we have
	\begin{equation*}
		\Hom_{R}(A,E_{R}(R/\p))\cong\begin{pmatrix}E_{R}(R/\p) & E_{R}(R/\p)\end{pmatrix}\oplus\frac{\begin{pmatrix}E_{R}(R/\p) & E_{R}(R/\p)\end{pmatrix}}{\begin{pmatrix}E_{R}(R/\p) & 0\end{pmatrix}}
	\end{equation*}
	as right $A$-modules. Hence, by \cref{DecompOfInjDualOfAlg},
	\begin{equation*}
		I_{A}(P_{1}(\p))=\begin{pmatrix}E_{R}(R/\p) & E_{R}(R/\p)\end{pmatrix}\quad\text{and}\quad I_{A}(P_{2}(\p))=\frac{\begin{pmatrix}E_{R}(R/\p) & E_{R}(R/\p)\end{pmatrix}}{\begin{pmatrix}E_{R}(R/\p) & 0\end{pmatrix}}
	\end{equation*}
	because each $I_{A}(P_{i}(\p))$ should have $S_{A}(P_{i}(\p))$ as a right $A$-submodule. Similarly, by \cref{DecompOfCompOfAlg},
	\begin{equation*}
		T_{A}(P_{1}(\p))=\begin{pmatrix}\widehat{R_{\p}} & 0\end{pmatrix}\quad\text{and}\quad T_{A}(P_{2}(\p))=\begin{pmatrix}\widehat{R_{\p}} & \widehat{R_{\p}}\end{pmatrix}
	\end{equation*}
	since each $T_{A}(P_{i}(\p))$ should have $S_{A}(P_{i}(\p))$ as a quotient $A$-module.
\end{example}

\begin{remark}\label{FinDimAlg}
	Let us consider the case where $R=k$ is a field, that is, $A$ is a finite-dimensional $k$-algebra. As mentioned in \cref{ArtinFlcovProjcov}, all flat right $A$-modules are projective and all right $A$-modules are cotorsion. Thus the flat cotorsion right $A$-modules are precisely the projective right $A$-modules. For every $P\in\Spec A$,
	\begin{equation*}
		T_{A}(P)=\Hom_{k}(I_{A^{\op}}(P),k)
	\end{equation*}
	is the projective cover of $S_{A}(P)$ (see \cref{DualOfInjEnv}), and the product in \cref{ClassifOfFlCotOverNoethAlg} can be written as
	\begin{equation*}
		\bigoplus_{P\in\Spec A}T_{A}(P)^{(B_{P})}
	\end{equation*}
	since $\Spec A$ is a finite set by \cref{SpecOfNoethAlg} and $P\cap k=0$ for each $P\in \Spec A$.
\end{remark}

\section{Flat cotorsion modules as flat covers and pure-injective envelopes}
\label{sec.FlCotModAsFlCovCotEnv}

Let $A$ be a Noether $R$-algebra. 
In this section, we prove \cref{FlatCoverAndCotorsionEnvelopeTheorem}, which gives other descriptions of each flat cotorsion right $A$-module in terms of a flat cover and a pure-injective envelope.

\begin{lemma}\label{ProdOfFlCover}
	For each $\p\in\Spec R$, let  $f(\p)\colon F(\p)\to M(\p)$ be a flat cover in $\Mod A$ such that $F(\p)$ is $\p$-local and $\p$-complete. Then the product
	\begin{equation*}
		\prod_{\p\in\Spec R}f(\p)\colon\prod_{\p\in\Spec R}F(\p)\to\prod_{\p\in\Spec R}M(\p)
	\end{equation*}
	is a flat cover.
\end{lemma}

\begin{proof}
	Denote the product of morphisms by $f\colon F\to M$, where $F$ is a flat right $A$-module since $A$ is left noetherian. 
	For every flat right $A$-module $F'$, the morphism $\Hom_{A}(F',f(\p))$ is an epimorphism since $f(\p)$ is a flat (pre)cover. Hence the product $\Hom_{A}(F',f)=\prod_{\p\in\Spec R}\Hom_{A}(F',f(\p))$ is also an epimorphism. This shows that $f$ is a flat precover.
	
	It remains to show that $f$ is right minimal. 
	Let $g\in\End_{A}(F)$ such that $fg=f$. For each $\q\in\Spec R$, we have a commutative diagram
	\begin{equation*}
		\begin{tikzcd}
			F(\q)\ar[ddr,"f(\q)"']\ar[rr,hookrightarrow,"\textnormal{inclusion}"] & & F\ar[dr,"f"']\ar[rr,"g"] & & F\ar[rr,twoheadrightarrow,"\textnormal{projection}"]\ar[dl,"f"] & & F(\q)\ar[ddl,"f(\q)"] \\
			& & & M\ar[drr,twoheadrightarrow,"\textnormal{projection}"] & & & \\
			& M(\q)\ar[urr,hookrightarrow,"\textnormal{inclusion}"]\ar[rrrr,equal] & & & & M(\q)\rlap{.} &
		\end{tikzcd}
	\end{equation*}
	Since $f(\q)$ is a flat cover, the composition in the first row is an isomorphism. 
	Therefore, $g$ is an isomorphism by \cref{IsomOfProdIffIsomOfComponent}.
\end{proof}

A special case of \cref{ProdOfFlCover} is discussed in the third paragraph of the proof of \cite[p.~183, Theorem]{MR754698}.

The assumption in \cref{ProdOfFlCover} that each $F(\p)$ is $\p$-local and $\p$-complete is satisfied if each 
$M(\p)$ is $\p$-local and $\p$-complete, by \cref{ClosureEnvCovLocal} and \cref{FlatCoverIdeal}\cref{FlatCoverIdeal.Complete}.

\begin{proposition}\label{PureInjectiveEnvelopeProjective}
Let $M$ be a right $A$-module that is finitely generated or projective. Then the morphism $M\to \prod_{\km\in \Max R} M^\wedge_\km$ induced by the completion maps $M\to M^\wedge_\km$ is a pure-injective envelope.
\end{proposition}

When $A=R$, this is shown in \cite[Proposition~6.7.3 and Remark~6.7.12]{MR1753146}. Let us first recall an elementary fact before giving a proof.

\begin{remark}\label{PureMonomorphismLocalization}
For every right $A$-module $M$, the morphism $g\colon M\to \prod_{\km\in \Max R} M_\km$ induced by the localization maps $M\to M_\km$ is a pure monomorphism, or equivalently, $g\otimes_AN$ is a monomorphism in $\Mod R$ for every finitely generated (presented) left $A$-module $N$ (see \cite[Lemma~2.19]{MR2985554}).
Indeed, the functor $-\otimes_AN$ commutes with arbitrary direct products (see \cite[Theorem~3.2.22]{MR1753146}) and the localization functors $(-)_\km$, so the morphism $g\otimes_AN$ can be written as 
$M\otimes_RN\to \prod_{\km\in \Max R} (M\otimes_RN)_\km$. This is a monomorphism, because if a given element $x$ of $M\otimes_RN$ becomes zero in $(M\otimes_RN)_\km$ for all $\km\in\Max R$, then $x$ is zero in $M\otimes_RN$; see \cite[Theorem~4.6]{MR1011461}.

It is also seen from the above argument that, given a family $\{M_b\}_{b\in B}$ of right $A$-modules, the canonical inclusion $\bigoplus_{b\in B}M_b\hookrightarrow \prod_{b\in B} M_b$ is a pure monomorphism (\cite[Lemma~2.1.10]{MR2530988}).
\end{remark}

\begin{proof}[Proof of \cref{PureInjectiveEnvelopeProjective}]
Denote by $f$  the morphism $M\to \prod_{\km\in \Max R} M^\wedge_\km$. To see that $f$ is left minimal, it suffices, by \cref{Orthogonality}, to show that each morphism $M\to M^\wedge_\km$ is left minimal for all $\km\in\Max R$ (since $M^\wedge_\km$ is $\km$-local and $\km$-complete), and this follows from the adjoint property of the $\km$-adic completion functor (\cref{CompAdj}).

It remains to check that $f$ is a pure-injective preenvelope.
If $M$ is finitely generated, then each $M^\wedge_\km$ is pure-injective by \cref{InjDualProperties}\cref{InjDualProperties.Cot}, \cref{ClosureEnvCovLocal}, and \cref{DoubleMatlisDual1}. 
If $M$ is projective, then $\prod_{\km\in \Max R} M^\wedge_\km$ is flat cotorsion (\cref{LocTorsFlModIsPInj}) and hence pure-injective by \cref{InjDualProperties3}.
Therefore, it suffices to check that $f$ is a pure monomorphism; see \cref{PInjEnv}\cref{PInjEnv.Preenv}.
By \cref{PureMonomorphismLocalization}, we only need to check that the completion map $M_\km\to M_\km^\wedge$ is a pure monomorphism for each $\km \in \Max R$, so 
we may assume that $A$ is a Noether algebra over a local ring $R$ with maximal ideal $\km$.

If $M$ is finitely generated, then the completion map $M\to \widehat{M}$ is a pure monomorphism since it coincides with the map induced by the pure monomorphism $R\to \widehat{R}$; see \cref{subsec.MatlisDual}.

If $M$ is projective, then we may replace it by a free module $A^{(B)}$ with basis $B$. The inclusion $g\colon A^{(B)} \into A^B$ and the canonical morphism 
$\eta\colon\id_{\Mod A}\to \varLambda^\km$ of functors $\Mod A\to\Mod A$ yield a commutative diagram:
\begin{equation*}
		\begin{tikzcd}
		A^{(B)}\ar[d,"g"']\ar[rr, "\eta (A^{(B)})"] && \varLambda^\km(A^{(B)})\ar[d,"\varLambda^\km g"] \\
			A^B\ar[rr,  "\eta (A^{B})"'] &&\varLambda^\km(A^B)\rlap{.}
		\end{tikzcd}
\end{equation*}
The functor $\varLambda^\km$ commutes with arbitrary direct products (\cref{CommutativityWithDirestProduct}), so $\varLambda^\km(A^B)\cong (\varLambda^\km A)^B$ and $\eta(A^B)$ is identified with $\eta(A)^B$. 

Now, $g$ is a pure monomorphism (\cref{PureMonomorphismLocalization}). The completion map $\eta(A)\colon A \to \varLambda^\km A=\widehat{A}$ is also a pure monomorphism as we recalled above, and hence so is $\eta (A)^B=\eta (A^B)$ (because tensoring a finitely generated module commutes with arbitrary direct products; see \cref{PureMonomorphismLocalization}). Therefore the commutative diagram above implies that $\eta (A^{(B)})$ is a pure monomorphism, as desired.
\end{proof}

As a consequence of \cref{PureInjectiveEnvelopeProjective}, we obtain the following remark:

\begin{remark}\label{CompOfDSumIsPInjEnv}
Let $P\in\Spec A$ and $\p:=P\cap R$. 
Recall that $T_A(P)$ is a projective right $\widehat{A_\p}$-module (\cref{HomomorphismBetweenLocalCompleteModules}). 
It then follows from \cref{PureInjectiveEnvelopeProjective} that the completion map 
\begin{equation*}f\colon T_{A}(P)^{(B)}\to (T_{A}(P)^{(B)})_{\p}^{\wedge}
\end{equation*}
 is a pure-injective envelope in $\Mod \widehat{A_\p}$, for every set $B$. 
Embedding this map into a pure exact sequence, we notice that the cokernel of $f$ is a flat right $\widehat{A_\p}$-module (see \cref{CotEnvAndPInjEnvForFlatMod,SpeFlCovCotEnv}). The cokernel is also a flat right $A$-module, as the canonical maps $A\to A_\p\to \widehat{A_\p}$ are flat ring homomorphisms. This shows that $f$ is a pure monomorphism in $\Mod A$ as well. Moreover, $f$ is left minimal in $\Mod A$ by \cref{CompAdj}, and $(T_{A}(P)^{(B)})_{\p}^{\wedge}$ is pure-injective in $\Mod A$ by \cref{DualOfInj}. Therefore, $f$ is a pure-injective envelope in $\Mod A$.

Let us consider the case where $A=R$ is a local ring and $P=\km$ is its maximal ideal. Then $T_R(\km)\cong \widehat{R}$, so the pure-injective envelope of $\widehat{R}^{(B)}$ (which is also the cotorsion envelope) is the completion map $f\colon\widehat{R}^{(B)}\to \widehat{R^{(B)}}$. If $B$ is an infinite set and the Krull dimension of $R$ is greater than $0$, then $f$ is not an isomorphism. This shows that a direct sum of copies of $T_A(P)$ is neither pure-injective nor cotorsion in general.
\end{remark}

\begin{lemma}\label{PInjEnvOfDSum}
	For each $\p\in\Spec R$, let  $g(\p)\colon M(\p)\to H(\p)$ be a pure-injective envelope in $\Mod A$ such that $H(\p)$ is $\p$-local and $\p$-complete. Then the morphism
	\begin{equation*}
		g\colon\bigoplus_{\p\in\Spec R}M(\p)\to\prod_{\p\in\Spec R}H(\p)
	\end{equation*}
	induced by $\{g(\p)\}_{\p\in \Spec R}$ is a pure-injective envelope.
\end{lemma}

\begin{proof}
We first remark  that $g$ is  factorized as the composition of the direct sum
\begin{equation*}
\bigoplus_{\p\in\Spec R} g(\p)\colon\bigoplus_{\p\in\Spec R}M(\p)\to  \bigoplus_{\p\in\Spec R}H(\p)
\end{equation*}
 and the canonical map $\bigoplus_{\p\in\Spec R}H(\p)\to \prod_{\p\in\Spec R}H(\p)$, where the former map is evidently a pure monomorphism, and so is the latter by \cref{PureMonomorphismLocalization}. It follows from the definition of pure-injective modules that the direct product $\prod_{\p\in\Spec R}H(\p)$ is pure-injective. 
Thus, $g$ is a pure monomorphism into a pure-injective module, that is, $g$ is a pure-injective preenvelope; see \cref{PInjEnv}\cref{PInjEnv.Preenv}. 

It remains to show the left minimality of $g\colon M\to H$, where $M:= \bigoplus_{\p\in\Spec R}M(\p)$ and $H:=\prod_{\p\in\Spec R}H(\p)$. 
 Let $h\in\End_{A}(H)$ with $hg=g$. For each $\q\in\Spec R$, we have a commutative diagram
	\begin{equation*}
		\begin{tikzcd}
			& \ar[ddl,"g(\q)"'] M(\q)\ar[drr,hookrightarrow,"\textnormal{inclusion}"']\ar[rrrr,equal] & & & & M(\q)\ar[ddr,"g(\q)"] & \\
			& & & \ar[dl,"g"'] M\ar[dr,"g"] \ar[urr,twoheadrightarrow,"\textnormal{projection}"'] & & & \\
			H(\q)\ar[rr,hookrightarrow,"\textnormal{inclusion}"'] & & H\ar[rr,"h"'] & & H\ar[rr,twoheadrightarrow,"\textnormal{projection}"'] & & H(\q)\rlap{.}
		\end{tikzcd}
	\end{equation*}
Since $g(\q)$ is a pure-injective envelope, the composition in the second row is an isomorphism. Therefore, $h$ is an isomorphism by \cref{IsomOfProdIffIsomOfComponent}.
\end{proof}

We can now prove the main result in this section.
Recall that, for a right $A$-module $M$, its pure-injective envelope and cotorsion envelope are denoted by $H_A(M)$ and $C_A(M)$, respectively.

\begin{theorem}
\label{FlatCoverAndCotorsionEnvelopeTheorem}
For every family of sets $\{B_{P}\}_{P\in\Spec A}$, we have isomorphisms of right $A$-modules
		\begin{align*}
			F_{A}(\prod_{P\in\Spec A}S_{A}(P)^{(B_{P})})\cong \prod_{P\in\Spec A}(T_{A}(P)^{(B_{P})})_{P\cap R}^{\wedge}\cong H_{A}(\bigoplus_{P\in\Spec A}T_{A}(P)^{(B_{P})}),
			\end{align*}
where $H_A$ can be replaced by $C_A$.
\end{theorem}

\begin{proof}
For every $P\in \Spec A$, we have a flat cover 
\begin{equation*}
f(P)\colon (T_{A}(P)^{(B_{P})})_{P\cap R}^{\wedge}\to S_{A}(P)^{(B_{P})}
\end{equation*}
	and a pure-injective envelope 
		\begin{equation*}
		g(P)\colon T_{A}(P)^{(B_{P})} \to (T_{A}(P)^{(B_{P})})_{P\cap R}^{\wedge}\end{equation*}
		 by \cref{DualOfInjEnv}, \cref{DualOfInj}, and \cref{CompOfDSumIsPInjEnv}.
 Fix $\p\in \Spec R$, and define $f(\p) \colon F(\p)\to M(\p)$ and $g(\p)\colon N(\p)\to H(\p)$ as the direct sum of $f(P)$ and the direct sum of $g(P)$, respectively, for all $P\in \Spec A$ with $P\cap R=\p$. 
Since there are only finitely many such $P$ (\cref{SpecOfNoethAlg}), $f(\p)$ and $g(\p)$ are a flat cover and a pure-injective envelope, respectively (see \cite[Theorems 1.2.5 and 1.2.10]{MR1438789}). Moreover, $F(\p)=H(\p)$ is $\p$-local and $\p$-complete.
Thus the first and the second isomorphisms in the theorem follow from \cref{ProdOfFlCover} and \cref{PInjEnvOfDSum}, respectively. \cref{CotEnvAndPInjEnvForFlatMod} shows that $H_A$ can be replaced by $C_A$.
\end{proof}

\begin{remark}\label{UniquenessOfCard}
It is known that each pure-injective right module $M$ over an arbitrary ring $A$ has a decomposition $M\cong H_A(\bigoplus_{c\in C}M_c^{(B_c)})\oplus N$, where $\set{M_c}_{c\in C}$ is a family of indecomposable pure-injective modules such that $M_c\not \cong M_{c'}$ whenever $c\neq c'$, $\set{B_{c}}_{c\in C}$ is a family of sets, and $N$ is a \emph{superdecomposable} module, that is, a module having no indecomposable direct summands. The cardinality of each $B_c$ and the isoclass of $N$ are uniquely determined by $M$. See \cite[Theorem~4.4.2]{MR2530988}.

For a flat cotorsion right module $M$ over a Noether $R$-algebra $A$, this fact has been explicitly realized as 
	\begin{equation*}
		M\cong H_{A}(\bigoplus_{P\in\Spec A}T_{A}(P)^{(B_{P})})
	\end{equation*}
by \cref{ClassifOfFlCotOverNoethAlg} and the second isomorphism in \cref{FlatCoverAndCotorsionEnvelopeTheorem}. Note that the superdecomposable summand $N$ is interpreted as the zero module.
\end{remark}

\begin{remark}
Another general fact we should mention is that every flat right module over an arbitrary ring admits a pure monomorphism into a direct product of indecomposable flat cotorsion right modules; this was shown by Guil Asensio and Herzog \cite[Corollary~10]{MR2377125}.
In particular, if the ring is left coherent, then this result implies that every flat cotorsion right module is a direct summand of a direct product of indecomposable flat cotorsion right modules, as flat cotorsion right modules are pure-injective (see \cref{NonnoethereianCase}).

In the case of a Noether $R$-algebra $A$, we can recover the result (\cite[Corollary~10]{MR2377125}) as follows:
Given a flat right $A$-module $M$, the pure-injective envelope $M\to H_{A}(M)$ is a pure monomorphism into a flat cotorsion module (see \cref{CotEnvAndPInjEnvForFlatMod,SpeFlCovCotEnv,PInjEnv}), so we may assume that $M$ itself is flat cotorsion, and thus
	\begin{equation*}
		M\cong\prod_{P\in\Spec A}(T_{A}(P)^{(B_{P})})_{P\cap R}^{\wedge}
	\end{equation*}
	by \cref{ClassifOfFlCotOverNoethAlg}. We show that the canonical morphism
	\begin{equation}\label{eq.EmbIntoProdOfIndecFlCot}
		\prod_{P\in\Spec A}(T_{A}(P)^{(B_{P})})_{P\cap R}^{\wedge}\to\prod_{P\in\Spec A}(T_{A}(P)^{B_{P}})_{P\cap R}^{\wedge}
	\end{equation}
	is a pure monomorphism. As we observed in \cref{PureMonomorphismLocalization}, it suffices to see that $-\otimes_{A}N$ applied to \cref{eq.EmbIntoProdOfIndecFlCot} is a monomorphism for every finitely generated left $A$-module $N$. We also observed that $-\otimes_{A}N$ commutes with direct products. Since $T_{A}(P)^{(B_{P})}$ and $T_{A}(P)^{B_{P}}$ are flat, \cref{FlTensorCompAsTensor} implies that $-\otimes_{A}N$ applied to \cref{eq.EmbIntoProdOfIndecFlCot} becomes
	\begin{equation*}
		\prod_{P\in\Spec A}((T_{A}(P)\otimes_{A}N)^{(B_{P})})_{P\cap R}^{\wedge}\to\prod_{P\in\Spec A}((T_{A}(P)\otimes_{A}N)^{B_{P}})_{P\cap R}^{\wedge},
	\end{equation*}
	which is clearly a monomorphism. Thus \cref{eq.EmbIntoProdOfIndecFlCot} is a pure monomorphism. Since completion commutes with direct products (\cref{CommutativityWithDirestProduct}) and each $T_{A}(P)$ is $(P\cap R)$-complete, the right-hand side of \cref{eq.EmbIntoProdOfIndecFlCot} is $\prod_{P\in\Spec A}T_{A}(P)^{B_{P}}$, which is a direct product of indecomposable flat cotorsion right $A$-modules.
\end{remark}

\section{Ziegler spectra and elementary duality}
\label{sec.ZgAndElementaryDual}

Let $A$ be a Noether $R$-algebra. Combining \cref{InjOverNoethRing} and \cref{ClassifOfIndecFlCotOverNoethAlg}, it follows that there exists a one-to-one correspondence between the isoclasses of indecomposable injective left $A$-modules and the isoclasses of indecomposable flat cotorsion right $A$-modules, given by $I_{A^{\op}}(P)\mapsto T_{A}(P)$ for each $P\in\Spec A$. In this section, we observe that this one-to-one correspondence is compatible, and is actually induced from, elementary duality between the Ziegler spectrum of $A^{\op}$ and that of $A$ (\cref{FlCotAndInjAreRefl}).

For a while, let $A$ be an arbitrary ring.
Denote by $\fp(\mod A,\Ab)$ the category of finitely presented additive functors $\mod A\to\Ab$, where $\mod A$ is the category of finitely presented right $A$-modules and $\Ab$ is the category of abelian groups.
Each functor $F\in\fp(\mod A,\Ab)$ admits a unique extension $\overrightarrow{F}\colon\Mod A\to\Ab$ (up to isomorphism) that commutes with (filtered) direct limits. By definition, there exists an exact sequence
$\Hom_A(M,-)\xrightarrow{-\circ f} \Hom_A(L,-)\to F\to 0$ in $\fp(\mod A,\Ab)$, where $f\colon L\to M$ is a morphism in $\mod A$, and the extension $\overrightarrow{F}$ can be defined as the cokernel of the same morphism $\Hom_A(M,-)\xrightarrow{-\circ f} \Hom_A(L,-)$ but regarded as a morphism of functors $\Mod A\to\Ab$; see \cite[Corollary~10.2.42]{MR2530988}.

Denote by $\Zg_A$ the \emph{Ziegler spectrum} of $A$, which is a topological space whose points are the isoclasses of indecomposable pure-injective right $A$-modules (they actually form a small set; see \cite[Corollary~4.3.38]{MR2530988}). The topology on $\Zg_{A}$ is defined so that $\setwithcondition{(F)}{F\in\fp(\mod A,\Ab)}$ is an open basis, where
\begin{equation*}
	(F):=\setwithcondition{N\in\Zg_{A}}{\overrightarrow{F} (N)\neq 0}
\end{equation*}
for each $F\in\fp(\mod A,\Ab)$; see \cite[Corollary~10.2.45]{MR2530988}.

Although this definition of the topology is convenient, it should be mentioned that the topology was originally introduced in terms of model theory for modules, and such a viewpoint helps us to understand elementary duality, particularly via \cref{CharDualOfZgPt}.
For this reason, we interpret the topology on the Ziegler spectrum via model theoretic language.

Let $l$, $m$, and $n$ be nonnegative integers. 
Let $H$ and $H'$ be matrices whose entries are elements of $A$, where $H$ is an $n\times l$ matrix and $H'$ is an $m\times l$ matrix.
A \emph{pp-formula} (\emph{positive primitive formula}) $\varphi$ for right $A$-modules is a formula of the form $\exists y(xH=yH')$, where $x=(x_1,\ldots,x_n)$ is a tuple of free variables and $y=(y_{1},\ldots,y_{m})$ is a tuple of variables bound by the existential quantifier. So $n$ is referred to as the number of free variables in $\varphi$.

For each right $A$-module $M$, the pp-formula $\varphi$ defines an abelian subgroup of $M^n$ as
\begin{equation*}
	F_{\varphi}(M):=\setwithcondition{x\in M^n}{\textnormal{there exists $y\in M^m$ such that $xH=yH'$}},
\end{equation*}
where $x$ and $y$ are regarded as row vectors. A subgroup of $M^n$ of this form (for some $l$ and $m$) is called a subgroup of $M^{n}$ \emph{pp-definable} in $M$. For every morphism $f\colon M\to N$ in $\Mod A$, the direct sum $f^{n}\colon M^{n}\to N^{n}$ restricts to a homomorphism $F_{\varphi}(M)\to F_{\varphi}(N)$ of abelian groups. So we obtain an additive functor $F_{\varphi}\colon\Mod A\to \Ab$.

Let $\varphi$ and $\psi$ be pp-formulas for right $A$-modules, and suppose that they have the same number, say $n$, of free variables. 
Then both $F_{\varphi}(M)$ and $F_{\psi}(M)$ are subgroups of $M^n$ for each $M\in \Mod A$. We write $\varphi\leq \psi$ if $F_{\varphi}(M)\subseteq F_{\psi}(M)$ for all right $A$-modules $M$. 
Moreover, $\varphi$ and $\psi$ are said to be \emph{equivalent} if $\varphi\leq \psi$ and $\varphi \geq \psi$, in which case we have equality of functors $F_{\varphi}=F_{\psi}$.

A \emph{pp-pair} $\varphi/\psi$ is a pair of pp-formulas with $\varphi \geq \psi$. Each pp-pair $\varphi/\psi$ defines an additive functor $F_{\varphi/\psi}\colon\Mod A \to \Ab$ by the assignment $M\mapsto F_{\varphi}(M)/F_{\psi}(M)$. The following remarkable fact allows us to understand the topology on $\Zg_A$ via pp-pairs.

\begin{theorem}\label{PpPairAndFunctorCategory}
Let $A$ be a ring. For each pp-pair $\varphi/\psi$, the functor $F_{\varphi/\psi}\colon\Mod A \to \Ab$ commutes with direct limits and its restriction to $\mod A$ belongs to $\fp(\mod A,\Ab)$. Conversely, for each $F\in \fp(\mod A,\Ab)$, there exists a pp-pair $\varphi\geq \psi$ such that $\overrightarrow{F}\cong F_{\varphi/\psi}$ as functors $\Mod A \to \Ab$. 
\end{theorem}

\begin{proof}
See \cite[Lemma~1.2.31 and Remark~10.2.29]{MR2530988} for the first statement, and \cite[Proposition~10.2.43]{MR2530988} for the second.
\end{proof}

In fact, the category $\fp(\mod A, \Ab)$ is equivalent to the \emph{category of pp-pairs} for right $A$-modules; see \cite[Theorem~10.2.30]{MR2530988}.

It follows from \cref{PpPairAndFunctorCategory} that
\begin{equation*}
	\setwithcondition{(F_{\varphi/\psi})}{\textnormal{$\varphi/\psi$ is a pp-pair for right $A$-modules}}
\end{equation*}
is an open basis for $\Zg_{A}$.

We now explain elementary duality, first in terms of pp-formulas. Let $\varphi$ be a pp-formula $\exists y(xH=yH')$ for right $A$-modules, where 
$H$ is an $n\times l$ matrix and $H'$ is an $m\times l$ matrix.
Regarding the transposes $H^{\transpose}$ and $H'^{\transpose}$ as matrices over $A^{\op}$, we can define the pp-formula $D\varphi$ for right $A^{\op}$-modules to be $\exists z(xK=zK')$, where $x=(x_{1},\ldots,x_{n})$ is a tuple of free variables, $z=(z_1\ldots, z_l)$ is a tuple of bound variables, $K:=\begin{pmatrix}I&0\end{pmatrix}$, $K':=\begin{pmatrix}H^{\transpose}&H'^{\transpose}\end{pmatrix}$, $I$ is the $n\times n$ identity matrix, and $0$ is the $n\times m$ zero matrix. The pp-formula $D\varphi$ is called the \emph{elementary dual} of $\varphi$. For each right $A^{\op}$-module $M$,
\begin{align*}
	F_{D\varphi}(M)&=\setwithcondition{x\in M^{n}}{\textnormal{there exists $z\in M^{l}$ such that $xK=zK'$}}\\
	&=\setwithcondition{x\in M^n}{\textnormal{there exists $z\in M^l$ such that $x=zH^{\transpose}$ and $zH'^{\transpose}=0$}}.
\end{align*}
If we regard $M$ as a left $A$-module and $x$ and $z$ as column vectors, then the equations in the second line can be written as $x=Hz$ and $H'z=0$.

We can apply the same construction to $D\varphi$ and obtain the pp-formula $D^{2}\varphi=DD\varphi$ for right $A$-modules. Elementary duality claims that $D^{2}\varphi$ is equivalent to $\varphi$, and moreover, two pp-formulas $\varphi$ and $\psi$ satisfies $\varphi\geq\psi$ if and only if $D\psi\geq D\varphi$. Denote by $\mathrm{pp}^n_A$ the poset of equivalent classes of pp-formulas in $n$ free variables for right $A$-modules. In fact, this is a modular lattice; see \cite[\S 1.1.3]{MR2530988}.

\begin{theorem}[Elementary duality of pp-formulas]\label{ElementaryDualityOfPpConditions}
Let $A$ be a ring. The operator $D$ is an anti-isomorphism from $\mathrm{pp}^n_A$ to $\mathrm{pp}^n_{A^\op}$ for each $n\geq 0$.
\end{theorem}

\begin{proof}
See \cite[Proposition~1.3.1]{MR2530988}.
\end{proof}

The following fact is the key to describe elementary duality of Ziegler spectra:

\begin{lemma}\label{CharDualOfZgPt}
Let $A$ be a ring and let $M$ be a right $A$-module. Fix a ring homomorphism $S\to\End_{A}(M)$ from a ring $S$.
Let $E$ be an injective cogenerator in $\Mod S^{\op}$ and set $M^*:=\Hom_{S^{\op}}(M, E)\in \Mod A^\op$.
For each pp-pair $\varphi/\psi$, we have $F_{\varphi/\psi}(M)=0$ if and only if $F_{D\psi/D\varphi}(M^*)=0$.
\end{lemma}

\begin{proof}
	See \cite[Theorem~1.3.15]{MR2530988}.
\end{proof}

Let $U$ be an open subset of $\Zg_A$ such that $U=(F_{\varphi/\psi})$ for some pp-pair $\varphi/\psi$ for right $A$-modules. Since $D\psi/D\varphi$ is a pp-pair for right $A^{\op}$-modules, $(F_{D\psi/D\varphi})$ is an open subset of $\Zg_{A^{\op}}$, which does not depend on the choice of the pp-pair $\varphi/\psi$ for $U$. Indeed, for each $N\in\Zg_{A^\op}$, \cref{CharDualOfZgPt} (applied to $S:=\End_{A^\op}(N)$ and arbitrary $E$) implies that
\begin{equation*}
	N\in (F_{D\psi/D\varphi})\iff F_{D\psi/D\varphi}(N)\neq 0\iff F_{\varphi/\psi}(N^{*})\neq 0\iff N^{*}\in U.
\end{equation*}
Therefore we can write $\Do U:=(F_{D\psi/D\varphi})$.
Then $D^{2}U=U$ by \cref{ElementaryDualityOfPpConditions}, so this gives an order-preserving bijection between open bases of $\Zg_A$ and $\Zg_{A^\op}$, so it extends uniquely to an order-preserving bijection between all open subsets of $\Zg_A$ and those of $\Zg_{A^\op}$.
We summarize these facts in the next theorem, which was originally shown by Herzog \cite[Proposition~4.4]{MR1091706}.

\begin{theorem}[Elementary duality of Ziegler spectra]\label{def.ElementaryDuality}
For every ring $A$, there is an order-preserving bijection
\begin{equation*}
		\Do\colon\setwithtext{open subsets of $\Zg_{A}$}\isoto\setwithtext{open subsets of $\Zg_{A^{\op}}$},
	\end{equation*}
	which sends $(F_{\varphi/\psi})$ to $(F_{D\psi/D\varphi})$ for each pp-pair $\varphi/\psi$.
\end{theorem}

\begin{proof}
See \cite[Theorem~5.4.1]{MR2530988}.
\end{proof}

We may also interpret elementary duality as an order-preserving bijection between the closed subsets of $\Zg_A$ and those of $\Zg_{A^\op}$ in an obvious way; that is, given a closed subset $C\subseteq\Zg_A$, its complement $C^{\compl}=\Zg_A\setminus C$ is open, so send $C$ to $DC:=(\Do(C^{\compl}))^{\compl}=\Zg_{A^{\op}}\setminus\Do(C^{\compl})$. Note that, if $C$ is also open, then elementary duality for open subsets and that for closed subsets send $C$ to the same open closed subset of $\Zg_{A^{\op}}$. Thus we can safely denote both bijections by $D$.

\begin{remark}\label{ElementaryDualityDoesNotInduceHomeomorphism}
Elementary duality does not mean that there is a homeomorphism between $\Zg_{A}$ and $\Zg_{A^{\op}}$. This is due to the fact that the Ziegler spectrum is not necessarily a $T_{0}$-space, that is, they may contain topologically indistinguishable points (see \cite[p.~267]{MR2530988}). It is not known in general whether $\Zg_A$ is homeomorphic to $\Zg_{A^\op}$; see \cite[Question~5.4.8]{MR2530988}.
\end{remark}

We next explain how elementary duality is interpreted in terms of finitely presented functors.

\begin{theorem}[Auslander-Gruson-Jensen duality]\label{AGJDuality}
	For every ring $A$, there is a duality of categories
	\begin{equation*}
		d\colon\fp(\mod A,\Ab)\isoto\fp(\mod A^{\op},\Ab)
	\end{equation*}
	given by $F\mapsto dF$, where $(dF)(L):=\Hom(F,-\otimes_{A}L)$ for $L\in\mod A^{\op}$. Its quasi-inverse is given by $G\mapsto dG$, where $(dG)(M):=\Hom(G,M\otimes_{A}-)$ for $M\in\mod A$.
\end{theorem}

\begin{proof}
	\cite[Theorem~10.3.4]{MR2530988}.
\end{proof}

It is known that the equivalence $d$ in \cref{AGJDuality} sends $F_{\varphi/\psi}$ to $F_{D\psi/D\varphi}$ for each pp-pair $\varphi/\psi$; see \cite[Corollary~10.3.8]{MR2530988} and its proof.
Thus the bijection in \cref{def.ElementaryDuality} can also be written as $(F)\mapsto (dF)$ for $F\in\fp(\mod A,\Ab)$.

It would be worth noting that there is an order-preserving bijection between the open subsets of $\Zg_A$ and the Serre subcategories of $\fp(\mod A, \Ab)$ (\cite[Theorem~3.8]{MR1434441} and \cite[Theorem~4.2]{MR1426488}). Hence the bijection in \cref{def.ElementaryDuality} induces a bijection between the Serre subcategories of $\fp(\mod A, \Ab)$ and those of $\fp(\mod A^{\op},\Ab)$; this also follows from \cref{AGJDuality}.

On the other hand, there is an order-preserving bijection between the closed subsets of $\Zg_A$ and the definable subcategories of $\Mod A$, that is, full subcategories of $\Mod A$ closed under direct limits, direct products, and pure submodules (see \cite[Corollary~5.1.6]{MR2530988}). Typical examples are the subcategory of injective right $A$-modules when $A$ is right noetherian and the subcategory of flat right A-modules when $A$ is left coherent (see \cite[Theorem~3.4.28(a) and Theorem~3.4.24]{MR2530988}). Given a definable subcategory, its corresponding closed subset is obtained by collecting the isoclasses of indecomposable pure-injective modules in the subcategory.

Now let $A$ be a Noether $R$-algebra. Denote by $\inj_{A}$ (\resp $\flcot_A$) the set of isoclasses of indecomposable injective (\resp indecomposable flat cotorsion) right $A$-modules. As we observed in \cref{PureInjectiveCotorsion,InjDualProperties3}, the flat cotorsion right $A$-modules are precisely the flat pure-injective right $A$-modules. So $\inj_{A}$ and $\flcot_{A}$ are closed subsets of $\Zg_A$ by the above observation. We endow $\inj_{A}$ and $\flcot_{A}$ with the topologies induced from $\Zg_A$.

\begin{lemma}\label{FlCotAndIndecInjAndOpen}
	Let $A$ be a Noether $R$-algebra and let $P\in\Spec A$. For each open subset $U\subset\Zg_{A^{\op}}$, we have $I_{A^{\op}}(P)\in U$ if and only if $T_{A}(P)\in \Do U$.
\end{lemma}

\begin{proof}
Since elementary duality $\Do$ is order-preserving and $\Zg_{A^{\op}}$ has an open basis $\set{(F_{\varphi/\psi})}$, we may assume that $U=(F_{\varphi/\psi})$ for some pp-pair $\varphi/\psi$ for right $A^\op$-modules, and hence $\Do U=(F_{D\psi/D\varphi})$.

Let $\p:=P\cap R$. 
Since $I_{A^\op}(P)$ is $\p$-local by  \cref{IndecomposableInjectiveLocal}, there is a ring homomorphism $R_\p\to \End_{A^\op}(I_{A^\op}(P))$ given by scalar multiplication. Moreover, $E_R(R/\p)\cong E_{R_\p}(\kappa(\p))$ is an injective cogenerator in $\Mod R_\p$ (\cite[Lemma~A.27]{MR2355715}), and by definition $T_A(P)=\Hom_R(I_{A^\op}(P), E_R(R/\p))$. 
Thus, it follows from \cref{CharDualOfZgPt} that  
$F_{\varphi/\psi}(I_{A^\op}(P))\neq 0$ if and only if $F_{D\psi/D\varphi}(T_A(P))\neq 0$.
Therefore $I_{A^{\op}}(P)\in U$ if and only if $T_{A}(P)\in \Do U$.
\end{proof}

\begin{theorem}\label{HomeoBwFlCotAndInjOverNoethAlg}
	Let $A$ be a Noether $R$-algebra. Then the bijection $\inj_{A^{\op}}\isoto \flcot_{A}$ given by $I_{A^{\op}}(P)\mapsto T_{A}(P)$ is a homeomorphism.
\end{theorem}

\begin{proof}
\cref{FlCotAndIndecInjAndOpen} implies that, for each open subset $U\subset\Zg_{A^{\op}}$, the bijection $\inj_{A^{\op}}\isoto \flcot_{A}$ restricts to a bijection $U\cap\inj_{A^{\op}}\isoto DU\cap\flcot_{A}$. Hence the result follows.
\end{proof}

We can deduce from \cref{FlCotAndIndecInjAndOpen} that   
\begin{equation}\label{DescriptionOfElementaryDual}
\Dc(\inj_{A^\op})=\flcot_A.
\end{equation}
for a Noether $R$-algebra $A$. 
Indeed, setting $U:=(\inj_{A^\op})^{\compl}$, we obtain $\Do U\cap \flcot_A=\emptyset$ from \cref{FlCotAndIndecInjAndOpen}, and hence $\Dc(\inj_{A^\op})=(\Do U)^{\compl}\supseteq \flcot_A$. On the other hand, setting $O:=(\flcot_{A})^{\compl}$ and applying \cref{FlCotAndIndecInjAndOpen} to $\Do O\subseteq \Zg_{A^\op}$, we obtain $\Do O\cap\inj_{A^{\op}}=\emptyset$. This implies that $\inj_{A^{\op}}\subset (DO)^{\compl}=\Dc(\flcot_{A})$, and hence $\Dc(\inj_{A^{\op}})\subset\Dc^{2}(\flcot_{A})=\flcot_{A}$. Therefore \cref{DescriptionOfElementaryDual} holds.

In fact, \cref{DescriptionOfElementaryDual} holds for an arbitrary left coherent ring $A$ (\cite[Theorem~9.3]{MR1091706}). Moreover, Herzog proved that elementary duality ``constitutes'' a homeomorphism $\inj_{A^{\op}}\isoto \flcot_{A}$ for a class of rings $A$, including all left noetherian rings (\cite[Corollary~9.6]{MR1091706}).
In the rest of this section, we prove that our homeomorphism in \cref{HomeoBwFlCotAndInjOverNoethAlg} coincides with Herzog's one 
when $A$ is a Noether $R$-algebra.

Recall that a \emph{generic point} of a topological space $X$ is a point $x\in X$ whose closure is the whole space $X$.

\begin{definition}\label{def.Refl}
Let $A$ be a ring. A point $N\in\Zg_{A}$ is called \emph{reflexive} if its closure $\overline{\set{N}}$ has a unique generic point (which is necessarily $N$) and if the elementary dual $\Dc\overline{\set{N}}$ of $\overline{\set{N}}$ also has 
a unique generic point. In this case, the generic point of $\Dc\overline{\set{N}}$ is denoted by $\Dc N$ and called the \emph{elementary dual} of $N$.
\end{definition}

If $N\in\Zg_{A}$ is reflexive, then $\Dc N$ is also reflexive and $\Dc^{2} N=N$ by definition. Thus we have a bijection between the reflexive points in $\Zg_{A}$ and those in $\Zg_{A^{\op}}$ given by $N\mapsto \Dc N$.
Herzog's homeomorphism $\inj_{A^\op}\isoto \flcot_A$ (for a left noetherian ring $A$) is realized as
a restriction of this bijection based on the fact that all points of $\inj_{A^{\op}}$ and $\flcot_{A}$ are reflexive; see \cite[the last paragraph of \S 4 and the paragraph preceding Corollary~9.6]{MR1091706}, where the definition of reflexivity (see \cite[the paragraph preceding Theorem~4.10]{MR1091706}) is stronger than ours following \cite[p.~271]{MR2530988}. The dual of a reflexive point $N$ in the former sense is actually $DN$ defined as above; see \cite[Theorems 5.3.2 and 5.4.12]{MR2530988}.

Therefore, to see that Herzog's homeomorphism coincides with ours for a Noether algebra $A$, it is enough to show that each $T_{A}(P)$ is the elementary dual of $I_{A^{\op}}(P)$; this will be done in \cref{HomeoBwFlCotAndInjOverNoethAlg}. We also give an explicit proof for the reflexivity of points in $\inj_{A^{\op}}$ and $\flcot_{A}$.

For this purpose,  we describe the topology on $\inj_{A}$ in terms of prime ideals of $A$. 
The description is merely a paraphrase of known results.

\begin{definition}\label{SupportOverNoetherianAlgebra}
Let $A$ be a Noether $R$-algebra.
	For a right $A$-module $M$, define the \emph{support} of $M$ to be
	\begin{equation*}
		\Supp_{A}M:=\setwithcondition{P\in\Spec A}{\Hom_{A}(M,I_{A}(P))\neq 0}.
	\end{equation*}
\end{definition}

This support coincides with the classical one in commutative algebra. Indeed, by \cref{IndecomposableInjectiveLocal} and \cref{LocAdj},
\begin{equation}\label{eq.HomToIndecInj}
\Hom_A(M, I_{A}(P))\cong \Hom_{A_\p}(M_\p, I_{A}(P)),
\end{equation}
where $\p:= P\cap R$. If $A=R$, then $I_A(P)=E_R(R/\p)$ is an injective cogenerator in $\Mod R_\p$.
It should also be mentioned that \cref{SupportOverNoetherianAlgebra} just imitates the description of an open basis for $\inj_A$ given by Herzog and Krause; see \cref{HerzogKrause} below.

Let us state an auxiliary proposition. We say that a subset $\Phi\subset\Spec A$ is \emph{specialization-closed} (\resp \emph{generalization-closed}) if, for every pair $P\subset Q$ in $\Spec A$, $P\in\Phi$ implies $Q\in\Phi$ (\resp $Q\in\Phi$ implies $P\in\Phi$).

\begin{proposition}\label{SuppOfModOverNoethAlg}
	Let $A$ be a Noether $R$-algebra.
	\begin{enumerate}
		\item\label{SuppOfModOverNoethAlg.ExSeq} For every short exact sequence $0\to L\to M\to N\to 0$ of right $A$-modules,
		\begin{equation*}
			\Supp_{A}M=\Supp_{A}L\cup\Supp_{A}N.
		\end{equation*}
		\item\label{SuppOfModOverNoethAlg.Prime} For every $P\in\Spec A$, $\Supp_{A}(A/P)=\setwithcondition{Q\in\Spec A}{P\subset Q}$, which is the smallest specialization-closed subset of $\Spec A$ containing $P$.
		\item\label{SuppOfModOverNoethAlg.SpCl} For every right $A$-module $M$, $\Supp_{A}M$ is specialization-closed.
	\end{enumerate}
\end{proposition}

\begin{proof}
	\cref{SuppOfModOverNoethAlg.ExSeq}: Applying the exact functor $\Hom_{A}(-,I_{A}(P))$, for each $P\in\Spec A$, to the given short exact sequence, we obtain the result.
	
	\cref{SuppOfModOverNoethAlg.Prime}: Let $Q\in\Spec A$. First assume that $P\subset Q$. We have canonical morphisms $A/P\onto A/Q\into E_{A}(A/Q)$. Since $E_{A}(A/Q)$ is a finite direct sum of copies of $I_{A}(Q)$ by \cref{InjectiveEnvelopeDecomposition}, there exists a nonzero morphism $A/P\to I_{A}(Q)$. Thus $Q\in\Supp_{A}(A/P)$.
	
	Conversely, assume that $Q\in\Supp_{A}(A/P)$. Then there exists a nonzero morphism $f\colon A/P\to E_{A}(A/Q)$. Since $A/Q$ is an essential submodule of $E_{A}(A/Q)$, the intersection $\Im f\cap(A/Q)$ is nonzero. This means that $A/Q$ has a nonzero submodule annihilated by $P$. Therefore $P\subset Q$ by the definition of prime ideals.
	
	\cref{SuppOfModOverNoethAlg.SpCl}: Let $P\subset Q$ in $\Spec A$ and $P\in\Supp_{A}M$. Then there exists a nonzero morphism $g\colon M\to I_{A}(P)$. Let $N:=\Im g$. Since $I_A(P)$ is $\p$-local by \cref{IndecomposableInjectiveLocal}, $N_{\p}$ is a nonzero $A_{\p}$-submodule of $I_A(P)=E_A(S_A(P))$, and hence $N_{\p}$ contains $S_{A}(P)$ as an $A_{\p}$-submodule. Thus, by \cite[Lemma~2.5.1]{MR1898632}, there is a monomorphism from $A/P$ to a finite direct sum of copies of $N$. Therefore $Q\in\Supp_{A}(A/P)\subset\Supp_{A}N\subset\Supp_{A}M$ by \cref{SuppOfModOverNoethAlg.ExSeq} and \cref{SuppOfModOverNoethAlg.Prime}.
\end{proof}

\begin{remark}\label{HerzogKrause}
It is known (for any ring $A$) that there is a bijection from $\Zg_A$ to the set of isoclasses of indecomposable injective objects in $\fp(\mod A^\op, \Ab)$ given by $M\mapsto M\otimes_A-$ (\cite[Corollary~12.1.9]{MR2530988}). 
Extending this viewpoint, Herzog \cite{MR1434441} and Krause \cite{MR1426488} studied 
the spectrum formed by isoclasses of indecomposable injective objects for an arbitrary locally coherent Grothendieck category. In particular, when $A$ is a Noether $R$-algebra (or more generally, when $A$ is a right coherent ring), their work provides another way to think of $\inj_{A}$ as a topological space, with open basis consisting of all subsets of the form
\begin{equation*}
(M):=\setwithcondition{I\in\inj_{A}}{\Hom_{A}(M,I)\neq 0}
\end{equation*}
for some finitely presented right $A$-module $M$; see \cite[Corollary~3.5]{MR1434441} or \cite[Corollary~4.6]{MR1426488}. It follows from \cite[Theorem~5.1.11]{MR2530988} and \cite[Corollary~4.3]{MR1426488} that this topology coincides with the induced topology on $\inj_{A}$ as a (closed) subset of $\Zg_{A}$.
\end{remark}

\begin{proposition}\label{SpClAndOpenInInj}
	Let $A$ be a Noether $R$-algebra. There is an order-preserving bijection
	\begin{equation*}
		\setwithtext{specialization-closed subsets of $\Spec A$}\isoto\setwithtext{open subsets of $\inj_{A}$}
	\end{equation*}
	given by $\Phi\mapsto\setwithcondition{I_{A}(P)}{P\in\Phi}$.
\end{proposition}

\begin{proof}
	We show that the bijection in \cref{InjOverNoethRing} induces the desired bijection. By the above observation, $\inj_{A}$ (with topology induced from $\Zg_{A}$) has an open basis consisting all subsets of the form $(M)$ for some finitely presented right $A$-module $M$. Furthermore, each subset $(M)\subseteq \inj_A$ corresponds to $\Supp_{A}M$ by the bijection in \cref{InjOverNoethRing}. So it suffices to show that a subset $\Phi\subset\Spec A$ is specialization-closed if and only if $\Phi$ is the union of subsets of the form $\Supp_{A}M$ for some $M\in\mod A$. The ``if'' part follows from \cref{SuppOfModOverNoethAlg}\cref{SuppOfModOverNoethAlg.SpCl}. Conversely, if $\Phi$ is specialization-closed, then $\Phi=\bigcup_{P\in \Phi}\Supp_A (A/P)$ by \cref{SuppOfModOverNoethAlg}\cref{SuppOfModOverNoethAlg.Prime}.
\end{proof}

By \cref{HomeoBwFlCotAndInjOverNoethAlg,SpClAndOpenInInj}, we obtain an order-preserving bijection
	\begin{equation}\label{SpClAndOpenInFlCot}
		\setwithtext{specialization-closed subsets of $\Spec A$}\isoto\setwithtext{open subsets of $\flcot_{A}$}
	\end{equation}
	given by $\Phi\mapsto\setwithcondition{T_{A}(P)}{P\in\Phi}$.

The following is the main theorem in this section:

\begin{theorem}\label{FlCotAndInjAreRefl}
	Let $A$ be a Noether $R$-algebra. Then all points in $\inj_{A^{\op}}$ and $\flcot_{A}$ are reflexive. For each $P\in\Spec A$, the elementary dual of $I_{A^{\op}}(P)$ is $T_{A}(P)$.
\end{theorem}

\begin{proof}
By \cref{SpClAndOpenInInj} and \cref{SpClAndOpenInFlCot}, the generalization-closed subsets of $\Spec A$ bijectively correspond to the closed subsets of $\inj_{A^{\op}}$ and the closed subsets of $\flcot_A$. Let  $P\in \Spec A$, $I:=I_{A^{\op}}(P)\in\inj_{A^{\op}}$, and $T:=T_{A}(P)\in\flcot_{A}$.
The generalization-closed subset $\Psi:=\setwithcondition{Q\in\Spec A}{Q\subset P}$ corresponds to the closures $\overline{\set{I}}\subset\inj_{A^{\op}}$ and $\overline{\set{T}}\subset\flcot_{A}$,
and none of the proper generalization-closed subsets of $\Psi$ contains $P$.
Hence $I$ and $T$ are the unique generic points of $\overline{\set{I}}$ and $\overline{\set{T}}$, respectively. Consequently, $I$ and $T$ are reflexive.

It remains to show that $\Dc\overline{\set{I}}=\overline{\set{T}}$. 
This follows from the next lemma.
\end{proof}

\begin{lemma}
	Let $A$ be a Noether $R$-algebra.
For every closed subset $C\subseteq \inj_{A^\op}$, we have 
\begin{equation*}
\Dc C=\setwithcondition{T_A(P)\in \flcot_A }{\textnormal{$P\in\Spec A$, $I_{A^\op}(P) \in C$}}.
\end{equation*}
For every $P\in\Spec A$, it follows that $D\overline{\set{I_{A^{\op}}(P)}}=\overline{\set{T_{A}(P)}}$.
\end{lemma}

\begin{proof}
Since $\inj_{A^{\op}}$ is closed in $\Zg_{A^\op}$, the subset $C$ is also closed in $\Zg_{A^\op}$ and $\Dc C\subseteq\Dc(\inj_{A^{op}})=\flcot_A$ by \cref{DescriptionOfElementaryDual}. 
Moreover, \cref{FlCotAndIndecInjAndOpen} implies that 
$I_{A^\op}(P)\in \inj_{A^\op}\setminus C$ if and only if $T_{A}(P)\in \flcot_{A}\setminus \Dc C$ for each $P\in \Spec A$. In other words, $I_{A^\op}(P)\in C$ if and only if $T_{A}(P)\in \Dc C$ for each $P\in \Spec A$.
Thus we obtain the desired description of $\Dc C$.
The last statement of the lemma follows because \cref{SpClAndOpenInInj} and \cref{SpClAndOpenInFlCot} show that 
the closures of $I_{A^\op}(P)$ and $T_A(P)$ both correspond to the generalization closure of $P$.
\end{proof}

\begin{example}\label{ex.FlCotOfTriMatAlg}
	Consider the algebra $A$ in \cref{ex.TriMatAlg}. For two prime ideals $P_{i}(\p)$ and $P_{j}(\q)$ of $A$, we have $P_{i}(\p)\subset P_{j}(\q)$ if and only if $i=j$ and $\p\subset\q$. So we have an order-preserving bijection from $\Spec A$ to the disjoint union $\Spec R\amalg\Spec R$ given by $P_{i}(\p)\mapsto (\text{$\p$ in the $i$th $\Spec R$})$. Every specialization-closed subset of $\Spec A$ is of the form $\Phi_1\amalg \Phi_2$, where each $\Phi_{i}$ is a specialization-closed subset of the $i$th $\Spec R$.
Hence, by \cref{SpClAndOpenInFlCot}, all open subsets of $\flcot_{A}$ are of the form
	\begin{equation*}
		\setwithcondition{T_{A}(P_{1}(\p))}{\p\in\Phi_{1}}\cup\setwithcondition{T_{A}(P_{2}(\p))}{\p\in\Phi_{2}},
	\end{equation*}
	where $\Phi_{1}$ and $\Phi_{2}$ are specialization-closed subsets of $\Spec R$. The closure of each $T_{A}(P_{i}(\q))$ in $\flcot_{A}$ is
	\begin{equation*}
		\setwithcondition{T_{A}(P_{i}(\p))}{\q\subset\p}.
	\end{equation*}
\end{example}

Although \cref{SpClAndOpenInFlCot} describes the induced topology on $\flcot_A$ explicitly, it is also possible to give an open basis for $\flcot_A$ in a similar way to \cref{HerzogKrause}:

\begin{proposition}\label{OpenBasisForFlCot}
	The set of subsets of $\flcot_{A}$ of the form
	\begin{equation*}
		\setwithcondition{T_A(P) \in \flcot_{A}}{T_{A}(P)\otimes_AM\neq 0}
	\end{equation*}
	for some finitely generated left $A$-module $M$ is an open basis for $\flcot_{A}$.
\end{proposition}

\begin{proof}
	Recall that $I_{A^\op}(P)\cong \Hom_{\widehat{R_\p}}(T_{A}(P), E_R(R/\p))$, where $\p:=P\cap R$ (\cref{DualOfFlCov}). Using this isomorphism and the tensor-hom adjunction, we obtain
\begin{equation*}
\Hom_{{A^\op}}(M, I_{A^\op}(P))\cong \Hom_{\widehat{R_\p}}(T_{A}(P)\otimes_A M, E_R(R/\p))
\end{equation*}
for every left $A$-module $M$.
Since $E_R(R/\p)\cong E_{\widehat{R_\p}}(\kappa(\p))$ is an injective cogenerator in $\Mod \widehat{R_\p}$, we have
\begin{equation*}
\Supp_{A^\op}M=\setwithcondition{P\in \Spec A}{T_{A}(P)\otimes_AM\neq 0}.
\end{equation*}
Thus the desired conclusion follows from \cref{SpClAndOpenInFlCot} and \cref{SuppOfModOverNoethAlg}.
\end{proof}

\appendix

\section{Ideal-adic completion}
\label{sec.IdealAdicComp}

Let $R$ be a commutative noetherian ring and $A$ a Noether $R$-algebra. This appendix provides basic facts on $\ka$-adic completion of right $A$-modules, where $\ka$ is an ideal of $R$. All results here are generalizations or restatements of known results for $R$. Although the proofs resemble those for the commutative case, we provide a precise proof to each result for the reader's sake.

We denote by $\Mod A$ (\resp $\mod A$) the category of all (\resp finitely generated) right $A$-modules, and interpret $\Mod A^{\op}$ as the category of all left $A$-modules, where $A^{\op}$ is the opposite ring. 
The \emph{$\ka$-adic completion functor}
	$\varLambda^{\ka}\colon\Mod A\to \Mod A$ is defined by 
	\begin{equation*}
	\varLambda^{\ka}:=\varprojlim_{n\geq 1}(-\otimes_R R/\ka^n).
	\end{equation*}
	The functor $\varLambda^{\ka}$ is often written as $(-)^\wedge_\ka$.
A right $A$-module $M$ is called \emph{$\ka$-complete} if the canonical morphism $M\to M^\wedge_\ka$ is an isomorphism.

We start with the following lemma, which follows from the Artin-Rees lemma over $R$ and an intersection property of a flat right $A$-module.

\begin{lemma}\label{FlTensorCompIsEx}
	Let $F$ be a flat right $A$-module and let $\ka\subset R$ be an ideal. Then the functor 
	\begin{equation*}(F\otimes_{A}-)_{\ka}^{\wedge}\colon\mod A^{\op}\to\Mod R
	\end{equation*}
	 is exact.
\end{lemma}

\begin{proof}
Let $0\to L \to M\to N\to 0$ be an exact sequence of finitely generated left $A$-modules.
This is sent by the functor $F\otimes_{A}-$ to an exact sequence of $R$-modules
\begin{equation*}
0\to F\otimes_A L \to F\otimes_A M\to F\otimes_A N\to 0.
\end{equation*}
We regard $L$ (\resp $F\otimes_{A}L$) as a submodule of $M$ (\resp $F\otimes_{A}M$). By \cite[Theorem~8.1]{MR1011461}, it is enough to see that the $\ka$-adic topology on $F\otimes_AL$ coincides with the topology induced from the $\ka$-adic topology on $F\otimes_AM$.

Let $n\geq 1$ be an integer. Since $F$ is flat, the inclusion $\ka^n M\hookrightarrow M$ induces a canonical injection $F\otimes_A (\ka^n M)\hookrightarrow F\otimes_A M$, and
\begin{equation}\label{EquationRsub}
F\otimes_A (\ka^n M)=\ka^n (F\otimes_A  M)
\end{equation}
 as $R$-submodules of $F\otimes_A M$.

By the Artin-Rees lemma \cite[Theorem~8.5]{MR1011461}, there is an integer $c>0$ such that 
\begin{equation*}
\ka^n L\subseteq (\ka^n M)\cap L \subseteq \ka^{n-c} L,
\end{equation*}
for every $n>c$.
Application of $F\otimes_{A}-$ to this sequence yields   
\begin{equation}\label{InclusionSequence}
F\otimes_A(\ka^n L) \subseteq F\otimes_A  ((\ka^n M)\cap L) \subseteq F\otimes_A(\ka^{n-c} L),
\end{equation}
where the middle term coincides with
\begin{equation*}
(F\otimes_A  (\ka^n M))\cap (F\otimes_AL)
\end{equation*}
because the exact functor $F\otimes_{A}-$ preserves intersections of submodules.
Hence, using \cref{EquationRsub}, we can rewrite \cref{InclusionSequence} as
\begin{equation*}
\ka^n  (F\otimes_A L) \subseteq (\ka^n(F\otimes_A M))\cap (F\otimes_AL) \subseteq \ka^{n-c}(F\otimes_AL),
\end{equation*}
and this shows that the $\ka$-adic topology on $F\otimes_AL$ coincides with the topology induced from the $\ka$-adic topology on $F\otimes_AM$, as desired.
\end{proof}

\begin{proposition}\label{FlTensorCompAsTensor}
	Let $F$ be a flat right $A$-module and let $\ka\subset R$ be an ideal. Then there is a canonical isomorphism
	\begin{equation*}
	F_{\ka}^{\wedge}\otimes_{A}-\isoto(F\otimes_{A}-)_{\ka}^{\wedge}
	\end{equation*}
	of functors $\mod A^{\op}\to\Mod R$.
\end{proposition}

\begin{proof}
By \cref{FlTensorCompIsEx}, the functor $(F\otimes_{A}-)_{\ka}^{\wedge}$ is right exact, so the Eilenberg-Watts theorem (\cite[Theorem~2]{MR118757}) gives a canonical isomorphism $(F\otimes_{A}A)_{\ka}^{\wedge}\otimes_{A}-\isoto(F\otimes_{A}-)_{\ka}^{\wedge}$.
The desired isomorphism follows from the canonical isomorphism $F\otimes_{A}A\isoto F$ of right $A$-modules.
\end{proof}

\begin{proposition}\label{FlCompIsFl}
Let $F$ be a flat right $A$-module and let $\ka\subset R$ be an ideal. Then $F_{\ka}^{\wedge}$ is a flat right $A$-module.
\end{proposition}

\begin{proof}
	By \cref{FlTensorCompIsEx,FlTensorCompAsTensor}, the functor $F_{\ka}^{\wedge}\otimes_{A}-$ is exact on $\mod A^{\op}$. This implies that $F_{\ka}^{\wedge}$ is a flat right $A$-module (see \cite[Proposition~I.10.6]{MR0389953}, for example).
\end{proof}

In the case where $A=R$, \cref{FlCompIsFl} was shown by Gruson and Raynaud \cite[Part~II, (2.4.2) and Proposition~2.4.3.1]{MR308104} when $\ka\subset R$ is a maximal ideal, and by Bartijn \cite[Chapter~1, Corollary~4.7]{Bartijn} for arbitrary $\ka$.
Another proof was given by Schenzel and Simon \cite[Theorem~2.4.4]{MR3838396}. See \cite[Theorem~1.6]{MR3826724} for a certain generalization to non-noetherian commutative rings. Our proof of \cref{FlCompIsFl} is essentially the same as Gabber and Ramero \cite[Lemma~7.1.6]{MR2004652} but the settings are different.

Schenzel and Simon \cite[Theorem~2.4.4]{MR3838396} also showed the flatness of $F^\wedge_\ka$ over $R^\wedge_\ka$. This will be generalized to Noether algebras in \cref{FlatnessOverCompletionOfA}.

The next two results are often used by experts implicitly.

\begin{lemma}\label{CompTensorRed}
	Let $\ka,\kb\subset R$ be ideals such that $\ka^{n}\subset\kb$ for some $n>0$ and  
	let $M$ be a right $A$-module. Then the canonical morphism $M\to M^\wedge_\ka$ induces an isomorphism $M\otimes_{R}(R/\kb)\isoto M_{\ka}^{\wedge}\otimes_{R}(R/\kb)$.
\end{lemma}

\begin{proof}
It suffices to prove this by regarding $M$ as just an $R$-module, so the proof can be found in \cite[Chapter~I, Theorem~3.1]{Bartijn} or \cite[Theorem~2.2.5]{MR1074178}, which deals with completion with respect to a finitely generated ideal of a (possibly non-noetherian) commutative ring. Note that, in \cite[Theorem~2.2.5]{MR1074178}, a result like \cref{FlTensorCompAsTensor} is implicitly used at the end of the proof. Another proof can be found in \cite[Theorem~2.2.2]{MR3838396}.
\end{proof}

\begin{proposition}\label{CompletionIdempotent}
Let $\ka$ be an ideal of $R$. Denote by $\eta\colon \id_{\Mod A}\to \varLambda^\ka$ the canonical morphism of functors $\Mod A\to\Mod A$. 
For every right $A$-module $M$, the morphisms 
$\varLambda^\ka (\eta M)\colon\varLambda^\ka M\to \varLambda^\ka \varLambda^\ka M$ 
and $\eta (\varLambda^\ka M)\colon\varLambda^\ka M\to \varLambda^\ka \varLambda^\ka M$ are isomorphisms. In particular, $\varLambda^\ka M=M^\wedge_\ka$ is $\ka$-complete.
\end{proposition}

\begin{proof}
\cref{CompTensorRed} applied to $\kb=\ka^{n}$ ($n\geq 1$) yields the isomorphism $f_{n}\colon M\otimes_{R}(R/\ka^{n})\isoto M_{\ka}^{\wedge}\otimes_{R}(R/\ka^{n})$ induced from the completion map $M\to M_{\ka}^{\wedge}$. This implies that $\varLambda^\ka(\eta(M))$ is an isomorphism.

Applying $-\otimes_{R}R/\ka^{n}$ to the canonical map $M_{\ka}^{\wedge}\to M\otimes_{R}(R/\ka^{n})$ appearing in the definition of the inverse limit, we obtain $g_{n}\colon M_{\ka}^{\wedge}\otimes_{R}(R/\ka^{n})\to M\otimes_{R}(R/\ka^{n})$. As mentioned in the proofs of \cite[Chapter~I, Proposition~2.3]{Bartijn} and \cite[Theorem~2.2.5]{MR1074178}, it is easy to see that $g_{n}f_{n}$ is the identity map, so $g_{n}=f_{n}^{-1}$ is also an isomorphism. One can also check that the composition
\begin{equation*}
	M_{\ka}^{\wedge}\to\varprojlim_{n\geq 1}M_{\ka}^{\wedge}\otimes_{R}(R/\ka^{n})\isoto\varprojlim_{n\geq 1}M\otimes_{R}(R/\ka^{n})=M_{\ka}^{\wedge}
\end{equation*}
of $\eta(\varLambda^\ka M)$ and the isomorphism induced by $(g_{n})_{n}$ is the identity map, so $\eta(\varLambda^\ka M)$ is also an isomorphism.
\end{proof}

\begin{remark}\label{bCompleteISaComplete}
Let $\ka$ and $\kb$ are ideals of $R$ with $\ka \subseteq\kb$.
Then every $\kb$-complete right $A$-module $M$ is $\ka$-complete.
Indeed, the composition of the completion maps $M\to M^\wedge_\ka$ and $M^\wedge_\ka\to (M^\wedge_\ka)^\wedge_\kb$ is an isomorphism since $(M^\wedge_\ka)^\wedge_\kb\cong M^\wedge_\kb$ by \cref{CompTensorRed}.
Thus $M$ is a direct summand of $M^\wedge_\ka$. This implies that $M$ is $\ka$-complete by \cref{CompletionIdempotent}.
\end{remark}

The functor $\varLambda^\ka\colon\Mod A\to \Mod A$ is not necessarily left exact or right exact (even if $A=R$; see \cite[Chapter~10, Exercise~1]{MR0242802}, for example) so it is not isomorphic to $-\otimes_AA^\wedge_\ka$. However, \cref{CompletionPreservesSurjectivity,CommutativityWithDirestProduct} below show some basic properties of $\varLambda^\ka$.

\begin{proposition}\label{CompletionPreservesSurjectivity}
	The functor $\varLambda^\ka\colon\Mod A\to\Mod A$ preserves epimorphisms.
\end{proposition}

\begin{proof}
Since this property is that for the functor $\varLambda^\ka\colon\Mod R\to\Mod R$, we may assume $A=R$. So the result follows from \cite[Theorem~8.1(ii)]{MR1011461} because the $\ka$-adic topology of a quotient module $M/N$ coincides with the topology induced from the $\ka$-adic topology of $M$.
\end{proof}

\begin{proposition}\label{CommutativityWithDirestProduct}
	The functor $\varLambda^\ka\colon\Mod A\to\Mod A$ commutes with arbitrary direct products.
\end{proposition}

\begin{proof}
The $R$-module $R/\ka^n R$ is finitely presented for each $n\geq 1$, so a standard argument shows that the functor $-\otimes_RR/\ka^n R\colon\Mod A\to \Mod A$ commutes with arbitrary direct products (see \cite[Theorem~3.2.22]{MR1753146}). Hence the functor $\varLambda^\ka=\varprojlim_{n\geq 1}(-\otimes_RR/\ka^n R)$ commutes with arbitrary direct products.
\end{proof}

Let $\ka$ be an ideal of $R$. The \emph{$\ka$-torsion functor} $\varGamma_\ka\colon\Mod A \to \Mod A$ is defined by
\begin{equation*}
\varGamma_\ka:=\varinjlim_{n\geq 1} \Hom_R(R/\ka^n,-).
\end{equation*}
A right $A$-module $M$ is called \emph{$\ka$-torsion} if the canonical morphism $\varGamma_\ka M\to M$ is an isomorphism.

It is well-known that the functor $\varGamma_\ka$ from $\Mod A$ to its full subcategory consisting of all $\ka$-torsion modules is a right adjoint to the inclusion functor. A similar result holds for $\varLambda^\ka$:

\begin{proposition}\label{CompAdj}
The functor $\varLambda^{\ka}$ from $\Mod A$ to its full subcategory consisting of all $\ka$-complete modules is a left adjoint to the inclusion functor.
\end{proposition}

\begin{proof}
This follows from \cref{CompletionIdempotent} and 
the general theory of categories; see \cite[Proposition~4.1.3(iii)]{MR2182076}.
\end{proof}

\cref{TorsionIsomorphism} and \cref{TorsionEquivalence} below are essentially stated in \cite[Remark~A.30(7) and (8)]{MR2355715} for the case $A=R$.

\begin{lemma}\label{TorsionIsomorphism}
Let $M$ be an $\ka$-torsion right $A$-module. Then the canonical morphism 
$M\to M\otimes_A A^\wedge_\ka$ is an isomorphism of right $A$-modules.
\end{lemma}

\begin{proof}
If $N$ is a right $A$-module such that $\ka^n N=0$ for some $n>0$, then $N\cong N\otimes_R R/\ka^n$, so 
\begin{equation*}
N\otimes_A A^\wedge_\ka\cong (N\otimes_R R/\ka^n)\otimes_A A^\wedge_\ka\cong N\otimes_A (A^\wedge_\ka \otimes_R R/\ka^n)\cong N\otimes_A (A \otimes_R R/\ka^n)\cong N
\end{equation*}
 as right $A$-modules, where the third isomorphism follows from \cref{CompTensorRed}.

Now, if $M$ is $\ka$-torsion, then $M$ is canonically isomorphic to $\varinjlim_{n\geq 1} \Hom_R(R/\ka^n,M)$. The above argument shows that each $\Hom_R(R/\ka^n,M)$ satisfies the property in the statement. Since $-\otimes_AA^\wedge_\ka$ commutes with direct limits, so does $M$.
\end{proof}

Note that a right $A^\wedge_\ka$-module is $\ka A^\wedge_\ka$-complete (\resp $\ka A^\wedge_\ka$-torsion) if and only if it is $\ka$-complete (\resp $\ka$-torsion) as a right $A$-module.

\begin{proposition}\label{TorsionEquivalence}
The functor $-\otimes_A A^\wedge_\ka\colon\Mod A\to \Mod A^\wedge_\ka$ induces an equivalence from 
the full subcategory of $\ka$-torsion right $A$-modules to the full subcategory of $\ka$-torsion right $A^\wedge_\ka$-modules. Its quasi-inverse is given by the scalar restriction functor.
\end{proposition}

\begin{proof}
If $M$ is an $\ka$-torsion right $A$-module, then we have the canonical isomorphism $M\isoto M\otimes_A A^\wedge_\ka$ of right $A$-modules by  \cref{TorsionIsomorphism}, and this means that 
the composition of $-\otimes_A A^\wedge_\ka\colon\Mod A\to \Mod A^\wedge_\ka$ and the scalar restriction functor $ \Mod A^\wedge_\ka\to \Mod A$ induces an autoequivalence on 
the full subcategory of $\ka$-torsion right $A$-modules.

Let $N$ be an $\ka$-torsion right $A^\wedge_\ka$-module. 
We only need to check that the canonical morphism 
$N\otimes_A A^\wedge_\ka\to N$ of right $A^\wedge_\ka$-modules is an isomorphism. This also follows from \cref{TorsionIsomorphism} because the composition of the canonical maps $N\isoto N\otimes_A A^\wedge_\ka\to N$ is the identity map.
\end{proof}

\begin{remark}\label{TwoActionOfCompletion}
For a right $A$-module $M$, its $\ka$-adic completion $M^\wedge_\ka$ is naturally realized as a right $A$-submodule of $\prod_{n\geq 1} M/\ka^n M$. In particular, we may interpret $A^\wedge_\ka$ as a subring of $\prod_{n\geq 1} A/\ka^n A$. So the componentwise action defines a canonical right $A^\wedge_\ka$-module structure on $M^\wedge_\ka$.
Moreover, taking the $\ka$-adic completion sends each $A$-homomorphism $M\to N$ an $A^\wedge_\ka$-homomorphism $M^\wedge_\ka\to N^\wedge_\ka$, so 
we may regard $(-)^\wedge_\ka$ as a functor $\Mod A\to \Mod A^\wedge_\ka$.

For a finitely generated right $A$-module $M$, \cref{FlTensorCompAsTensor} gives a canonical isomorphism $M \otimes_A A^\wedge_\ka\to M^\wedge_\ka$ of right $A$-modules. It is easily seen from the proof that this is an isomorphism of right $A^\wedge_\ka$-modules as well.
\end{remark}

\begin{proposition}\label{FlatnessOverCompletionOfA}
For every flat right $A$-module $F$, its $\ka$-adic completion $F^\wedge_\ka$ is a flat right $A^\wedge_\ka$-module.
\end{proposition}

\begin{proof}
If $L$ is an $\ka$-torsion left $A^\wedge_\ka$-module, then $L\cong A^\wedge_\ka\otimes_A L$ as left $A^\wedge_\ka$-modules by \cref{TorsionEquivalence}, so
\begin{equation}\label{TorsionTensor}
-\otimes_{A^\wedge_\ka}L \cong -\otimes_{A^\wedge_\ka}(A^\wedge_\ka\otimes_A L)\cong -\otimes_A L
\end{equation}
as functors $\Mod A^\wedge_\ka\to \Mod R^\wedge_\ka$.
Similarly to the proof of \cref{FlTensorCompIsEx}, we show that the functor 
\begin{equation*}
(F^\wedge_{\ka}\otimes_{A^\wedge_\ka}-)_{\ka}^{\wedge}\colon\mod {A^\wedge_\ka}^{\op}\to \Mod R^\wedge_\ka
\end{equation*}
is exact.
By \cref{CompTensorRed} and \cref{TorsionTensor},
\begin{equation*}
(F^\wedge_{\ka}\otimes_{A^\wedge_\ka}-)^\wedge_\ka=\varprojlim_{n\geq 1}(F^\wedge_{\ka}\otimes_{A^\wedge_\ka}-)\otimes_{R} (R/\ka^n)\cong \varprojlim_{n\geq 1}(F\otimes_{A}-)\otimes_{R} (R/\ka^n)=(F\otimes_{A}-)^\wedge_\ka
\end{equation*}
as functors  $\mod {A^\wedge_\ka}^{\op}\to \Mod R^\wedge_\ka$.
The exactness of the functor $(F\otimes_{A}-)^\wedge_\ka$ can be shown in the same way as \cref{FlTensorCompIsEx}, using the Artin-Rees lemma for finitely generated left $R^\wedge_\ka$-modules.
Hence the functor $(F^\wedge_{\ka}\otimes_{A^\wedge_\ka}-)^\wedge_\ka$ is also exact. In the same way as in the proofs of \cref{FlTensorCompAsTensor,FlCompIsFl}, we have $F^\wedge_{\ka}\otimes_{A^\wedge_\ka}-\cong (F^\wedge_{\ka}\otimes_{A^\wedge_\ka}-)^\wedge_\ka$, so  $F^\wedge_{\ka}$ is a flat right $A^\wedge_\ka$-module.
\end{proof}

The following result is analogous to \cref{TorsionEquivalence}.

\begin{proposition}\label{CompleteEquivalence}
The functor $(-)^\wedge_\ka\colon\Mod A \to \Mod A^\wedge_\ka$ induces an equivalence from the full subcategory of $\ka$-complete right $A$-modules to the full subcategory of $\ka$-complete right $A^\wedge_\ka$-modules. Its quasi-inverse is given by the scalar restriction functor.
\end{proposition}

\begin{proof}
If $M$ is an $\ka$-complete right $A$-module, then by definition we have the canonical isomorphism $M\isoto M^\wedge_\ka$ of right $A$-modules, and this means that 
the composition of $\varLambda^\ka\colon\Mod A\to \Mod A^\wedge_\ka$ and the scalar restriction functor $ \Mod A^\wedge_\ka\to \Mod A$ induces an autoequivalence on 
the full subcategory of $\ka$-complete right $A$-modules.

Let $N$ be an $\ka$-complete right $A^\wedge_\ka$-module. Then, by definition, we have an isomorphism $N\to N^{\wedge}_{\ka}$ of right $A$-modules. We show that this is an isomorphism of right $A^\wedge_\ka$-module, where the $A^\wedge_\ka$-module structure on $N^{\wedge}_{\ka}$ is the one defined in \cref{TwoActionOfCompletion}. The embedding $f\colon N\hookrightarrow \prod_{n\geq 1} N/\ka^n N$ induced by the projections $N\onto N/\ka^{n}N$ is an $A^\wedge_\ka$-homomorphism if we regard $\prod_{n\geq 1} N/\ka^n N$ as the product of right $A^\wedge_\ka$-modules $N/\ka^n N$. On the other hand, we observed in \cref{TwoActionOfCompletion} that the natural embedding $N^{\wedge}_{\ka}\into\prod_{n\geq 1} N/\ka^n N$ is an $A^\wedge_\ka$-homomorphism, but here $A^\wedge_\ka$ acts on the product componentwise. Since these embeddings are identified via the isomorphism $N\to N^{\wedge}_{\ka}$, it suffices to prove that those two $A^\wedge_\ka$-module structures on $\prod_{n\geq 1} N/\ka^n N$ coincides. In other words, it suffices to prove that, for each $n\geq 1$, the $A^\wedge_\ka$-module structure on $N/\ka^{n}N$ induced from that of $N$ is the same as the $A^\wedge_\ka$-module structure on $N/\ka^{n}N$ obtained from the $A/\ka^{n}A$-module structure of $N/\ka^{n}N$ via the canonical map $A^\wedge_\ka\to A/\ka^{n}A$. The former structure factors through the $A^\wedge_\ka/\ka^{n}A^\wedge_\ka$-structure on $N/\ka^{n}N$. As we observed in the proof of \cref{CompletionIdempotent}, the map $A^\wedge_\ka\to A/\ka^{n}A$ induces an isomorphism $A^\wedge_\ka/\ka^{n}A^\wedge_\ka\isoto A/\ka^{n}A$. So we have a commutative diagram
\begin{equation*}
	\begin{tikzcd}
		A\ar[d]\ar[r,twoheadrightarrow] & A/\ka^{n}A \\
		A^\wedge_\ka\ar[ur]\ar[r,twoheadrightarrow] & A^\wedge_\ka/\ka^{n}A^\wedge_\ka\ar[u,"\wr"']\rlap{,}
	\end{tikzcd}
\end{equation*}
where all maps are canonical ones. This means that each $A^\wedge_\ka$-module structures on $N/\ka^{n}N$ is determined by the induced $A$-module structure on $N/\ka^{n}N$ via the canonical map $A\to A^\wedge_\ka$. Since the induced $A$-module structures are the same, so are the $A^\wedge_\ka$-module structures. This completes the proof.
\end{proof}

The following fact is shown in \cite[Proposition~2.1.15(a)]{MR3838396} for $\ka$-torsion $R$-modules.

\begin{proposition}\label{UniqueCompleteStructure}
Every $\ka$-torsion (\resp $\ka$-complete) right $A$-module has a unique right $A^\wedge_\ka$-module structure that is compatible with the right $A$-module structure via the canonical map $A\to A^\wedge_\ka$.
\end{proposition}

\begin{proof}
This follows from \cref{TorsionEquivalence} (\resp \cref{CompleteEquivalence}). Indeed, such a structure exists since every $\ka$-torsion (\resp $\ka$-complete) right $A$-module $M$ belongs to the essential image of the scalar restriction functor $\Mod A^\wedge_\ka\to\Mod A$. If $N_{1}$ and $N_{2}$ are right $A^\wedge_\ka$-modules that are equal to $M$ as right $A$-modules, then they are $\ka$-torsion (\resp $\ka$-complete), and the scalar restriction functor gives a bijection $\Hom_{A^\wedge_\ka}(N_{1},N_{2})\to\Hom_{A}(M,M)$. Therefore the identity map $M\to M$ gives the equality of $N_{1}$ and $N_{2}$ as right $A^\wedge_\ka$-modules.
\end{proof}

Assume that an ideal $\ka\subset R$ is contained by the Jacobson radical of $R$ (which is by definition the intersection of all maximal ideals of $R$). Then the ring homomorphism $R\to R^\wedge_\ka$ is faithfully flat (\cite[Theorem~8.14]{MR1011461}), and hence
the induced map $\Spec R^\wedge_\ka\to \Spec R$ is surjective by \cite[Theorem~7.3(i)]{MR1011461}.

It is natural to ask whether this holds for a Noether $R$-algebra $A$. Under the same assumption on $\ka\subset R$, it follows that the canonical ring homomorphism $A\to A^\wedge_\ka$ is a pure monomorphism in $\Mod A$ (since $R\to R^\wedge_\ka$ is a pure monomorphism by \cite[Theorem~7.5(i)]{MR1011461} and $A^\wedge_\ka=A\otimes_{R}R^\wedge_\ka$). Thus the following proposition gives an affirmative answer to the question:

\begin{proposition}\label{SurjBwSpec}
	Let $R$ be a commutative ring and let $A$ and $B$ be rings. Let $\phi\colon R\to A$ and $f\colon A\to B$ be ring homomorphisms such that $\phi(R)$ and $f(\phi(R))$ are contained in the centers of $A$ and $B$, respectively (that is, $f$ is an $R$-algebra homomorphism). Assume that $A$ is finitely generated as an $R$-module, $B$ is a centralizing extension of $f(A)$, and $f$ is a pure monomorphism in $\Mod A$. Then the induced map $\Spec B\to\Spec A$ is surjective.
\end{proposition}
\begin{proof}
	Let $P\in\Spec A$ and $\p:=P\cap R$, which belongs to $\Spec R$ by \cref{MapOfSpectra}. Since $B$ is a centralizing extension of $f(A)$, $BP=PB$ is a (two-sided) ideal of $B$. We have a commutative diagram
	\begin{equation*}
		\begin{tikzcd}
			A\ar[d]\ar[r,hookrightarrow,"f"] & B\ar[d] \\
			A_\p/P_\p\ar[r,hookrightarrow] & B_\p/PB_\p
		\end{tikzcd}
	\end{equation*}
	of canonical ring homomorphisms, where the second horizontal map is injective since it can be identified with $f\otimes_A (A_\p/P_\p)$ and $f$ is a pure monomorphism in $\Mod A$. By \cref{NonnoetherianSpectra,MapOfSpectra}, the diagram induces the following commutative diagram of maps:
	\begin{equation*}
		\begin{tikzcd}
			\Spec A & \Spec B\ar[l] \\
			\Spec(A_\p/P_\p)\ar[u] & \Spec(B_\p/PB_\p)\ar[u]\ar[l]\rlap{.}
		\end{tikzcd}
	\end{equation*}
	Since the ring $B_\p/PB_\p$ is nonzero, it has at least one maximal (hence prime) ideal $Q$. By \cref{SpecOfNoethAlg} along with \cref{NonnoetherianSpectra}, $\Spec(A_\p/P_\p)=\Spec ((A/P)\otimes_{R}\kappa(\p))$ consists of only the zero ideal, and it is sent to $P\in\Spec A$ by the left vertical map in the diagram. By the commutativity of the last diagram, the image of $Q$ in $\Spec B$ is sent to $P$ by the map $\Spec B\to\Spec A$.
\end{proof}

Let $f\colon A\to A^\wedge_\ka$ be the canonical ring homomorphism. Then the induced map $\Spec A^\wedge_\ka\to \Spec A$ given by $Q\mapsto f^{-1}(Q)$ is surjective by \cref{SurjBwSpec}.
The next proposition shows that, when $R$ is local and $\ka$ is its maximal ideal, the correspondence of maximal ideals can be understood well. The completion functor with respect to the maximal ideal of $R$ is written as $\widehat{(-)}$.

\begin{proposition}\label{CompletionFiber}
Assume that $(R,m,k)$ is a commutative noetherian local ring, and let $f\colon A\to \widehat{A}$ be the canonical ring homomorphism.
\begin{enumerate}
\item\label{CompletionFiber.MaximalInverse} Let $I\subset\widehat{A}$ be an ideal. Then $I\in\Max\widehat{A}$ if and only if $f^{-1}(I)\in\Max A$. If this is the case, then $I=\widehat{f^{-1}(I)}$, and $f\colon A\to\widehat{A}$ induces an isomorphism $A/f^{-1}(I)\isoto\widehat{A}/I$ of rings.

\item\label{CompletionFiber.MaximalBijective} The canonical surjection $\Spec \widehat{A}\to \Spec A$ restricts to a bijection $\Max \widehat{A} \isoto \Max A$ between the sets of maximal ideals, and $\Max \widehat{A}= \setwithcondition{\widehat{P}}{P\in \Max A}$.

\item\label{CompletionFiber.SimpleInjective} For every $P\in \Max A$, we have isomorphisms $S_{A}(P)\cong S_{\widehat{A}}(\widehat{P})$ and $I_{A}(P)\cong I_{\widehat{A}}(\widehat{P})$ in $\Mod\widehat{A}$ (and also in $\Mod A$).
\end{enumerate}
\end{proposition}

\begin{proof}
\cref{CompletionFiber.MaximalInverse}: Let $J:=f^{-1}(I)$. We have a commutative diagram
\begin{equation*}
	\begin{tikzcd}
		R/(J\cap R)\ar[d,hookrightarrow]\ar[r,hookrightarrow] & \widehat{R}/(I\cap\widehat{R})\ar[d,hookrightarrow] \\
		A/J\ar[r,hookrightarrow,"\overline{f}"] & \widehat{A}/I\rlap{,}
	\end{tikzcd}
\end{equation*}
in which all maps are canonical ones.
If $J\in\Max A$, then $J\cap R\in\Max R$ by \cref{SpecOfNoethAlgMax}, and hence $J\cap R=\km$ and $R/(J\cap R)=k$. This means that $A/J$ is a finite-dimensional $k$-algebra. In particular, $A/J$ is of finite length as an $R$-module, so it is $\km$-complete. On the other hand, $\widehat{A}/I$ is also $\km$-complete as it is finitely generated $\widehat{R}$-module (see the third paragraph of \cref{subsec.MatlisDual}). Thus $\overline{f}$ is canonically identified with its completion $\varLambda^{\km}\overline{f}$. By \cref{FlTensorCompIsEx}, $\varLambda^\km(A/J)\cong \widehat{A}/\widehat{J}$, so $\varLambda^{\km}\overline{f}$ is the ring homomorphism $\widehat{A}/\widehat{J}\to\widehat{A}/I$, which is surjective. 
It then follows that $\overline{f}=\varLambda^{\km}\overline{f}$ is an isomorphism and $\widehat{J}=I$. 
The isomorphism $\overline{f}:A/J\isoto\widehat{A}/I$ of rings implies that $I$ is a maximal ideal of $\widehat{A}$ since $J$ is maximal. This proves the ``if'' part of the first claim and the second claim of \cref{CompletionFiber.MaximalInverse}.

Conversely, if $I\in\Max\widehat{A}$, then $I\cap\widehat{R}=\widehat{\km}\in\Max\widehat{R}$ by \cref{SpecOfNoethAlgMax}.
Since the preimage of $I\cap \widehat{R}$ by the canonical map $R\to \widehat{R}$ is $J\cap R$ (see the commutative diagram above), we have $J\cap R=\km\in \Max R$. Thus $J\in\Max A$ by \cref{SpecOfNoethAlgMax} again.

\cref{CompletionFiber.MaximalBijective}: 
It follows from the first claim of \cref{CompletionFiber.MaximalInverse} that the canonical surjection $\Spec \widehat{A}\to\Spec A$ restricts to a surjection $\Max \widehat{A}\to\Max A$, and this must be injective by the second claim of \cref{CompletionFiber.MaximalInverse}.

\cref{CompletionFiber.SimpleInjective}:
For every $P\in \Max A$, $A/P$ is a finite direct sum of copies of $S_{A}(P)$ as a right $A$-module; see \cref{nP}. By \cref{TorsionEquivalence}, this decomposition can be regarded as that of right $\widehat{A}$-modules, and $S_{A}(P)$ is a simple $\widehat{A}$-module. Since $A/P\cong \widehat{A}/\widehat{P}$ by \cref{CompletionFiber.MaximalInverse} and \cref{CompletionFiber.MaximalBijective}, and $\widehat{A}/\widehat{P}$ is a finite direct sum of copies of $S_{\widehat{A}}(\widehat{P})$ as a right $\widehat{A}$-module, we obtain an isomorphism $S_{A}(P)\cong S_{\widehat{A}}(\widehat{P})$ of right $\widehat{A}$-modules.

By this isomorphism, the injective envelope $E_{\widehat{A}}(S_A(P))$ of $S_A(P)$ coincides with the injective envelope $I_{\widehat{A}}(\widehat{P})=E_{\widehat{A}}(S_{\widehat{A}}(\widehat{P}))$ of $S_{\widehat{A}}(\widehat{P})$ in $\Mod \widehat{A}$. 
As $I_{\widehat{A}}(\widehat{P})$ is $\widehat{\km}$-torsion (\cref{ArtinianInjective}), it is $\km$-torsion, so the essential extension $S_A(P)\hookrightarrow E_{\widehat{A}}(S_A(P))\cong I_{\widehat{A}}(\widehat{P})$ in $\Mod \widehat{A}$ is also an essential extension in $\Mod A$ by \cref{TorsionEquivalence}.
Moreover, $E_{\widehat{A}}(S_A(P))$ is injective as a right $A$-module, because $\widehat{A}$ is a flat left $A$-module by \cref{FlCompIsFl} and
\begin{equation*}
	\Hom_A(-,E_{\widehat{A}}(S_A(P)))\cong\Hom_A(-,\Hom_{\widehat{A}}(\widehat{A},E_{\widehat{A}}(S_A(P))))\cong \Hom_{\widehat{A}}(-\otimes_A \widehat{A},E_{\widehat{A}}(S_A(P)))
\end{equation*}
by the tensor-hom adjunction. 
Therefore $S_A(P)\hookrightarrow E_{\widehat{A}}(S_A(P))$ is an injective envelope in $\Mod A$ as well, and hence $I_A(P)=E_{A}(S_A(P))\cong E_{\widehat{A}}(S_A(P))=I_{\widehat{A}}(\widehat{P})$ in $\Mod A$. This is also an isomorphism in $\Mod \widehat{A}$ by \cref{TorsionEquivalence}.
\end{proof}

\bibliography{bib}

\newcommand{\etalchar}[1]{$^{#1}$}
\providecommand{\bysame}{\leavevmode\hbox to3em{\hrulefill}\thinspace}
\providecommand{\MR}{\relax\ifhmode\unskip\space\fi MR }
\providecommand{\MRhref}[2]{%
  \href{http://www.ams.org/mathscinet-getitem?mr=#1}{#2}
}
\providecommand{\href}[2]{#2}
\begin{thebibliography}{AHMH09}

\bibitem[AB89]{MR1044344}
Maurice Auslander and Ragnar-Olaf Buchweitz, \emph{The homological theory of
  maximal {C}ohen-{M}acaulay approximations}, Colloque en l'honneur de Pierre
  Samuel (Orsay, 1987), M\'{e}m. Soc. Math. France (N.S.) (1989), no.~38,
  5--37. \MR{1044344}

\bibitem[AF92]{MR1245487}
Frank~W. Anderson and Kent~R. Fuller, \emph{Rings and categories of modules},
  second ed., Graduate Texts in Mathematics, vol.~13, Springer-Verlag, New
  York, 1992. \MR{1245487}

\bibitem[AHMH09]{MR2584945}
Lidia Angeleri~H\"{u}gel and Octavio Mendoza~Hern\'{a}ndez, \emph{Homological
  dimensions in cotorsion pairs}, Illinois J. Math. \textbf{53} (2009), no.~1,
  251--263. \MR{2584945}

\bibitem[AM69]{MR0242802}
M.~F. Atiyah and I.~G. Macdonald, \emph{Introduction to commutative algebra},
  Addison-Wesley Publishing Co., Reading, Mass.-London-Don Mills, Ont., 1969.
  \MR{0242802}

\bibitem[AR91]{MR1097029}
Maurice Auslander and Idun Reiten, \emph{Applications of contravariantly finite
  subcategories}, Adv. Math. \textbf{86} (1991), no.~1, 111--152. \MR{1097029}

\bibitem[AS81a]{MR630621}
M.~Auslander and Sverre~O. Smal\o, \emph{Addendum to ``{A}lmost split sequences
  in subcategories''}, J. Algebra \textbf{71} (1981), no.~2, 592--594.
  \MR{630621}

\bibitem[AS81b]{MR617088}
M.~Auslander and Sverre~O. Smal\o, \emph{Almost split sequences in
  subcategories}, J. Algebra \textbf{69} (1981), no.~2, 426--454. \MR{617088}

\bibitem[Bar85]{Bartijn}
Jacob Bartijn, \emph{Flatness, completion, regular sequences: un m{\'e}nage
  {\`a} trois}, Ph{D} thesis, 1985.

\bibitem[BBOS20]{MR4091895}
Ulrich Bauer, Magnus~B. Botnan, Steffen Oppermann, and Johan Steen,
  \emph{Cotorsion torsion triples and the representation theory of filtered
  hierarchical clustering}, Adv. Math. \textbf{369} (2020), 107171, 51.
  \MR{4091895}

\bibitem[BCIE20]{MR4140057}
Silvana Bazzoni, Manuel Cort\'{e}s-Izurdiaga, and Sergio Estrada,
  \emph{Periodic modules and acyclic complexes}, Algebr. Represent. Theory
  \textbf{23} (2020), no.~5, 1861--1883. \MR{4140057}

\bibitem[BEBE01]{MR1832549}
L.~Bican, R.~El~Bashir, and E.~Enochs, \emph{All modules have flat covers},
  Bull. London Math. Soc. \textbf{33} (2001), no.~4, 385--390. \MR{1832549}

\bibitem[BH98]{MR1251956+}
Winfried Bruns and J\"{u}rgen Herzog, \emph{Cohen-{M}acaulay rings}, revised
  ed., Cambridge Studies in Advanced Mathematics, vol.~39, Cambridge University
  Press, Cambridge, 1998. \MR{1251956}

\bibitem[Buc86]{Buc86}
Ragnar-Olaf Buchweitz, \emph{Maximal {C}ohen-{M}acaulay modules and
  {T}ate-cohomology over {G}orenstein rings}, unpublished manuscript, available
  at http://hdl.handle.net/1807/16682.

\bibitem[CET20]{MR4132086}
Lars~Winther Christensen, Sergio Estrada, and Peder Thompson, \emph{Homotopy
  categories of totally acyclic complexes with applications to the
  flat--cotorsion theory}, Categorical, homological and combinatorial methods
  in algebra, Contemp. Math., vol. 751, Amer. Math. Soc., Providence, RI, 2020,
  pp.~99--118. \MR{4132086}

\bibitem[Cha60]{MR120260}
Stephen~U. Chase, \emph{Direct products of modules}, Trans. Amer. Math. Soc.
  \textbf{97} (1960), 457--473. \MR{120260}

\bibitem[Dau94]{MR1301329}
John Dauns, \emph{Modules and rings}, Cambridge University Press, Cambridge,
  1994. \MR{1301329}

\bibitem[DK19]{arXiv:1912.07117}
Christopher~M. Drupieski and Jonathan~R. Kujawa, \emph{Support varieties and
  modules of finite projective dimension for modular {L}ie
  superalgebras\textup{:} with an appendix on homological dimensions over
  {N}oether {A}lgebras by {L}uchezar {L}. {A}vramov and {S}rikanth {B}.
  {I}yengar}, arXiv:1912.07117v1.

\bibitem[EJ00]{MR1753146}
Edgar~E. Enochs and Overtoun M.~G. Jenda, \emph{Relative homological algebra},
  De Gruyter Expositions in Mathematics, vol.~30, Walter de Gruyter \& Co.,
  Berlin, 2000. \MR{1753146}

\bibitem[Eno81]{MR636889}
Edgar~E. Enochs, \emph{Injective and flat covers, envelopes and resolvents},
  Israel J. Math. \textbf{39} (1981), no.~3, 189--209. \MR{636889}

\bibitem[Eno84]{MR754698}
Edgar Enochs, \emph{Flat covers and flat cotorsion modules}, Proc. Amer. Math.
  Soc. \textbf{92} (1984), no.~2, 179--184. \MR{754698}

\bibitem[Fuc70]{MR0255673}
L\'{a}szl\'{o} Fuchs, \emph{Infinite abelian groups. {V}ol. {I}}, Pure and
  Applied Mathematics, Vol. 36, Academic Press, New York-London, 1970.
  \MR{0255673}

\bibitem[GAH07]{MR2377125}
Pedro~A. Guil~Asensio and Ivo Herzog, \emph{Indecomposable flat cotorsion
  modules}, J. Lond. Math. Soc. (2) \textbf{76} (2007), no.~3, 797--811.
  \MR{2377125}

\bibitem[Gil04]{MR2052954}
James Gillespie, \emph{The flat model structure on {${\rm Ch}(R)$}}, Trans.
  Amer. Math. Soc. \textbf{356} (2004), no.~8, 3369--3390. \MR{2052954}

\bibitem[Gil11]{MR2811572}
James Gillespie, \emph{Model structures on exact categories}, J. Pure Appl.
  Algebra \textbf{215} (2011), no.~12, 2892--2902. \MR{2811572}

\bibitem[Gil17]{MR3623182}
James Gillespie, \emph{The flat stable module category of a coherent ring}, J.
  Pure Appl. Algebra \textbf{221} (2017), no.~8, 2025--2031. \MR{3623182}

\bibitem[GN02]{MR1898632}
Shiro Goto and Kenji Nishida, \emph{Towards a theory of {B}ass numbers with
  application to {G}orenstein algebras}, Colloq. Math. \textbf{91} (2002),
  no.~2, 191--253. \MR{1898632}

\bibitem[GPGA00]{MR1758412}
Jos\'{e}~L. G{\'{o}}mez~Pardo and Pedro~A. Guil~Asensio, \emph{Chain conditions
  on direct summands and pure quotient modules}, Interactions between ring
  theory and representations of algebras ({M}urcia), Lecture Notes in Pure and
  Appl. Math., vol. 210, Dekker, New York, 2000, pp.~195--203. \MR{1758412}

\bibitem[GR03]{MR2004652}
Ofer Gabber and Lorenzo Ramero, \emph{Almost ring theory}, Lecture Notes in
  Mathematics, vol. 1800, Springer-Verlag, Berlin, 2003. \MR{2004652}

\bibitem[GT12]{MR2985554}
R\"{u}diger G\"{o}bel and Jan Trlifaj, \emph{Approximations and endomorphism
  algebras of modules: {V}olume 1 -- {A}pproximations}, extended ed., De
  Gruyter Expositions in Mathematics, vol.~41, Walter de Gruyter GmbH \& Co.
  KG, Berlin, 2012. \MR{2985554}

\bibitem[GW04]{MR2080008}
K.~R. Goodearl and R.~B. Warfield, Jr., \emph{An introduction to noncommutative
  {N}oetherian rings}, second ed., London Mathematical Society Student Texts,
  vol.~61, Cambridge University Press, Cambridge, 2004. \MR{2080008}

\bibitem[Her93]{MR1091706}
Ivo Herzog, \emph{Elementary duality of modules}, Trans. Amer. Math. Soc.
  \textbf{340} (1993), no.~1, 37--69. \MR{1091706}

\bibitem[Her97]{MR1434441}
Ivo Herzog, \emph{The {Z}iegler spectrum of a locally coherent {G}rothendieck
  category}, Proc. London Math. Soc. (3) \textbf{74} (1997), no.~3, 503--558.
  \MR{1434441}

\bibitem[HJ19]{MR4013804}
Henrik Holm and Peter J{\o}rgensen, \emph{Model categories of quiver
  representations}, Adv. Math. \textbf{357} (2019), 106826, 46. \MR{4013804}

\bibitem[Hov02]{MR1938704}
Mark Hovey, \emph{Cotorsion pairs, model category structures, and
  representation theory}, Math. Z. \textbf{241} (2002), no.~3, 553--592.
  \MR{1938704}

\bibitem[Hov07]{MR2355778}
Mark Hovey, \emph{Cotorsion pairs and model categories}, Interactions between
  homotopy theory and algebra, Contemp. Math., vol. 436, Amer. Math. Soc.,
  Providence, RI, 2007, pp.~277--296. \MR{2355778}

\bibitem[IK20]{arXiv:2010.05676}
Srikanth~B. Iyengar and Henning Krause, \emph{The {N}akayama functor and its
  completion for {G}orenstein algebras}, arXiv:2010.05676v1.

\bibitem[IK21]{arXiv:2106.00469}
Osamu Iyama and Yuta Kimura, \emph{Classifying torsion pairs of {N}oetherian
  algebras}, arXiv:2106.00469v2.

\bibitem[ILL{\etalchar{+}}07]{MR2355715}
Srikanth~B. Iyengar, Graham~J. Leuschke, Anton Leykin, Claudia Miller, Ezra
  Miller, Anurag~K. Singh, and Uli Walther, \emph{Twenty-four hours of local
  cohomology}, Graduate Studies in Mathematics, vol.~87, American Mathematical
  Society, Providence, RI, 2007. \MR{2355715}

\bibitem[IR08]{MR2427009}
Osamu Iyama and Idun Reiten, \emph{Fomin-{Z}elevinsky mutation and tilting
  modules over {C}alabi-{Y}au algebras}, Amer. J. Math. \textbf{130} (2008),
  no.~4, 1087--1149. \MR{2427009}

\bibitem[IY08]{MR2385669}
Osamu Iyama and Yuji Yoshino, \emph{Mutation in triangulated categories and
  rigid {C}ohen-{M}acaulay modules}, Invent. Math. \textbf{172} (2008), no.~1,
  117--168. \MR{2385669}

\bibitem[Kim20]{arXiv:2006.01677}
Yuta Kimura, \emph{Tilting theory of noetherian algebras}, arXiv:2006.01677v1.

\bibitem[Kra97]{MR1426488}
Henning Krause, \emph{The spectrum of a locally coherent category}, J. Pure
  Appl. Algebra \textbf{114} (1997), no.~3, 259--271. \MR{1426488}

\bibitem[Kra05]{MR2157133}
Henning Krause, \emph{The stable derived category of a {N}oetherian scheme},
  Compos. Math. \textbf{141} (2005), no.~5, 1128--1162. \MR{2157133}

\bibitem[KS03]{MR2009441}
Henning Krause and {\O}yvind Solberg, \emph{Applications of cotorsion pairs},
  J. London Math. Soc. (2) \textbf{68} (2003), no.~3, 631--650. \MR{2009441}

\bibitem[KS06]{MR2182076}
Masaki Kashiwara and Pierre Schapira, \emph{Categories and sheaves},
  Grundlehren der Mathematischen Wissenschaften [Fundamental Principles of
  Mathematical Sciences], vol. 332, Springer-Verlag, Berlin, 2006. \MR{2182076}

\bibitem[Lam91]{MR1125071}
T.~Y. Lam, \emph{A first course in noncommutative rings}, Graduate Texts in
  Mathematics, vol. 131, Springer-Verlag, New York, 1991. \MR{1125071}

\bibitem[LN19]{MR3928292}
Yu~Liu and Hiroyuki Nakaoka, \emph{Hearts of twin cotorsion pairs on
  extriangulated categories}, J. Algebra \textbf{528} (2019), 96--149.
  \MR{3928292}

\bibitem[Mat58]{MR99360}
Eben Matlis, \emph{Injective modules over {N}oetherian rings}, Pacific J. Math.
  \textbf{8} (1958), 511--528. \MR{99360}

\bibitem[Mat89]{MR1011461}
Hideyuki Matsumura, \emph{Commutative ring theory}, translated from the
  {J}apanese by {M}. {R}eid, second ed., Cambridge Studies in Advanced
  Mathematics, vol.~8, Cambridge University Press, Cambridge, 1989.
  \MR{1011461}

\bibitem[MR01]{MR1811901}
J.~C. McConnell and J.~C. Robson, \emph{Noncommutative {N}oetherian rings},
  revised ed., Graduate Studies in Mathematics, vol.~30, American Mathematical
  Society, Providence, RI, 2001, With the cooperation of L. W. Small.
  \MR{1811901}

\bibitem[MS11]{MR2737778}
Daniel Murfet and Shokrollah Salarian, \emph{Totally acyclic complexes over
  noetherian schemes}, Adv. Math. \textbf{226} (2011), no.~2, 1096--1133.
  \MR{2737778}

\bibitem[Nak11]{MR2861070}
Hiroyuki Nakaoka, \emph{General heart construction on a triangulated category
  ({I}): {U}nifying {$t$}-structures and cluster tilting subcategories}, Appl.
  Categ. Structures \textbf{19} (2011), no.~6, 879--899. \MR{2861070}

\bibitem[Nee08]{MR2439608}
Amnon Neeman, \emph{The homotopy category of flat modules, and {G}rothendieck
  duality}, Invent. Math. \textbf{174} (2008), no.~2, 255--308. \MR{2439608}

\bibitem[NP19]{MR3931945}
Hiroyuki Nakaoka and Yann Palu, \emph{Extriangulated categories, {H}ovey twin
  cotorsion pairs and model structures}, Cah. Topol. G\'{e}om. Diff\'{e}r.
  Cat\'{e}g. \textbf{60} (2019), no.~2, 117--193. \MR{3931945}

\bibitem[NT20]{MR4127282}
Tsutomu Nakamura and Peder Thompson, \emph{Minimal semi-flat-cotorsion
  replacements and cosupport}, J. Algebra \textbf{562} (2020), 587--620.
  \MR{4127282}

\bibitem[Pre09]{MR2530988}
Mike Prest, \emph{Purity, spectra and localisation}, Encyclopedia of
  Mathematics and its Applications, vol. 121, Cambridge University Press,
  Cambridge, 2009. \MR{2530988}

\bibitem[PZ20]{arXiv:2007.06536}
David Pauksztello and Alexandra Zvonareva, \emph{Co-t-structures, cotilting and
  cotorsion pairs}, arXiv:2007.06536v2.

\bibitem[Rah09]{MR2471985}
Hamidreza Rahmati, \emph{Contracting endomorphisms and {G}orenstein modules},
  Arch. Math. (Basel) \textbf{92} (2009), no.~1, 26--34. \MR{2471985}

\bibitem[RG71]{MR308104}
Michel Raynaud and Laurent Gruson, \emph{Crit\`eres de platitude et de
  projectivit\'{e}. {T}echniques de ``platification'' d'un module}, Invent.
  Math. \textbf{13} (1971), 1--89. \MR{308104}

\bibitem[Sal79]{MR565595}
Luigi Salce, \emph{Cotorsion theories for abelian groups}, Symposia
  {M}athematica, {V}ol. {XXIII} ({C}onf. {A}belian {G}roups and their
  {R}elationship to the {T}heory of {M}odules, {INDAM}, {R}ome, 1977), Academic
  Press, London-New York, 1979, pp.~11--32. \MR{565595}

\bibitem[SS18]{MR3838396}
Peter Schenzel and Anne-Marie Simon, \emph{Completion, \v{C}ech and local
  homology and cohomology: {I}nteractions between them}, Springer Monographs in
  Mathematics, Springer, Cham, 2018. \MR{3838396}

\bibitem[{\v{S}}{\v{S}}20]{MR4076700}
Jan {\v{S}}aroch and Jan {\v{S}}{\v{t}}ov\'{\i}\v{c}ek, \emph{Singular
  compactness and definability for {$\Sigma$}-cotorsion and {G}orenstein
  modules}, Selecta Math. (N.S.) \textbf{26} (2020), no.~2, Paper No. 23, 40.
  \MR{4076700}

\bibitem[Ste75]{MR0389953}
Bo~Stenstr\"{o}m, \emph{Rings of quotients: {A}n introduction to methods of
  ring theory}, Springer-Verlag, New York-Heidelberg, 1975, Die Grundlehren der
  Mathematischen Wissenschaften, Band 217. \MR{0389953}

\bibitem[{\v{S}}{\v{t}}o14]{arXiv:1412.1615}
Jan {\v{S}}{\v{t}}ov{\'\i}{\v{c}}ek, \emph{On purity and applications to
  coderived and singularity categories}, arXiv:1412.1615v1.

\bibitem[Str90]{MR1074178}
Jan~R. Strooker, \emph{Homological questions in local algebra}, London
  Mathematical Society Lecture Note Series, vol. 145, Cambridge University
  Press, Cambridge, 1990. \MR{1074178}

\bibitem[Tho19]{MR3904746}
Peder Thompson, \emph{Minimal complexes of cotorsion flat modules}, Math.
  Scand. \textbf{124} (2019), no.~1, 15--33. \MR{3904746}

\bibitem[Wat60]{MR118757}
Charles~E. Watts, \emph{Intrinsic characterizations of some additive functors},
  Proc. Amer. Math. Soc. \textbf{11} (1960), 5--8. \MR{118757}

\bibitem[Xu96]{MR1438789}
Jinzhong Xu, \emph{Flat covers of modules}, Lecture Notes in Mathematics, vol.
  1634, Springer-Verlag, Berlin, 1996. \MR{1438789}

\bibitem[Yek18]{MR3826724}
Amnon Yekutieli, \emph{Flatness and completion revisited}, Algebr. Represent.
  Theory \textbf{21} (2018), no.~4, 717--736. \MR{3826724}

\end{thebibliography}
\bibliographystyle{customamsalpha}

\end{document}